\documentclass[11pt,a4paper]{amsart}

\makeatletter
\def\subsection{\@startsection{subsection}{2}%
  \z@{.7\linespacing\@plus.7\linespacing}{.2\linespacing}%
  {\centering\normalfont\scshape}}
\makeatother

\usepackage[letterpaper,top=3.7cm, bottom=3.7cm, left=3.7cm, right=3.7cm]{geometry}

\usepackage{microtype}
\usepackage{amssymb}
\usepackage{mathtools}
\usepackage{bbm}
\usepackage{ytableau}
\usepackage{enumitem}
\usepackage{hyperref}
\hypersetup{
    colorlinks,
    linkcolor={red!60!black},
    citecolor={green!60!black},
}
\usepackage{mathabx}
\usepackage{tikz}

\newtheorem{theorem}{Theorem}[section]
\newtheorem{proposition}[theorem]{Proposition}
\newtheorem{corollary}[theorem]{Corollary}
\newtheorem{lemma}[theorem]{Lemma}
\newtheorem{conjecture}[theorem]{Conjecture}
\theoremstyle{definition}
\newtheorem{remark}[theorem]{Remark}
\newtheorem{definition}[theorem]{Definition}
\newtheorem{example}[theorem]{Example}
\newtheorem{remark*}[theorem]{Remark}

\newcommand{\N}{{\mathbb{N}}}
\newcommand{\CC}{{\mathbb{C}}}
\newcommand{\NN}{{\mathbb{N}}}
\newcommand{\R}{{\mathbb{R}}}

\newcommand{\ubeta}{{\underline{\beta}}}
\newcommand{\ulambda}{{\underline{\lambda}}}
\newcommand{\umu}{{\underline{\mu}}}

\newcommand{\usigma}{{\underline{\sigma}}}
\newcommand{\utau}{{\underline{\tau}}}
\newcommand{\abs}[1]{\lvert#1\rvert}
\newcommand{\norm}[1]{\lVert#1\rVert}
\newcommand{\ang}[1]{\langle#1\rangle}
\newcommand{\bigang}[1]{\big\langle#1\big\rangle}

\DeclareMathOperator{\id}{id}
\DeclareMathOperator{\colsp}{colsp}
\DeclareMathOperator{\Par}{Par}
\DeclareMathOperator{\Stab}{Stab}
\DeclareMathOperator{\GL}{GL}
\DeclareMathOperator{\AGL}{AGL}

\newcommand{\ytbox}[3]
{
\draw [thick] (#1,#2) rectangle +(1,1)
 +(0.5,0.1) node [above] {#3};
}

\begin{document}

\title{Transitivity in wreath products with symmetric groups}

\author{Lukas Klawuhn \and Kai-Uwe Schmidt}
\thanks{Funded by the Deutsche Forschungsgemeinschaft (DFG, German Research Foundation) -- Project-ID 491392403 -- TRR 358}
\address{Department of Mathematics, Paderborn University, Warburger Str.\ 100, 33098 Paderborn, Germany.}
\email{klawuhn@math.upb.de}

\date{April 21, 2026}


\begin{abstract}
It is known that the notion of a transitive subgroup of a permutation group $P$ extends naturally to the subsets of $P$. We study transitive subsets of the wreath product $G \wr S_n$, where $G$ is a finite abelian group. This includes the hyperoctahedral group for $G=C_2$. We give structural characterisations of transitive subsets using the character theory of $G \wr S_n$ and interpret such subsets as designs in the conjugacy class association scheme of $G \wr S_n$. In particular, we prove a generalisation of the Livingstone--Wagner theorem and give explicit constructions of transitive sets. Moreover, we establish connections to orthogonal polynomials, namely the Charlier polynomials, and use them to study codes and designs in $C_r \wr S_n$. Many of our results extend results about the symmetric group $S_n$.
\end{abstract}

\maketitle

\thispagestyle{empty}


\section{Introduction}\label{sec:intro}

The study of groups with a transitive action on an interesting set of objects is a classical and well-investigated topic (see for example \cite{Cam1999}). Two kinds of transitivity that have attracted substantial research are \emph{$t$-transitivity} and \emph{$t$-homogeneity}. A subgroup of the symmetric group $S_n$ is called \emph{$t$-transitive} if it acts transitively on the tuples of $t$ pairwise distinct elements of $\{1,\ldots,n\}$. A subgroup of the symmetric group $S_n$ is called \emph{$t$-homogeneous} if it acts transitively on the subsets of $\{1,\ldots,n\}$ of size $t$. It is immediate that a $t$-transitive subgroup is also $t$-homogeneous and $(t-1)$-transitive. Livingstone and Wagner proved the following famous result.
\begin{theorem}[{\cite[Theorem~2]{LivWag1965}}]\label{thm:livingstone wagner}
Let $G$ be a subgroup of $S_n$ that is $t$-homogeneous for some $t$ with $1 \leq t \leq n/2$. Then $G$ is also $(t-1)$-homogeneous.
\end{theorem}
This theorem was generalised by Martin and Sagan \cite{MarSag2007} in multiple ways. Their results include subsets instead of only subgroups and apply to a more general form of transitivity called $\lambda$-transitivity where $\lambda$ is a partition of $n$. Let us briefly explain these notions.

Let $G$ be a group acting on a set $\Omega$. A subset $Y \subseteq G$ is called \emph{transitive} on $\Omega$ if there exists a constant $c>0$ such that for any $a,b \in \Omega$, there are exactly $c$ elements $g \in Y$ with $ga=b$. If $c=1$, then $Y$ is called \emph{sharply transitive}. For every $a \in \Omega$, the integer~$c$ is precisely the number of elements in $Y$ that stabilise $a$, so $ c = |Y \cap \operatorname{Stab}(a)|$ is independent of $a$. If $Y$ is a subgroup, then this definition coincides with the definition of a transitive group action.

One reason to study transitive sets that are not subgroups is their size. It is well-known that for $t\geq 6$, the only $t$-transitive groups are $S_n$ (for $n \geq 6$) and $A_n$ (for $n \geq 8$). However, it is possible to construct $t$-transitive subsets of $S_n$ with $t \geq 6$ that are considerably smaller than $S_n$ or $A_n$. See \cite[Section 6]{MarSag2007} for details.

The results by Martin and Sagan also deal with transitivity on ordered set partitions of $\Omega$. An (integer) \emph{partition} $\lambda$ of a non-negative integer $n$ is a sequence $\lambda = (\lambda_1,\lambda_2,\ldots)$ of non-negative integers with $\lambda_i \geq \lambda_{i+1}$ for every $i \in \N$ and $\sum_{i\geq 1} \lambda_i = n$. For two partitions $\lambda = (\lambda_1,\lambda_2,\ldots)$ and $\mu = (\mu_1,\mu_2,\ldots)$ of $n$, we write $\lambda \unlhd \mu$ if
\[
	\sum_{i=1}^k \lambda_i \leq \sum_{i=1}^k \mu_i \quad \text{ for all } k \geq 1.
\]
An \emph{ordered set partition} $P=(P_1,P_2,\ldots)$ of $\{1,2,\ldots,n\}$ is a tuple of pairwise disjoint subsets $P_i \subseteq \{1,\ldots,n\}$ whose union is $\{1,\ldots,n\}$. For an integer partition $\lambda$, such an ordered set partition is called a \emph{$\lambda$-partition} if $|P_i| = \lambda_i$ for all $i \geq 1$. A subset $Y$ of $S_n$ is called \emph{$\lambda$-transitive} if it is transitive on the set of all $\lambda$-partitions. One of the main theorems in \cite{MarSag2007} is the following.
\begin{theorem}[{\cite[Thm. 3]{MarSag2007}}]\label{thm:lambda transitivity Livingstone Wagner}
Let $\lambda$ and $\mu$ be partitions of $n$ and let $Y$ be a $\lambda$-transitive subset of $S_n$. If $\lambda \unlhd \mu$, then $Y$ is also $\mu$-transitive.
\end{theorem}
Since $t$-homogeneity is the same as $\lambda$-transitivity for the partition $\lambda = (n-t,t)$, Theorem~\ref{thm:lambda transitivity Livingstone Wagner} can be applied to $t$-homogeneity. Taking $t\in \N$ with $1 \leq t \leq n/2$, $\lambda = (n-t,t)$ and $\mu = (n-t+1,t-1)$, we find that $\lambda \unlhd  \mu$, so Theorem~\ref{thm:livingstone wagner} is an immediate corollary of Theorem~\ref{thm:lambda transitivity Livingstone Wagner}.

In this paper, we extend these results to wreath products with the symmetric group~$S_n$. These products can be interpreted as groups of permutations of the numbers $1$ to $n$ where every number is coloured with one of $r$ different colours for some $r \in \N$. More precisely, the permutations act on the set $\{(c,i) : c \in \{1,\ldots,r\},i \in \{1,\ldots,n\}\}$. The action of the wreath product is imprimitive and for each $i\in[n]$, the sets $\{ (c,i) : c \in \{1,\ldots,r\}\}$ form blocks of imprimitivity. Thus, it is not possible for such a group to be 2-transitive, but there are many other interesting types of transitivity that we define and investigate in this paper. We show that these types of transitivity are meaningful in that they satisfy a version of the Livingstone--Wagner theorem and that they can be characterised algebraically. The class of wreath products with symmetric groups also includes symmetry groups of regular polytopes, which gives rise to a geometric interpretation of transitivity that shall motivate our studies and will serve as a basis for further generalisation.

\subsection{Transitive actions on regular polytopes}

Consider an $n$-dimensional regular polytope, for example a cube in the case $n=3$ (see \cite{Cox1973} for an introduction to regular polytopes). A \emph{flag} is a chain $F_0 \subseteq F_1 \subseteq \ldots \subseteq F_n$ of parts of the polytope, where each $F_i$ is an $i$-dimensional face of the polytope (so $\dim F_i - \dim F_{i-1} = 1$ for every $i$). By definition, the symmetry group of a regular polytope acts transitively on its flags. It is easy to see that for a regular polytope, the symmetry group also acts transitively on the parts of the polytope, for example the set of edges. We are interested in algebraic characterisations of sets that are transitive on substructures of the polytope and how they can be compared to each other.

In order to illustrate this, let us consider the $n$-dimensional cube. Let $k$ be an integer with $0 \leq k \leq n$ and $\sigma = (\sigma_0,\sigma_1,\ldots)$ be a sequence of positive integers that sum to $n-k$ (so $\sigma$ is a \emph{composition} of $n-k$, not necessarily a partition of $n-k$). We call a chain of faces $F_0 \subseteq F_1 \subseteq \ldots$ of the cube a \emph{$(\sigma,k)$-flag} if $\dim F_0 = k$ and $\dim F_i - \dim F_{i-1} = \sigma_i$ for every $i$. A subset $Y$ of the symmetry group of the cube is called \emph{$(\sigma,k)$-transitive} if it is transitive on the set of all $(\sigma,k)$-flags. We shall later see that the order of the parts of $\sigma$ does not matter. That is, a subset~$Y$ is $(\sigma,k)$-transitive if and only if it is $(\widetilde{\sigma},k)$-transitive where $\widetilde{\sigma}$ is a permutation of $\sigma$. Different permutations $\widetilde{\sigma}$ of $\sigma$ result in different geometric interpretations of $(\sigma,k)$-transitivity, but all of them are equivalent. Thus, for simplicity, we can assume that $\sigma$ is a partition of $n-k$.

Examples of some $(\sigma,k)$-flags are given in Figure~\ref{fig:transitivity examples} for the case $n=3$. We omit the brackets and commas in $\sigma$ for better readability. For example, $21$ denotes the partition $(2,1)$, so a $(21,0)$-flag is a $((2,1),0)$-flag.
\begin{figure}[h]
\begin{center}
\begin{tikzpicture}[scale=2]
\coordinate (v1) at (0,0,0) {};
\coordinate (v2) at (1,0,0) {};
\coordinate (v3) at (1,1,0) {};
\coordinate (v4) at (0,1,0) {};
\coordinate (v5) at (0,0,-1) {};
\coordinate (v6) at (1,0,-1) {};
\coordinate (v7) at (1,1,-1) {};
\coordinate (v8) at (0,1,-1) {};

\draw [fill=blue!90!white,opacity=0.65] (v1) -- (v2) -- (v3) -- (v4) -- cycle;
\draw [fill=blue!90!white,opacity=0.5] (v4) -- (v3) -- (v7) -- (v8) -- cycle;
\draw [fill=blue!90!white,opacity=0.8] (v2) -- (v3) -- (v7) -- (v6) --cycle;
\draw [fill=blue!90!white,opacity=0.4] (v1) -- (v2) -- (v6) -- (v5) --cycle;
\draw [fill=blue!90!white,opacity=0.4] (v1) -- (v5) -- (v8) -- (v4) --cycle;
\draw [fill=blue!90!white,opacity=0.4] (v5) -- (v6) -- (v7) -- (v8) --cycle;

\draw (v1) [fill=black] circle(2pt);
\draw (v2) [fill=black] circle(2pt);
\draw (v3) [fill=black] circle(2pt);
\draw (v4) [fill=black] circle(2pt);
\draw (v5) [fill=black] circle(2pt);
\draw (v6) [fill=black] circle(2pt);
\draw (v7) [fill=black] circle(2pt);
\draw (v8) [fill=black] circle(2pt);

\draw [very thick] (v1) -- (v2) -- (v3) -- (v4) -- (v8) -- (v7) -- (v6) -- (v2);
\draw [very thick] (v1) -- (v4);
\draw [very thick] (v3) -- (v7);
\draw [very thick, dashed] (v1) -- (v5);
\draw [very thick, dashed] (v5) -- (v6);
\draw [very thick, dashed] (v5) -- (v8);

\draw [fill=yellow,opacity=0.8] (v1) -- (v2) -- (v3) -- (v4) -- cycle;

\draw (v2) [fill=red] circle(2.5pt);

\node at (0.5,-0.5,0) {$(21,0)$};
\end{tikzpicture}
\hspace*{1.5cm}
\begin{tikzpicture}[scale=2]
\coordinate (v1) at (0,0,0) {};
\coordinate (v2) at (1,0,0) {};
\coordinate (v3) at (1,1,0) {};
\coordinate (v4) at (0,1,0) {};
\coordinate (v5) at (0,0,-1) {};
\coordinate (v6) at (1,0,-1) {};
\coordinate (v7) at (1,1,-1) {};
\coordinate (v8) at (0,1,-1) {};

\draw [fill=blue!90!white,opacity=0.65] (v1) -- (v2) -- (v3) -- (v4) -- cycle;
\draw [fill=blue!90!white,opacity=0.5] (v4) -- (v3) -- (v7) -- (v8) -- cycle;
\draw [fill=blue!90!white,opacity=0.8] (v2) -- (v3) -- (v7) -- (v6) --cycle;
\draw [fill=blue!90!white,opacity=0.4] (v1) -- (v2) -- (v6) -- (v5) --cycle;
\draw [fill=blue!90!white,opacity=0.4] (v1) -- (v5) -- (v8) -- (v4) --cycle;
\draw [fill=blue!90!white,opacity=0.4] (v5) -- (v6) -- (v7) -- (v8) --cycle;

\draw (v1) [fill=black] circle(2pt);
\draw (v2) [fill=black] circle(2pt);
\draw (v3) [fill=black] circle(2pt);
\draw (v4) [fill=black] circle(2pt);
\draw (v5) [fill=black] circle(2pt);
\draw (v6) [fill=black] circle(2pt);
\draw (v7) [fill=black] circle(2pt);
\draw (v8) [fill=black] circle(2pt);

\draw [very thick] (v1) -- (v2) -- (v3) -- (v4) -- (v8) -- (v7) -- (v6) -- (v2);
\draw [very thick] (v1) -- (v4);
\draw [very thick] (v3) -- (v7);
\draw [very thick, dashed] (v1) -- (v5);
\draw [very thick, dashed] (v5) -- (v6);
\draw [very thick, dashed] (v5) -- (v8);

\draw [fill=yellow,opacity=0.8] (v1) -- (v2) -- (v3) -- (v4) -- cycle;

\draw [fill=orange] (v2)+(-0.03,0) rectangle ++(0.03,1);

\node at (0.5,-0.5,0) {$(11,1)$};
\end{tikzpicture}
\hspace*{1.5cm}
\begin{tikzpicture}[scale=2]
\coordinate (v1) at (0,0,0) {};
\coordinate (v2) at (1,0,0) {};
\coordinate (v3) at (1,1,0) {};
\coordinate (v4) at (0,1,0) {};
\coordinate (v5) at (0,0,-1) {};
\coordinate (v6) at (1,0,-1) {};
\coordinate (v7) at (1,1,-1) {};
\coordinate (v8) at (0,1,-1) {};

\draw [fill=blue!90!white,opacity=0.65] (v1) -- (v2) -- (v3) -- (v4) -- cycle;
\draw [fill=blue!90!white,opacity=0.5] (v4) -- (v3) -- (v7) -- (v8) -- cycle;
\draw [fill=blue!90!white,opacity=0.8] (v2) -- (v3) -- (v7) -- (v6) --cycle;
\draw [fill=blue!90!white,opacity=0.4] (v1) -- (v2) -- (v6) -- (v5) --cycle;
\draw [fill=blue!90!white,opacity=0.4] (v1) -- (v5) -- (v8) -- (v4) --cycle;
\draw [fill=blue!90!white,opacity=0.4] (v5) -- (v6) -- (v7) -- (v8) --cycle;

\draw (v1) [fill=black] circle(2pt);
\draw (v2) [fill=black] circle(2pt);
\draw (v3) [fill=black] circle(2pt);
\draw (v4) [fill=black] circle(2pt);
\draw (v5) [fill=black] circle(2pt);
\draw (v6) [fill=black] circle(2pt);
\draw (v7) [fill=black] circle(2pt);
\draw (v8) [fill=black] circle(2pt);

\draw [very thick] (v1) -- (v2) -- (v3) -- (v4) -- (v8) -- (v7) -- (v6) -- (v2);
\draw [very thick] (v1) -- (v4);
\draw [very thick] (v3) -- (v7);
\draw [very thick, dashed] (v1) -- (v5);
\draw [very thick, dashed] (v5) -- (v6);
\draw [very thick, dashed] (v5) -- (v8);

\draw [fill=orange] (v2)+(-0.03,0) rectangle ++(0.03,1);

\draw (v2) [fill=red] circle(2.5pt);

\node at (0.5,-0.5,0) {$(12,0)$};
\end{tikzpicture}
\caption[Some $(\sigma,k)$-flags of the cube]{Some $(\sigma,k)$-flags of the cube}
\label{fig:transitivity examples}
\end{center}
\end{figure}

Consider the group $S_2 \wr S_3$ of order 48. It is the symmetry group of the 3-dimensional cube. Since the cube is a regular polytope, its symmetry group acts transitively on its flags (which are $(111,0)$-flags of the cube). We are interested in whether a subset~$Y$ of $S_2 \wr S_3$ that is $(\sigma,k)$-transitive for some partition $\sigma$ and integer $k$ is also $(\tau,l)$-transitive for some partition $\tau$ and integer $\ell$. For example, for $(2,1)$-transitivity and $(1,2)$-transitivity, this is indeed the case. It follows from our results that if a subset $Y$ of $S_2 \wr S_3$ is transitive on $(2,1)$-flags, then it is also transitive on $(1,2)$-flags. Geometrically, a $(2,1)$-flag is an edge of the cube and a $(1,2)$-flag is a face of the cube. Hence, this means that a subset $Y$ of $S_2 \wr S_3$ that is transitive on the set of the 12 edges of the cube is also transitive on the set of its 6 faces.

We can generalise this observation to the $n$-dimensional cube. The following is a special case of Theorem~\ref{thm:designs_parabolic}, one of our main results, for the $n$-dimensional cube. In the theorem, $\sigma \cup (k)$ denotes the partition obtained by adding a part of size $k$ to $\sigma$. In the case $k=0$, we have $\sigma \cup (k) = \sigma$.
\begin{theorem}\label{thm:main_result_introduction}
Let $n \in \N$ and let $Y$ be a $(\sigma,k)$-transitive subset of $S_2 \wr S_n$, the symmetry group of the $n$-dimensional cube. Then $Y$ is also $(\tau,\ell)$-transitive if $k\leq \ell$ and $(\sigma\cup (k))\unlhd (\tau\cup(\ell))$.
\end{theorem}
Taking $n=3$, $k=1$, $\ell=2$, $\sigma = (2)$ and $\tau = (1)$ in Theorem~\ref{thm:main_result_introduction} gives the statement about edge- and face-transitivity for the 3-dimensional cube from before.

We obtain our results by studying association schemes of the symmetry groups of regular polytopes. This leads us to studying Coxeter groups and wreath products and their representation theory. In what follows, it will be helpful to keep the basic example of the cube in mind when dealing with more general results later.

We organise this paper in the following way: In Section~\ref{sec:transitivity and designs}, we describe the wreath products we are interested in. We also recall some basic facts about association schemes. Moreover, we describe a way to investigate transitive subsets of a group $G$ by studying an association scheme associated with it. This is done using representation theory. Section~\ref{sec:generalised symmetric group} then deals with the representation theory of the groups we want to investigate. In Section~\ref{sec:notions of transitivity}, we combine the results from Sections~\ref{sec:transitivity and designs} and \ref{sec:generalised symmetric group} to show that the transitive sets we are interested in can be characterised as special subsets in the corresponding association scheme. We then use this characterisation in Section~\ref{sec:livingstone wagner} to describe relations between sets that are transitive on different sets of objects and obtain a generalisation of Theorem~\ref{thm:lambda transitivity Livingstone Wagner}. In Section~\ref{sec:polys_and_designs}, we study Charlier polynomials and their connections to transitive sets in wreath products with symmetric groups, and in Section~\ref{sec:constructions}, we describe an explicit construction of transitive sets. Finally, we describe an application of our results to finite projective planes in Section~\ref{sec:further research} and end with some open problems.

\section{Transitivity and association schemes}\label{sec:transitivity and designs}

In this section, we fix some notations and introduce the main definitions of transitive sets and association schemes.
\subsection{Transitive subsets}\label{sec:first_def_of_multiple_transitivity}
In this section, $(G,+)$ is a finite abelian group. The wreath product $W = G\wr S_n$ is a group, namely the semidirect product $W = G^n \rtimes S_n$ where $\sigma \in S_n$ acts on $G^n$ from the left by mapping $(g_1,\ldots,g_n) \in G^n$ to $\sigma g = (g_{\sigma^{-1}(1)},\ldots,g_{\sigma^{-1}(n)})$. We can represent every element of $W$ as a pair $(g,\pi)$ with $g \in G^n$ and $\pi \in S_n$. The product of two elements $(f,\sigma),(g,\pi) \in W$ is given by
\[
	(f,\sigma)(g,\pi) = (f+\sigma g,\sigma\pi).
\]
For more information about wreath products see \cite[Chapter 2.6]{DixMor1996}, for example.

If $G=C_1$ or $G=C_2$, then $W$ is the Coxeter group of type $A_{n-1}$ or $B_n$, respectively. If $G=C_r$ is the cyclic group of order $r$, then $W$ is the complex reflection group $G(r,1,n)$. The group $C_r \wr S_n$ is also called the \emph{generalised symmetric group}. For more information about Coxeter groups and complex reflection groups, we refer the reader to \cite{Hum1990}.

Given an integer $n \in \N$, we write $[n]$ for the set $\{1,\ldots,n\}$. There is an action of $W$ on the set $M = G \times [n]$. For $((g_1,\dots,g_n),\pi)\in W$ and $(c,i) \in G \times [n]$, it is given by
\[
	((g_1,\dots,g_n),\pi) \cdot (c,i) = (c+g_{\pi(i)},\pi(i)).
\]
It is easy to see that the wreath product $W$ consists of all permutations $\sigma$ of the set $M$ such that
\[
\sigma(0,i)=(c,j)\;\Longrightarrow\; \sigma(g,i)=(g+c,j)
\]
for all $g\in G$ and all $i\in [n]$ together with composition of mappings. We can think of the set $M$ as the set of $|G|$ copies of $[n]$ where each $i \in [n]$ appears in $|G|$ different colours.

In this paper, we study the action of $W$ on certain sets $\Omega$. We first collect some basic results about transitive group actions. They are easily found in the literature, see for example \cite{Cam1999} for details.
\begin{lemma}
Let $W$ be a group acting transitively on a set $\Omega$.
\begin{enumerate}[label=\textup{(}\alph*\textup{)}]\label{lem:transitive group action}
\item{
For every $x,y \in \Omega$ we have
\[
	\operatorname{Stab}(x) = w\operatorname{Stab}(y)w^{-1}
\]
for some $w \in W$.
}
\item{
The action of $W$ on $\Omega$ is equivalent to the action of $W$ by left multiplication on the cosets $W / H = \{wH : w \in W\}$ where $H = \operatorname{Stab}(x)$ for some $x \in \Omega$.
}
\item{
Two coset actions on $W/H$ and $W/K$ by left multiplication are isomorphic if and only if $H$ and $K$ are conjugate.
}
\end{enumerate}
\end{lemma}
This lemma allows us to investigate transitivity solely within the group $W$ itself. Every subgroup $H$ of $W$ gives rise to a transitive action of $W$ on the cosets $W/H$ and every transitive action arises this way. We can think of each subgroup $H$ as a certain type of transitivity.

We aim to investigate as many different transitivity types in the group $W$ as possible (that is investigate as many subgroups $H$ in Lemma \ref{lem:transitive group action} as possible). In order to illustrate this point of view, let us look at $t$-transitivity. The action of $W$ on $M$ is imprimitive, so no subset of $W$ under this action is ever $t$-transitive for $t \geq 2$. However, we can consider a different definition of $t$-transitivity, namely transitivity on $t$-tuples where each element is from a different block. Let
\[
	\Omega = \left\{ \big((g_1,a_1),\ldots,(g_t,a_t)\big) \in M^t : g_i \in G,\,a_i \in [n] \text{ with } a_i \neq a_j \text{ for } i \neq j \right\}
\]
be the set of all $t$-tuples of elements of $M$ where the elements of $[n]$ are pairwise distinct. Then the stabiliser of an element $x \in \Omega$ is the subgroup
\[
\text{Stab}(x)=G \wr S_A \;\cong\; G \wr S_{n-t} = G^{n-t} \rtimes S_{n-t},
\]
where $A = [n] \, \setminus \, \{a_1,\ldots,a_t\}$. Thus, in view of Lemma~\ref{lem:transitive group action}, a subset $Y$ of $W$ is transitive on $\Omega$ if and only if it is transitive on the cosets $W/H$ where $H = G \wr S_{n-t}$.

\subsection{Characters}

Since there is a close connection between association schemes of groups and their representations, we now turn to representation and character theory. We refer the reader to \cite{Ser1977} for a more extensive introduction to the representation theory of finite groups.

Let $G$ be a finite group. A \emph{representation} of $G$ is a homomorphism $M:G \to \GL_n(\CC)$ where $\GL_n(\CC)$ denotes the group of invertible $(n \times n)$-matrices over $\CC$. Each representation $M$ has a \emph{character} $\chi : G \to \CC$ which is given by $\chi(g) = \operatorname{tr} M(g)$ where $\operatorname{tr}$ denotes the trace function. The \emph{degree} $\deg(\chi)$ of a character~$\chi$ is the dimension of the vector space of the underlying representation. It is also given by $\deg(\chi) = \chi(e) = \operatorname{tr} I_n = n$ where $e$ is the identity of $G$ and $I_n$ is the identity matrix. Every representation can be decomposed into so called \emph{irreducible} representations, and hence every character can be expressed as a sum of irreducible characters.

Functions from $G$ to $\CC$ that are constant on conjugacy classes are called \emph{class functions}. Since all characters are class functions, we can write $\chi(C)$ for the value of the character $\chi$ on the conjugacy class $C$ of $G$. The class functions form a vector subspace of the space $\CC^G$ of all functions from $G$ to $\CC$. It has an inner product given by
\[
	\langle \chi, \psi \rangle = \frac{1}{|G|}\sum_{g \in G} \chi(g)\overline{\psi(g)}
\]
and the irreducible characters form an orthonormal basis of the space of class functions with respect to $\langle \cdot,\cdot \rangle$. In particular, the number of irreducible characters is equal to the number of conjugacy classes of $G$. Furthermore, if $\varphi$ is an arbitrary character and $\chi$ is an irreducible character, then $\langle \varphi,\chi \rangle$ is the \emph{multiplicity} of $\chi$ in the decomposition of $\varphi$ into irreducible characters. If $\langle \varphi,\chi \rangle \neq 0$, then $\chi$ is called an \emph{irreducible constituent} of $\varphi$.  One of our main tasks is identifiying the irreducible constituents of certain permutation characters.

If $G$ acts on a set $\Omega$, then the function $\xi_\Omega: G \to \CC$ with $\xi_\Omega(g) = |\operatorname{Fix}(g)|$ is called the \emph{permutation character of $G$ on $\Omega$}. If $H \subseteq G$ is a subgroup, then $G$ acts transitively on the set of left cosets $G/H$ via left multiplication and the corresponding permutation character $\xi_H$ is called the \emph{permutation character of $H$}. By Lemma~\ref{lem:transitive group action}, the permutation character of a transitive action of $G$ on $\Omega$ equals the permutation character $\xi_H$ where $H$ is the stabiliser of some element $\omega \in \Omega$.

Let $H \subseteq G$ be a subgroup. If $\chi$ is a character of $G$, then restricting $\chi$ to $H$ gives a character of $H$. This character is called the \emph{restriction} of $\chi$ to $H$ and denoted by $\chi \downarrow^G_H$. There is a way to extend a character of a subgroup $H$ to a character of $G$ called \emph{induction}. If $\chi$ is a character of $H$, then the induced character is denoted by $\chi \uparrow^G_H$.

The function $1_G$ with $1_G(g) = 1$ for all $g \in G$ is a character of $G$, called the \emph{trivial character} of $G$. If we restrict $1_G$ to a subgroup $H \subseteq G$, then we get the trivial character of $H$, so $1_G \downarrow^G_H = 1_H$. It is well-known that the induction of the trivial character $1_H$ of $H$ to $G$ equals the permutation character of $H$, that is
\[
	1_H \uparrow^G_H = \xi_H.
\]

In order to prove our results about permutation characters, we need Mackey's formula. It describes the result of inducing a character followed by restriction. Let $H,K \subseteq G$ be subgroups. If $\chi$ is a character of $H$ and $s \in G$, we define a character $\chi^s$ of $H^s = sHs^{-1}$ via
\[
	\chi^s(x) = \chi(s^{-1}xs)
\]
for every $x \in H^s$. We have the following theorem (see for example \cite[Section 7.3]{Ser1977}).
\begin{theorem}[Mackey]\label{thm:mackey}
With the notation as above, we have
\[
	\chi \uparrow^G_H \downarrow^G_K = \sum_{s \in R} \chi^s \downarrow^{H^s}_{H^s \cap K} \uparrow^K_{H^s \cap K} 
\]
where $R$ is a system of representatives for the double cosets $KgH$ with $g \in G$.
\end{theorem}

\subsection{The conjugacy class association scheme}
We will mainly study the combinatorics of $G \wr S_n$ from the viewpoint of association schemes. We refer the reader to \cite{BanIto1984} for an extensive introduction to association schemes. Every finite group gives rise to an association scheme (see \cite[Section 2.7]{BanIto1984} for details), called the \emph{conjugacy class scheme} of the group, but the theory of association schemes is much more general than that. We collect the relevant details about conjugacy class schemes here. In view of Section \ref{sec:general method}, we will first describe conjugacy class schemes in general and later specialise to the group $G \wr S_n$.

For finite and non-empty sets $X$ and $Y$, let $\CC(X,Y)$ denote the set of all complex $|X| \times |Y|$-matrices with rows indexed by $X$ and columns indexed by $Y$. For a matrix $A \in \CC(X,Y)$ and $x\in X$, $y \in Y$, the $(x,y)$-entry of $A$ is denoted by $A(x,y)$. If $|Y|=1$, then we omit $Y$, so $\CC(X)$ is the set of complex column vectors indexed by $X$.

Let $G$ be a finite group and let $C_0 = \{e\},C_1,\ldots,C_n$ denote its conjugacy classes. For $i \in \{0,\ldots,n\}$, let $A_i \in \CC(G,G)$ be given by
\begin{equation}\label{eqn:def_incidence_matrices_general}
	A_i(x,y) = \begin{cases} 1,& \text{ if } x^{-1}y \in C_i\\
									 0,& \text{ otherwise}.
					  \end{cases}
\end{equation}
Let $\mathbb{A} = \operatorname{sp}(A_0,\ldots,A_n)$ be the vector space generated by $A_0,\ldots,A_n$ over the complex numbers. Then $\mathbb{A}$ is a commutative matrix algebra that contains the identity matrix and is closed under conjugate transposition. Thus, the zero-one-matrices $A_i$ define an association scheme, called the \emph{conjugacy class scheme} of $G$. The algebra $\mathbb{A}$ is called the \emph{Bose--Mesner algebra} of this association scheme.

Since $\mathbb{A}$ is commutative and closed under taking the conjugate transpose, all of its elements can be simultaneously diagonalised, and therefore there exists a basis $E_0,\ldots,E_n$ of $\mathbb{A}$ consisting of Hermitian matrices with the property
\begin{equation}\label{eqn:minimal_idempotent_property_general}
	E_k E_l = \delta_{kl} E_k.
\end{equation}
The matrices $E_k$ are the minimal idempotents of $\mathbb{A}$. These matrices are given by (see \cite[Theorem II.7.2]{BanIto1984})
\[
	E_k = \frac{\deg(\chi_k)}{|G|} \sum_{l=0}^n \chi_k (C_l) A_l.
\]
Using (\ref{eqn:def_incidence_matrices_general}), the entries of $E_k$ are given by
\begin{equation}\label{eqn:entries_idempotents_general}
	E_k (x,y) = \frac{\deg(\chi_k)}{|G|} \chi_{k} ( x^{-1}y).
\end{equation}
Let $V_k$ be the column space of $E_k$. The spaces $V_k$ are called the \emph{eigenspaces} of the association scheme because they are the common eigenspaces of the matrices $A_i$. From~(\ref{eqn:minimal_idempotent_property_general}) we obtain that these vector spaces are pairwise orthogonal and that
\begin{equation}\label{eqn:direct_sum_eigenspaces}
	\CC(G) = \bigoplus_{k = 0}^n V_k.
\end{equation}
Let $Y$ be a non-empty subset of $G$. We can associate two sequences of numbers with~$Y$: The \emph{inner distribution} $a$ and the \emph{dual distribution} $a'$. The inner distribution of $Y$ is the tuple $(a_0,\ldots,a_n)$, where
\begin{equation}\label{eqn:def_inner_distribution_general}
	a_i = \frac{1}{|Y|}\sum_{x,y \in Y} A_i(x,y),
\end{equation}
and the dual distribution of $Y$ is the tuple $(a'_0,\ldots,a'_n)$, where
\begin{equation}\label{eqn:def_dual_distribution_general}
	a'_k = \frac{|G|}{|Y|}\sum_{x,y \in Y} E_k(x,y).
\end{equation}
We can use (\ref{eqn:entries_idempotents_general}) to obtain a formula for the dual distribution $a'$ in terms of characters:
\begin{equation}\label{eqn:dual_distribution_characters_general}
	a'_k = \frac{\deg(\chi_k)}{|Y|}\sum_{x,y \in Y} \chi_k(x^{-1}y)
\end{equation}
It is immediate that the inner distribution is non-negative. The same holds for the dual distribution. Let $\mathbbm{1}_Y \in \CC(G)$ be the characteristic vector of $Y$, so $\mathbbm{1}_Y(x) = 1$ if $x \in Y$ and $\mathbbm{1}_Y(x) = 0$ otherwise. Since the matrix $E_k$ is hermitian, it follows from (\ref{eqn:dual_distribution_characters_general}) that
\begin{equation}\label{eqn:dual_distr_orthogonality_general}
	\frac{|Y|}{|G|} a'_k = \mathbbm{1}_Y^\top E_k \mathbbm{1}_Y = \mathbbm{1}_Y^* E_k^* E_k \mathbbm{1}_Y = \norm{E_k \mathbbm{1}_Y}^2.
\end{equation}
Thus, the dual distribution is real and non-negative. Furthermore, the case $a'_k = 0$ occurs if and only if $\mathbbm{1}_Y$ is orthogonal to $V_k$.

The main reason why the theory of association schemes is such a powerful tool in combinatorics, is the observation that interesting combinatorial structures can often be characterised using the inner or dual distribution. Delsarte \cite{Del1973} calls these objects \emph{cliques} and \emph{designs}, respectively.
\begin{definition}
Let $Y$ be a non-empty subset of $G$ with inner distribution $(a_0,\ldots,a_n)$ and dual distribution $(a'_0,\ldots,a'_n)$.
\begin{enumerate}[label=\textup{(}\alph*\textup{)}]
\item{
Let $D \subseteq [n]$. We call $Y$ a \emph{$D$-clique} or a \emph{$D$-code} if $a_d = 0$ for every $d \not\in D$.
}
\item{
Let $T \subseteq [n]$. We call $Y$ a \emph{$T$-design} if $a'_t = 0$ for every $t \in T$.
}
\end{enumerate}
\end{definition}
We show in Section~\ref{sec:notions of transitivity} that the structures we are interested in are in fact $T$-designs in the sense of Delsarte.

\subsection{The general method}\label{sec:general method}

We now describe the method we use to connect transitive subsets to designs in an association scheme. The method can be applied to any finite group $G$ acting on a set~$\Omega$. The result is essentially that if the decomposition of the permutation character on~$\Omega$ is known, then this decomposition gives rise to a characterisation of transitivity on~$\Omega$ as a $T$-design in an association scheme.

Let $G$ be a finite group acting on a set $\Omega$. We define the \emph{incidence matrix} of this action to be the matrix $M\in\CC(G,\Omega\times\Omega)$ given by
\[
M(g,(a,b))=\begin{cases}
1 & \text{if $ga=b$},\\
0 & \text{otherwise}.
\end{cases}
\]
Denote by $\colsp(M)$ the vector space generated by the columns of $M$. The following theorem is crucial.
\begin{theorem}
\label{thm:colsp_incidence_matrix}
Let $G$ be a finite group acting on a set $\Omega$ with incidence matrix~$M$. Let $\chi_0,\chi_1,\dots,\chi_n$ be the irreducible characters of $G$ and let $V_0,V_1,\dots,V_n$ be the corresponding eigenspaces of the conjugacy class association scheme of $G$. Define $I \subseteq \{0,1,\dots,n\}$ such that $k \in I$ if and only if $\chi_k$ is an irreducible constituent of the permutation character $\xi$ of $G$ on $\Omega$. Then 
\[
\colsp(M)=\bigoplus_{k\in I}V_k.
\]
\end{theorem}
\begin{proof}
We have $MM^\top \in \CC(G,G)$ and for $g,h \in G$ we have
\[
	(MM^\top)(g,h) = \sum_{(a,b) \in \Omega \times \Omega} M(g,(a,b))M(h,(a,b)).
\]
The summands are either 0 or 1. A summand is equal to 1 if and only if $ga=b$ and $ha=b$, so $(g^{-1}h) a = a$. Thus, the entry of $MM^\top$ at position $(g,h)$ is equal to the number of fixed points $|\operatorname{Fix}(g^{-1}h)|$ of $g^{-1}h$.

For the permutation character $\xi$, we have $\xi(g) = |\operatorname{Fix}(g)|$. Thus, considering the matrix $P \in \CC(G,G)$ corresponding to $\xi$ given by $P(g,h) = \xi(g^{-1}h) = |\operatorname{Fix}(g^{-1}h)|$, we obtain $P = MM^\top$. This implies
\[
	\colsp(M) = \colsp(MM^*) = \colsp(MM^\top) = \colsp(P).
\]
Now consider the decomposition
\[
	\xi = m_0 \chi_0 + \ldots + m_n \chi_n
\]
of $\xi$ into irreducible characters $\chi_k$ where $m_k \in \NN_0$. Recall that by \eqref{eqn:entries_idempotents_general}, the idempotent matrix $E_k$ of the conjugacy class association scheme is given by
\[
	E_k(g,h) = c_k \chi_k(g^{-1}h)
\]
for some constant $c_k>0$. Since
\[
	P(g,h) = \xi(g^{-1} h) = \sum_{k=0}^n m_k \chi_k(g^{-1} h) = \left( \sum_{k=0}^n m_k c^{-1}_k E_k \right)(g,h),
\]
we obtain (since $k \in I$ if and only if $m_k \neq 0$)
\begin{equation}\label{eqn:decomp_matrix_permchar}
	P = \sum_{k=0}^n m_kc^{-1}_kE_k = \sum_{k \in I} m_kc^{-1}_kE_k.
\end{equation}
Recall that $\colsp(E_k) = V_k$ and that the spaces $V_k$ are pairwise orthogonal. Hence, the column space of $P$ is contained in $\bigoplus_{k\in I} V_k$. Conversely, let $k \in I$ and $v$ be a column vector of $E_k$. From \eqref{eqn:minimal_idempotent_property_general}, we have $E_kv = v$ and $E_l v = 0$ for $l \neq k$, so it follows from \eqref{eqn:decomp_matrix_permchar} that
\[
	Pv = PE_kv = m_kc_k^{-1}v.
\]
Since $m_k \neq 0$ for $k \in I$, we conclude that $V_k = \colsp(E_k) \subseteq \colsp(P)$ for every $k \in I$. Thus, we have
\[
	\colsp(M) =	\colsp(P) = \bigoplus_{k \in I} V_k. \qedhere
\]
\end{proof}

Now we can describe the connection between transitivity and the dual distribution.

\begin{theorem}\label{thm:general_method}
Let $G$ be a finite group acting transitively on a set $\Omega$. Let $\chi_0,\chi_1,\dots,\chi_n$ be the irreducible characters of $G$ and define $I \subseteq \{0,1,\dots,n\}$ such that $k \in I$ if and only if $\chi_k$ is an irreducible constituent of the permutation character $\xi$ of $G$ on $\Omega$. Let $Y$ be a non-empty subset of $G$ with dual distribution $(a'_0,a'_1,\dots,a'_n)$. Then $Y$ is transitive on $\Omega$ if and only if
\[
a'_k=0\quad\text{for each $k\in I \,\setminus \{0\}$}.
\]
\end{theorem}
\begin{proof}
Note that $G$ acts transitively on $\Omega$ if and only if
\[
	M^\top \mathbbm{1}_G = c \cdot \mathbbm{1}_{\Omega \times \Omega}
\]
for some constant $c$, since this equation means that for any pair $(\omega,\omega') \in \Omega \times \Omega$ there exist precisely $c$ elements of $G$ that map $\omega$ to $\omega'$. Thus, $Y$ is transitive on $\Omega$ if and only if
\[
	M^\top \mathbbm{1}_Y = c' \cdot \mathbbm{1}_{\Omega \times \Omega}
\]
for some constant $c'$. This constant can be calculated by observing that if $Y$ is transitive, then so are the cosets $gY$ for $g \in G$. Since the disjoint union of transitive sets is a transitive set, we find that $c = |G/Y| \cdot c'$ which gives
\[
	c' = \frac{|Y|}{|G|} \cdot c. 
\]
Now we see that $Y$ is transitive on $\Omega$ if and only if
\[
\frac{1}{\abs{Y}}M^\top \mathbbm{1}_Y=\frac{1}{\abs{G}}M^\top \mathbbm{1}_G,
\]
hence if and only if 
\[
\mathbbm{1}_Y-\frac{\abs{Y}}{\abs{G}}\mathbbm{1}_G
\]
is orthogonal to the column space of $M$. By Theorem~\ref{thm:colsp_incidence_matrix}, the column space of $M$ is given by $\colsp(M)=\bigoplus_{k\in I}V_k$. Hence, $\mathbbm{1}_Y-\frac{\abs{Y}}{\abs{G}}\mathbbm{1}_G$ is orthogonal to the column space of $M$ if and only if it is orthogonal to the spaces $V_k$ with $k \in I$.

First, note that $V_0$ is spanned by $\mathbbm{1}_G$ and that
\[
	\langle \mathbbm{1}_Y-\frac{\abs{Y}}{\abs{G}}\mathbbm{1}_G, \mathbbm{1}_G \rangle = 0.
\]
Now let $k \in \{1,\ldots,n\}$ and $v \in V_k$. Since the spaces $V_k$ are pairwise orthogonal and $\mathbbm{1}_G \in V_0$, we find that
\[
	\langle \mathbbm{1}_Y-\frac{\abs{Y}}{\abs{G}}\mathbbm{1}_G, v \rangle = \langle \mathbbm{1}_Y, v \rangle.
\]
Thus, for $k \neq 0$, $\mathbbm{1}_Y-\frac{\abs{Y}}{\abs{G}}\mathbbm{1}_G$ is orthogonal to $V_k$ if and only if $\mathbbm{1}_Y$ is orthogonal to $V_k$. Recall that by \eqref{eqn:dual_distr_orthogonality_general}, $\mathbbm{1}_Y$ is orthogonal to $V_k$ if and only if $a'_k = 0$, which concludes the proof.
\end{proof}
In view of Theorem \ref{thm:general_method}, our next task is to study the representation theory of $G \wr S_n$ and decompose the permutation characters $\xi_H$ for a given subgroup $H$.


\section{Wreath products with symmetric groups}\label{sec:generalised symmetric group}

In this section, we collect the definitions and results about permutation groups and their representation theory that we need later on. We mostly follow \cite[Chapter I]{Mac1995}.

\subsection{Partitions}

An (integer) \emph{partition} is a sequence $\lambda = (\lambda_1,\lambda_2,\ldots)$ of non-negative integers that sum to a finite number and satisfy $\lambda_i \geq \lambda_{i+1}$ for every $i \in \N$. The \emph{size} of $\lambda$ is the number $|\lambda| = \sum_{i\geq 1} \lambda_i$ and if $|\lambda| = n$, we say that $\lambda$ is a partition of $n$. We often write partitions as finite sequences $(\lambda_1,\ldots,\lambda_k)$ only writing down the non-zero parts. When no confusion can arise, we omit the commas when writing down partitions. We also write $k^t$ to denote $t$ parts of size $k$, for example $2^11^2$ is the sequence $(2,1,1,0,0,\ldots)$. Furthermore, we denote by $\Par$ the set of all integer partitions (of any size) while $\Par_n$ denotes the set of partitions of $n$. The unique partition of $0$ will be denoted by~$\emptyset$.

If $\lambda$ is a partition of $n$, then the \emph{Young diagram} of $\lambda$ is an array of $n$ boxes with left-justified rows and top-justified columns such that row $i$ contains exactly $\lambda_i$ boxes. Each partition gives rise to a \emph{conjugate partition} $\lambda'$ whose parts are the number of boxes in the columns of the Young diagram of $\lambda$ (so the Young diagram of $\lambda'$ is the transpose of the Young diagram of $\lambda$).

We use the dominance order $\unlhd$ to compare partitions. Let $\lambda,\mu \in \Par$ be two partitions (not necessarily of the same size). We write $\lambda \unlhd \mu$ and say that $\lambda$ is \emph{dominated} by $\mu$ or $\mu$ \emph{dominates} $\lambda$ if
\[
	\sum_{i=1}^k \lambda_i \leq \sum_{i=1}^k \mu_i \quad \text{ for all } k \geq 1.
\]
Note that if $|\lambda|=|\mu|$, then we have
\[
	\lambda \unlhd \mu \; \Longleftrightarrow \; \lambda' \unrhd \mu'.
\]
Another helpful fact is that we have $\lambda \unlhd \mu$ if and only if the Young diagram of $\lambda$ can be transformed into the Young diagram of $\mu$ by moving boxes from the end of a row of $\lambda$ to a higher row one at a time (see for example \cite[Proposition~2.3]{Bry73}). When interpreting a partition as a sequence, this corresponds to adding vectors of the form $(0,\ldots,0,1,0,\ldots,0,-1,0,\ldots)$ to the partition.

We use two methods for combining partitions. If $\lambda$ and $\mu$ are partitions, then $\lambda \cup \mu$ is the partition of $|\lambda| + |\mu|$ that has as its parts exactly the parts of $\lambda$ and $\mu$, and $\lambda + \mu$ is the componentwise sum, that is $\lambda + \mu = (\lambda_1 + \mu_1,\lambda_2 + \mu_2,\ldots)$. It holds that
\[
	\left(\lambda \cup \mu \right)' = \lambda' + \mu ' \quad \text{ and } \quad \left(\lambda + \mu \right)' = \lambda' \cup \mu'.
\]

A partition $\lambda$ is \emph{contained} in a partition $\mu$, denoted by $\lambda \subseteq \mu$, if $\lambda_i \leq \mu_i$ for every $i \geq 1$. In terms of Young diagrams, this means that the Young diagram of $\lambda$ lies inside the Young diagram of $\mu$ if they are superimposed.

\subsection{Conjugacy classes}\label{sec:cc_wreath_product}

We now describe the conjugacy classes of $G \wr S_n = G^n \rtimes S_n$, where $G$ is finite abelian. We refer to \cite[Appendix B, Chapter I]{Mac1995} for further reading. The conjugacy classes of $G \wr S_n$ are indexed by functions $\ulambda: G \to \Par$ such that
\[
	\sum_{g \in G} |\ulambda(g)| = n.
\]
We write
\[
	\Lambda_n(G) = \{ \ulambda : G \to \Par : \sum_{g \in G} |\ulambda(g)| = n \}
\]
for the set of all these partition valued functions $\ulambda$.

Consider an element $((g_1,\dots,g_n),\pi) \in G \wr S_n$. The permutation $\pi \in S_n$ can be written as a product of disjoint cycles of the form $(i_1,\ldots,i_k)$. It turns out that the element $g_{i_1} + \ldots + g_{i_k}$ is determined up to conjugacy. Thus, since $G$ is abelian in our setting, it is determined uniquely. We call this element $g_{i_1} + \ldots + g_{i_k}$ the \emph{sign} of the cycle $(i_1,\ldots,i_k)$. For each $g \in G$, we can determine all the cycles in $\pi$ that have sign~$g$. Writing down all the cycle lengths that occur and ordering them decreasingly, we obtain a partition associated with the element $g$. The \emph{cycle type} of an element $((g_1,\dots,g_n),\pi) \in G\wr S_n$ is defined to be the mapping $\ulambda: G \to \Par$, where $\ulambda(g)$ equals the partition of cycle lengths of $\pi$ that have sign $g$. Then two elements of $G\wr S_n$ are conjugate if and only if they have the same cycle type. Since the sum of all cycle lengths of $\pi$ is $n$, the conjugacy classes of $G \wr S_n$ are indexed by the set $\Lambda_n(G)$.

We denote by $C_\ulambda$ the conjugacy class indexed by $\ulambda$. For example, the conjugacy class of the identity of $G \wr S_n$ is indexed by the element $\ulambda\in\Lambda_n(G)$ with $\ulambda(0) = (1^n)$ and $\ulambda(g) = \emptyset$ for every $g \neq 0$.

\subsection{Characters of wreath products}
In this section, we review the representation theory of wreath products of the form $G \wr S_n = G^n \rtimes S_n$, where $G$ is a finite abelian group. We first consider the more general case $A\rtimes H$, where $A$ and $H$ are finite groups with $A$ being abelian. We essentially follow~\cite[Sect. 8.2]{Ser1977}.

Since $A$ is abelian, each irreducible character of $A$ has degree $1$ and the irreducible characters $\theta$ of $A$ form a group $A'$. Note that $A$ is normal in $A \rtimes H$, so the group $H$ acts on $A'$ by
\[
(h\theta)(a)=\theta(h^{-1}ah) \quad \text{for $h\in H$, $\theta\in A'$ and $a\in A$}.
\]
Let $\theta$ be an irreducible character of $A$ and let $K$ be the stabiliser of $\theta$ under the action of $H$. Now let $\psi$ be an irreducible character of $K$. Then we can define the character $\theta\cdot\psi$ of $A\rtimes K$ by
\[
(\theta\cdot\psi)(ak)=\theta(a)\psi(k)\quad\text{for all $a\in A$, $k\in K$}.
\]
Finally, we induce $\theta\cdot\psi$ to a character $\chi_{\theta,\psi}$ of $A\rtimes H$, so that
\[
\chi_{\theta,\psi}=(\theta\cdot\psi)\uparrow_{A\rtimes K}^{A\rtimes H}.
\]
We have the following theorem.
\begin{theorem}[Serre {\cite[Proposition 25]{Ser1977}}]
\label{thm:irreps_wreath_product}
With the notation as above, $\chi_{\theta,\psi}$ is an irreducible character of $A\rtimes H$ and every irreducible character of $A\rtimes H$ is of this form. Moreover, if $\chi_{\theta,\psi}=\chi_{\theta',\psi'}$, then $\psi=\psi'$ and $\theta$ and $\theta'$ are in the same orbit under the action of $H$ on $A'$.
\end{theorem}

We now want to apply Theorem~\ref{thm:irreps_wreath_product} to the group $G \wr S_n = G^n \rtimes S_n$, so we need the characters of $G^n$ and the characters of some subgroups of the symmetric group. Since $G$ is abelian, all representations of $G$ are one-dimensional and the group of the corresponding characters of $G$ is isomorphic to $G$ itself. We fix an isomorphism $g \mapsto\theta^g$ from $G$ to its character group. Then the irreducible characters of $G$ are the characters $\theta^g$ for $g \in G$.

The characters of $G^n$ are products of characters of $G$. Thus, they are indexed by the tuples $(g_1,\ldots,g_n)\in G^n$. The irreducible character $\theta^{(g_1,\ldots,g_n)}$ is given by
\[
	\theta^{(g_1,\ldots,g_n)} (x_1,\ldots,x_n) = \prod_{i=1}^n \theta^{g_i}(x_i) \; \text{ for every } (x_1,\ldots,x_n) \in G^n.
\]

Now consider an element $(g_1,\ldots,g_n) \in G^n$. The symmetric group $S_n$ acts on the characters $\theta^{(g_1,\ldots,g_n)}$ by permuting the entries of $(g_1,\dots,g_n)$. For $g\in G$, let $k(g)$ be the number of times $g$ occurs in $(g_1,\dots,g_n)$. Then $\sum_{g\in G}k(g)=n$ and the stabiliser of the character $\theta^{(g_1,\ldots,g_n)}$ under the action of $S_n$ is
\[
\Stab(\theta^{(g_1,\ldots,g_n)}) = \prod_{g\in G} S_{k(g)}
\]
where $S_{k(g)}$ permutes the indices $i \in [n]$ with $g_i = g$. This is a \emph{Young subgroup} of $S_n$. The irreducible characters of a factor $S_{k(g)}$ are indexed by the partitions of $k(g)$ and denoted by $\psi^\mu$ where $\mu$ is a partition. Now, the irreducible characters of $\Stab(\theta^{(g_1,\ldots,g_n)})$ are given by
\[
\psi^\umu = \bigotimes_{g\in G}\psi^{\umu(g)},
\]
where $\umu: G \to \Par$ is a function such that $\umu(g)$ is a partition of $k(g)$. Combining $\psi^\umu$ with the character $\theta^{(g_1,\ldots,g_n)}$ as in Theorem~\ref{thm:irreps_wreath_product} gives an irreducible character of $G \wr S_n$. Since we only need the orbit of $\theta^{(g_1,\ldots,g_n)}$, so the numbers $k(g)$ for $g \in G$, all information needed is encoded in $\umu$ (because $k(g) = |\umu(g)|$). Thus, the irreducible characters of $G \wr S_n$ are indexed by the set of mappings $\ulambda:G\to \Par$ such that
\[
\sum_{g\in G}\abs{\ulambda(g)}=n.
\]
This is the set $\Lambda_n(G)$ that we have seen earlier. For $\ulambda\in\Lambda_n(G)$, the corresponding irreducible character is denoted by $\chi^{\ulambda}$. For example, the trivial character of $G \wr S_n$ is $\chi^{\ulambda}$, where $\ulambda\in\Lambda_n(G)$ is the unique element with $\ulambda(0) = (n)$ and $\ulambda(g) = \emptyset$ for every $g \neq 0$.

Next, we present a different way to write down the irreducible characters of $G \wr S_n$. For a character $\theta$ of $G$ and a character $\psi$ of $S_k$, we define $\theta\wr\psi$ to be the character of $G\wr S_k$ given by
\[
(\theta\wr\psi)((g_1,\dots,g_k),\pi)=\theta(g_1)\cdots\theta(g_k)\,\psi(\pi).
\]
for all $(g_1,\ldots,g_k) \in G^k$ and all $\pi \in S_k$. Now consider $\ulambda\in\Lambda_n(G)$. Then we have partitions $\ulambda(g)$ for $g \in G$. The character $\chi^{\ulambda}$ of $G\wr S_n$ is given by
\begin{equation}
\chi^{\ulambda}=\bigotimes_{g\in G}\left(\theta^g \wr \psi^{\ulambda(g)}\right)\uparrow_{G\wr S}^{G\wr S_n},   \label{eqn:def:chi_lambda}
\end{equation}
where
\begin{equation}
S=\prod_{g\in G} S_{|\ulambda(g)|}   \label{eqn:def_group_S}
\end{equation}
is a Young subgroup of $S_n$ and $G \wr S$ denotes the subgroup $G^n \rtimes S$ of $G \wr S_n = G^n \rtimes S_n$.

\begin{remark}
Note that we did not specify a tuple $(g_1,\ldots,g_n) \in G^n$ as before. The element $\ulambda \in \Lambda_n(G)$ defines a tuple $(g_1,\ldots,g_n)$ up to a permutation of the entries (the element $g \in G$ appears $|\ulambda(g)|$ times). Different choices of tuples lead to conjugate Young subgroups $S$ and ultimately to the same character $\chi^{\ulambda}$.

In most cases, it is not important how exactly the group $S$ lies in the group $S_n$. For example, $S_{\{1,2\}} \times S_{\{3,4\}}$ and $S_{\{1,4\}} \times S_{\{2,3\}}$ are different subgroups of $S_4$, but they are conjugate to each other. The first subgroup corresponds to the ordered set partition $(\{1,2\},\{3,4\})$ while the latter corresponds to the ordered set partition $(\{1,4\},\{2,3\})$. Every time we write down a Young subgroup $S$, it implicitly comes with an ordered set partition $P$. In our case, the parts of $P$ are indexed by the group $G$. The factor $S_{|\ulambda(g)|}$ permutes the set of indices $i \in [n]$ of the tuple $(g_1,\ldots,g_n)$ such that $g_i = g$. This set of indices is also the part of $P$ that is indexed by $g \in G$.
\end{remark}

If $\ulambda(g) = \emptyset$ for some $g \in G$, the corresponding factor in the product in (\ref{eqn:def:chi_lambda}) or (\ref{eqn:def_group_S}) is trivial. The reader should interpret equation (\ref{eqn:def:chi_lambda}) as first choosing for each $g \in G$ how often each character $\theta^g$ appears, then choosing a partition $\ulambda(g)$ of the corresponding size to get a character $\psi^{\ulambda(g)}$, and then combining all these characters. In view of Theorem \ref{thm:irreps_wreath_product}, the $\chi^\ulambda$ are all the irreducible characters of $G \wr S_n$.


\section{Notions of Transitivity}\label{sec:notions of transitivity}

We continue to consider the group $W=G\wr S_n$ for a fixed abelian group $G$. We will now focus on the different types of transitivity in $W$, which means the different $\Omega$ in Theorem~\ref{thm:general_method}.
\subsection{Transitivity types}
Let $S(G)$ be the set of subgroups of $G$. For a mapping $\usigma:S(G)\to\Par$, we define
\[
\norm{\usigma}=\sum_{U\le G} \abs{\usigma(U)}.
\]
We write
\[
	\Sigma_n(G) = \{ \usigma : S(G) \to \Par : \norm{\usigma}=n \}
\]
for the set of all mappings $\usigma:S(G)\to\Par$ satisfying $\norm{\usigma}=n$. In what follows, an element $\usigma \in \Sigma_n(G)$ will be associated with a subgroup of $W$. Hence, every $\usigma$ will represent a certain notion of transitivity (for example edge- and face-transitivity for the 3-dimensional cube described in Section \ref{sec:intro}).

Each $\usigma\in\Sigma_n(G)$ can be represented by an ordered set of Young diagrams $Y_U$, one for each subgroup $U \subseteq G$, such that the total number of boxes equals $n$. We define a \emph{$\usigma$-tableau} to be a filling of the boxes with pairs $(c,i)$, where $c \in G$ and $i\in [n]$ such that all numbers $i \in [n]$ appearing in the pairs $(c,i)$ are pairwise distinct.

Note that $W$ acts on a filled Young diagram by letting $w \in W$ act on the fillings $(c,i)$ of the boxes (see Section \ref{sec:first_def_of_multiple_transitivity}). Consider a Young diagram $Y_U$ indexed by a subgroup $U$ and one of its rows $R$ of length $\ell$. Then the group $U \wr S_\ell$ acts on the row $R$ by permuting the fillings $(c,i)$ and changing the elements $c$ using elements of $U$. The partition $\usigma(U)$ contains the lengths of the rows of $Y_U$, so the group
\[
	U \wr S_{\usigma(U)} \cong U \wr S_{\usigma(U)_1} \times U \wr S_{\usigma(U)_2} \times \dots
\]
acts on a filled Young diagram $Y_U$ in the same way. We call two filled Young diagrams indexed by $U$ \emph{row-equivalent} if they are in the same orbit under this action. That is, the corresponding rows of the Young diagrams contain the same pairs $(c,i)$ up to permutation and multiplication by elements of $U$. Hence, we can also think of the colours $c$ appearing in the filling of the Young diagrams associated with $U$ as being cosets of $U$.

Two $\usigma$-tableaux are \emph{row-equivalent} if all of their Young diagrams $Y_U$ are row-equivalent for every subgroup $U$ of $G$. This induces an equivalence relation on $\usigma$-tableaux. A \emph{$\usigma$-tabloid} is an equivalence class of this relation.
\begin{example}
Consider $W = C_4 \wr S_6$ where we identify $C_4$ with $\{0,1,2,3\}$ in the natural way. Then the subgroups of $C_4$ are $\{0\}$, $\{0,2\}$ and $\{0,1,2,3\}$. Hence, every element $\usigma \in \Sigma_6(C_4)$ consists of three Young diagrams, one for each subgroup of $C_4$, such that the total number of boxes is 6. Consider $\usigma\in\Sigma_{6}(C_4)$ given by
\begin{center}
\begin{tabular}{ccc}
$\{0\}$ & $\{0,2\}$ & $\{0,1,2,3\}$\\[1ex]
\begin{ytableau}
\, & \, \\
\,
\end{ytableau}
&
\begin{ytableau}
\, & \, & \,
\end{ytableau}
&
\raisebox{1.3ex}{$\emptyset$}
\end{tabular}.
\end{center}
A possible $\usigma$-tableau is given by the following filling of the boxes (where we write $c,i$ for the filling $(c,i)$ for better readability)
\begin{center}
\begin{tabular}{ccc}
$\{0\}$ & $\{0,2\}$ & $\{0,1,2,3\}$\\[1ex]
\begin{tikzpicture}[scale=0.8]
\ytbox{0}{1}{$1,6$}
\ytbox{1}{1}{$0,4$}
\ytbox{0}{0}{$3,2$}
\end{tikzpicture}
&
\begin{tikzpicture}[scale=0.8]
\ytbox{0}{1}{$2,3$}
\ytbox{1}{1}{$0,1$}
\ytbox{2}{1}{$1,5$}
\node at (0,0.15) {};
\end{tikzpicture}
&
\begin{tikzpicture}
\node at (0,0.15) {};
\node at (0,1.25) {$\emptyset$};
\end{tikzpicture}
\end{tabular}
\end{center}
where we think of a pair $(c,i)$ as the number $i \in \{1,2,3,4,5,6\}$ coloured with the colour $c \in \{0,1,2,3\}$. The corresponding $\usigma$-tabloid is given by
\begin{center}
\begin{tabular}{ccc}
$\{0\}$ & $\{0,2\}$ & $\{0,1,2,3\}$\\[1ex]
\begin{tikzpicture}[scale=0.8]
\draw [thick] (0,0) rectangle + (1,1);
\draw [thick] (0,1) rectangle + (2,1);
\node [above] at (0.5,0.1) {$3,2$};
\node [above] at (0.5,1.1) {$1,6$};
\node [above] at (1.5,1.1) {$0,4$};
\end{tikzpicture}
&
\begin{tikzpicture}[scale=0.8]
\draw [thick] (0,1) rectangle + (3,1);
\node [above] at (0.5,1.1) {$2,3$};
\node [above] at (1.5,1.1) {$0,1$};
\node [above] at (2.5,1.1) {$1,5$};
\node at (0,0.15) {};
\end{tikzpicture}
&
\begin{tikzpicture}
\node at (0,0.15) {};
\node at (0,1.25) {$\emptyset$};
\end{tikzpicture}
\end{tabular}
\end{center}
where the order of the elements in a row does not matter any more. Furthermore, the $\usigma$-tabloid can also be written as
\begin{center}
\begin{tabular}{ccc}
$\{0\}$ & $\{0,2\}$ & $\{0,1,2,3\}$\\[1ex]
\begin{tikzpicture}[scale=0.8]
\draw [thick] (0,0) rectangle + (1,1);
\draw [thick] (0,1) rectangle + (2,1);
\node [above] at (0.5,0.1) {$3,2$};
\node [above] at (0.5,1.1) {$0,4$};
\node [above] at (1.5,1.1) {$1,6$};
\end{tikzpicture}
&
\begin{tikzpicture}[scale=0.8]
\draw [thick] (0,1) rectangle + (3,1);
\node [above] at (0.5,1.1) {$0,1$};
\node [above] at (1.5,1.1) {$0,3$};
\node [above] at (2.5,1.1) {$3,5$};
\node at (0,0.15) {};
\end{tikzpicture}
&
\begin{tikzpicture}
\node at (0,0.15) {};
\node at (0,1.25) {$\emptyset$};
\end{tikzpicture}
\end{tabular}
\end{center}
as the fillings only differ by a permutation of the rows and colour changes using the elements of $U$ in the Young diagram associated with $U$. Hence, we can think of the colours in the diagram associated with $U$ as being elements of $C_4/U$.
\end{example}

Note that $W$ acts on the set of $\usigma$-tabloids. We define a notion of transitivity using $\usigma$-tabloids.
\begin{definition}
A subset $Y$ of $W$ is called \emph{$\usigma$-transitive} if $Y$ is transitive on the set of $\usigma$-tabloids. 
\end{definition}
The stabiliser of a $\usigma$-tabloid is obtained by stabilising every row of every diagram. In the example above, the stabiliser is given by
\[
	\left(\{0\} \wr S_{\{4,6\}} \times \{0\} \wr S_{\{2\}} \right) \times \{0,2\} \wr S_{\{1,3,5\}} \cong \left(S_2 \times S_1 \right) \times C_2 \wr S_3.
\]
In general, the stabiliser of the tableau associated with $U$ is $U \wr S_{\usigma(U)}$. Hence, the stabiliser of a $\usigma$-tabloid is given by
\[
H_{\usigma}=\prod_{U\le G}\left(U \wr S_{\usigma(U)}\right) \cong  \prod_{U\le G}\left(U \wr S_{\usigma(U)_1} \times U \wr S_{\usigma(U)_2} \times \cdots\right)
\]
where $S_{\usigma(U)}$ denotes the Young subgroup associated with the partition $\usigma(U)$. If $\usigma(U) = \emptyset$ for some subgroup $U$, then $U \wr S_{\usigma(U)}$ is understood to be the trivial group. We write $\xi^\usigma$ for the permutation character $\xi_{H_\usigma}$.

By Lemma~\ref{lem:transitive group action}, a subset $Y$ of $W$ is \emph{$\usigma$-transitive} if and only if $Y$ is transitive on the cosets $W/H_{\usigma}$. Thus, we do not get more notions of transitivity by considering compositions instead of partitions in the definition of the set $\Sigma_n(G)$ since the subgroups $H_{\usigma}$ involved are conjugate to each other. Note that this also justifies that in Theorem~\ref{thm:main_result_introduction} we only consider $(\sigma,k)$-flags where $\sigma$ is a partition, since the stabiliser of a $(\sigma,k)$-flag is $S_\sigma \times C_2 \wr S_k$ where $S_\sigma = S_{\sigma_1} \times S_{\sigma_2} \times \cdots$ is a Young subgroup.

If $G$ is the group of one element, then $\usigma\in\Sigma_n(G)$ is given by a partition $\sigma$ of $n$ via $\usigma(G)=\sigma$ and the notion of a $\usigma$-transitive set coincides with that of a $\sigma$-transitive subset of $S_n$ studied in~\cite{MarSag2007}.

We also emphasise the important special case that arises for $\abs{G}>1$ when $\usigma\in\Sigma_n(G)$ is given by $\usigma(\{0\})=(1^t)$ and $\usigma(G)=(n-t)$. The corresponding subgroup is $G \wr S_{n-t}$ which appeared in Section~\ref{sec:first_def_of_multiple_transitivity}. Then a $\usigma$-transitive subset of $W$ is transitive on the set of $t$-tuples $((c_1,i_1),\dots,(c_t,i_t))$, where $i_1,i_2,\dots,i_t \in [n]$ are pairwise distinct integers and $c_1,c_2,\dots,c_t$ are elements of $G$. For $G=C_r$, we shall study these subsets in more detail in Section~\ref{sec:polys_and_designs}.

We can also give a geometric interpretation of some transitivity types using the $n$-dimensional cube.
\begin{example}
Let $G = C_2$ and $n \in \N$ so that $W$ is the symmetry group of the $n$-dimensional cube. Consider $\usigma\in \Sigma_n(G)$ given by $\usigma(\{0\}) = (t)$ and $\usigma(C_2) = (n-t)$. Then the stabiliser of a $\usigma$-tabloid is the subgroup 
\[
	H_{t} = S_t \times C_2 \wr S_{n-t}.
\]
The group $H_t$ is also the stabiliser of an $(n-t)$-dimensional face of the cube. Thus, a subset $Y$ of $W$ is $\usigma$-transitive if and only if $Y$ is transitive on the $(n-t)$-dimensional faces of the $n$-dimensional cube.

We can also give an interpretation for the hyperoctahedron. If we view $C_2 \wr S_n$ as the symmetry group of the hyperoctahedron, then the subgroup
\[
	H_{t} = S_t \times C_2 \wr S_{n-t}
\]
is the stabiliser of a $(t-1)$-dimensional face of the hyperoctahedron. Thus, a subset $Y$ of $W$ is $\usigma$-transitive if and only if it is transitive on the $(t-1)$-dimensional faces of the hyperoctahedron.
\end{example}

The following properties of the subgroup $H_\usigma$ will be helpful.

\begin{lemma}\label{lem:properties of H}
Let $\usigma \in \Sigma_n(G)$ and write $H = H_\usigma$. Then:
\begin{enumerate}[label=\textup{(}\alph*\textup{)}]
\item{
	$H \cap G^n = \prod\limits_{U \leq G} U^{|\usigma(U)|}$
}
\item{
	$H \cap S_n = \prod\limits_{U \leq G}S_{\usigma(U)}$
}
\item{
	$(H \cap G^n) \cap (H \cap S_n) = \{\id\}$
}
\item{
	$H = (H \cap G^n) (H \cap S_n)$
}
\end{enumerate}
\end{lemma}

\begin{proof}
Parts (a) and (b) follow immediately from the fact that $H_\usigma$ is a direct product of wreath products on disjoint subsets of $[n]$. Intersecting $H$ with the base group gives the direct product of the base groups of the factors. Part~(c) follows because $G^n \cap S_n = \{\id\}$. For (d), notice that $H \cap G^n$ and $H \cap S_n$ are subgroups of $H$ with trivial intersection. Since $H$ is a finite group, the result follows because $|H \cap G^n||H \cap S_n| = |H|$.
\end{proof}

In the next section, we show how these structures can be characterised as designs in an association scheme.

\subsection{Decomposition of the permutation characters}

For decomposing the permutation characters of the subgroups $H_{\usigma}$, a certain algorithmically defined relation is helpful. A similar approach for decomposing permutation characters of the hyperoctahedral group $B_n$ (the case $G = C_2$) appears in \cite{May1975}, stated in a very restricted manner. We adapt this approach to wreath products and generalise it to cover all finite abelian groups $G$.

\begin{remark}
Let us pause for a moment and carefully go through the setup because the notation can become overwhelming quickly.

We have two index sets that we are working with, the set $\Sigma_n(G)$ and the set~$\Lambda_n(G)$. The first one indexes the \emph{types of transitivity} and the second one indexes the \emph{irreducible characters}. Our goal is to find a bijection between transitivity types and sets of irreducible characters. This bijection takes as input a permutation character (in other words, a transitivity type) and outputs the set of its irreducible constituents. We do not get to choose this bijection. It is the function that outputs the subset $I$ in Theorem~\ref{thm:general_method}. In other words, it tells us the subset $T$ such that the structures we are interested in are Delsarte $T$-designs.

The set $\Sigma_n(G)$ indexes the transitivity types. Some of them have a nice geometric interpretation on regular polytopes. In general, an element $\usigma \in \Sigma_n(G)$ defines a $\usigma$-tabloid involving coloured set partitions and a subset~$Y$ of~$W$ is $\usigma$-transitive if it is transitive on $\usigma$-tabloids. It can be helpful to stay inside the symmetry group of the hypercube $C_2 \wr S_n$ when reading the results and think of the elements of $\Sigma_n(G)$ as defining substructures of the polytope, giving a geometric interpretation of the transitivity type $\usigma$.

On the other hand, the set $\Lambda_n(G)$ indexes the irreducible characters of the wreath product $G \wr S_n$. We can think of an element $\ulambda \in \Lambda_n(G)$ as a tuple with $|G|$ entries where each entry is a partition. In contrast, we can think of the elements of $\Sigma_n(G)$ as tuples with one entry per subgroup of $G$ where each entry is a partition.
\end{remark}

We now define a relation on $\Sigma_n(G)\times \Lambda_n(G)$ that gives us the bijection that we need. For a subgroup $U$ of $G$, we write
\[
U^\circ =\{\text{$g\in G:\theta^g(u)=1$  for all $u\in U$}\},
\]
so $U^\circ$ consists of the elements of $G$ that correspond to irreducible characters that are trivial on $U$.
\begin{definition}
\label{def:arrow_order}
For $\usigma\in \Sigma_n(G)$ and $\ulambda\in\Lambda_n(G)$, we write
\[
\usigma \; \rightarrow \; \ulambda
\]
if there exists $\umu\in\Lambda_n(G)$ such that $\umu(g)\unlhd\ulambda(g)$ for each $g\in G$ and $\umu$ can be obtained from $\usigma$ as follows.
\begin{enumerate}[label=\textup{(}\arabic*\textup{)}]
\item{
For each $U\le G$, colour each box of the Young diagram of $\usigma(U)$ with an element of~$U^\circ$.
}
\item{
The element $\umu \in \Lambda_n(G)$ arises from the colouring from step (1) in the following way:\\
For each $g\in G$, consider only the boxes coloured with $g$. Each row of each partition $\usigma(U)$ with exactly $k$ boxes coloured with $g$ becomes a part of size $k$ in $\umu(g)$ (the actual position of the elements in a row does not matter, only how often $g$ was placed there). In other words, $\umu(g)$ has a part of size $k$ for every row of a Young diagram $\usigma(U)$ with exactly $k$ appearances of $g$.
}
\end{enumerate}
\end{definition}
It is important to note that $\{0\}^\circ = G$ and $G^\circ = \{0\}$. This means that the boxes in $\usigma(\{0\})$ can be coloured with every element of $G$ while the boxes in $\usigma(G)$ all have to be coloured with the element $0$.
\begin{example}
We take $G=C_6$ and identify $C_6$ with $\{0,1,2,3,4,5\}$ in the natural way. Then the subgroups of $C_6$ are $\{0\}$, $\{0,3\}$, $\{0,2,4\}$, and $\{0,1,2,3,4,5\}$. We get
\begin{align*}
	\{0\}^\circ &= \{0,1,2,3,4,5\},\\
	\{0,3\}^\circ &= \{0,2,4\},\\
	\{0,2,4\}^\circ &= \{0,3\},\\
	\{0,1,2,3,4,5\}^\circ &= \{0\}.
\end{align*}
These are the elements we can use in the colouring of the Young diagrams in step (1) of Definition~\ref{def:arrow_order}. Consider $\usigma\in\Sigma_{14}(C_6)$ given by
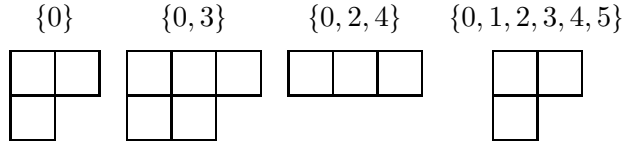
\begin{center}
\begin{tabular}{cccc}
$\{0\}$ & $\{0,3\}$ & $\{0,2,4\}$ & $\{0,1,2,3,4,5\}$\\[1ex]
\begin{ytableau}
\, & \, \\
\,
\end{ytableau}
&
\begin{ytableau}
\, & \, & \,\\
\, & \,
\end{ytableau}
&
\begin{ytableau}
\, & \, & \,\\
\end{ytableau}
&
\begin{ytableau}
\, & \,\\
\,\\
\end{ytableau}
\end{tabular}.
\end{center}
\vspace*{1ex}
A possible colouring of $\usigma$ is given by
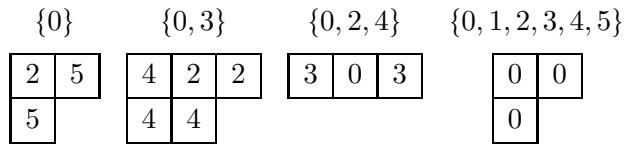
\begin{center}
\begin{tabular}{cccc}
$\{0\}$ & $\{0,3\}$ & $\{0,2,4\}$ & $\{0,1,2,3,4,5\}$\\[1ex]
\begin{ytableau}
2 & 5\\
5
\end{ytableau}
&
\begin{ytableau}
4 & 2 & 2\\
4 & 4
\end{ytableau}
&
\begin{ytableau}
3 & 0 & 3\\
\end{ytableau}
&
\begin{ytableau}
0 & 0\\
0\\
\end{ytableau}
\end{tabular}.
\end{center}
\vspace*{1ex}
Then the corresponding $\umu\in\Lambda_{14}(C_6)$ is given by
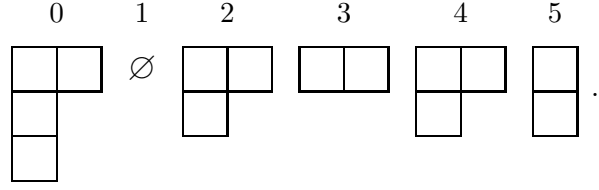
\begin{center}
\begin{tabular}{cccccc}
$0$ & $1$ & $2$ & $3$ & $4$ & $5$\\[1ex]
\ydiagram{2,1,1}
&
\raisebox{1.3ex}{$\emptyset$}
&
\ydiagram{2,1}
&
\ydiagram{2}
&
\ydiagram{2,1}
&
\ydiagram{1,1}
\end{tabular}.
\end{center}
\vspace*{1ex}
This shows that $\usigma \to \umu$ where
\[
	\umu(0) = (211),\; \umu(1) = \emptyset,\; \umu(2) = (21),\; \umu(3) = (2),\; \umu(4) = (21),\; \umu(5) = (11).
\]
It also shows that $\usigma \to \ulambda$ for every $\ulambda \in \Lambda_n(G)$ with $\umu(g) \unlhd \ulambda(g)$ for every $g \in G$. Using the dominance order $\unlhd$ to obtain $\ulambda$ from $\umu$ corresponds to moving up boxes in the Young diagrams of $\umu(g)$ for each $g \in G$ independently.
\end{example}
\begin{remark}
Here is another way to think about the relation $\rightarrow$. Start with the Young diagrams $\usigma(U)$ and $|G|$ empty containers that are indexed with the elements of $G$ (the tuple of these containers will become the element $\ulambda \in \Lambda_n(G)$ after the algorithm). For each row of each $\usigma(U)$, we have to move the boxes of the row to the containers indexed by the elements of $U^\circ$. Moving $k$ boxes from a row of $\usigma(U)$ to a container indexed by $g \in U^\circ$ means that we add a part of size $k$ to the container indexed by $g$. Only the number $k$ matters, not the actual position of the boxes in the row. Importantly, this means that for each row of each $\usigma(U)$ and each $g \in U^\circ$, we add either one part to the container indexed by $g$ or none. Each box of $\usigma$ must be moved into some container.

After all boxes have been moved, order all the parts in each container so that they form a partition. This defines an element $\umu \in \Lambda_n(G)$. Finally, one has the option to use the dominance order $\unlhd$ for each partition $\umu(g)$ independently so that we obtain an element $\ulambda \in \Lambda_n(G)$ with $\umu(g) \unlhd \ulambda(g)$ for every $g \in G$.
\end{remark}

We now prove that the relation $\to$ characterises the irreducible constituents of the permutation character $\xi^\usigma$.

\begin{theorem}
\label{thm:decomposition_perm_char}
For each $\usigma\in \Sigma_n(G)$ and each $\ulambda\in\Lambda_n(G)$, we have
\[
\bigang{\xi^\usigma, \chi^{\ulambda}} \ne 0 \; \Longleftrightarrow \; \usigma \; \rightarrow \; \ulambda.
\]
\end{theorem}

\begin{proof}
Let $\ulambda \in \Lambda_n(G)$. We pick a tuple $(g_1,\ldots,g_n) \in G^n$ such that every $g \in G$ occurs precisely $|\ulambda(g)|$ times. This defines a Young subgroup
\[
S=\prod_{g\in G}S_{g}
\]
where $S_{g}$ permutes the indices $i \in [n]$ with $g_i = g$. Hence, we have an ordered set partition $P$ of $[n]$ indexed by $G$. The part of $P$ indexed by $g \in G$ is the set of indices $i \in [n]$ with $g_i = g$.

We write $K=G\wr S = G^n \rtimes S$. By Theorem~\ref{thm:irreps_wreath_product}, we have $\chi^\ulambda=(\phi \cdot \rho)\uparrow_K^W$ where $\phi$ is the irreducible character of $G^n$ of type $(g_1,\ldots,g_n)$ and $\rho$ is the irreducible character of $S$ given by
\[
\rho=\bigotimes_{g\in G}\chi^{\ulambda(g)}.
\]
Write $H=H_\usigma$. By Frobenius reciprocity, we have
\[
\bigang{\xi^\usigma, \chi^{\ulambda}} = \bigang{1 \uparrow_H^W, \chi^\ulambda}=\bigang{1 \uparrow_H^W, (\phi \cdot \rho)\uparrow_K^W}=\bigang{1 \uparrow_H^W \downarrow_K^W, \phi \cdot \rho}.
\]
For $s\in W$, write $H^s=sHs^{-1}$. Let $R$ be a complete set of double coset representatives of $K\,\backslash W/H$. Then by Mackey's formula (Theorem~\ref{thm:mackey}), we have
\[
1 \uparrow_H^W \downarrow_K^W=\sum_{s \in R} 1 \downarrow_{H^s\cap K}^{H^s} \uparrow_{H^s\cap K}^K=\sum_{s \in R} 1 \uparrow_{H^s\cap K}^K,
\]
and by Frobenius reciprocity, we have
\begin{equation}
\label{eqn:scalar_prod_double_cosets}
\bigang{\xi^\usigma, \chi^{\ulambda}}=\sum_{s \in R} \bigang{1\uparrow_{H^s\cap K}^K , \phi \cdot \rho}=\sum_{s \in R} \bigang{1 , (\phi \cdot \rho)\downarrow_{H^s\cap K}^K}.
\end{equation}
Note that all summands on the right hand side are non-negative as each summand is the multiplicity of the irreducible character $1$ in the decomposition of the character $(\phi \cdot \rho)\downarrow_{H^s\cap K}^K$. Thus, the total sum is non-zero if and only if at least one summand is non-zero. Since $K=G\wr S$, we have $G^n \subseteq K$, so we may assume that all representatives in~$R$ are elements of $S_n$. Hence, \eqref{eqn:scalar_prod_double_cosets} is non-zero if and only if there exists $s\in S_n$ such that
\begin{equation}\label{eqn:thing_that_has_to_be_nonzero}
\bigang{1,(\phi \cdot \rho)\downarrow_{H^s\cap K}^K}\ne 0.
\end{equation}
From Lemma~\ref{lem:properties of H}, we find $H^s = (H^s \cap G^n)(H^s \cap S_n)$. Hence, as $K = G^n \rtimes S$, it follows that $H^s \cap K = (H^s \cap G^n)(H^s \cap S)$. Since $\phi$ is a character of~$G^n$ and $\rho$ is a character of~$S$ and since we have $(H^s \cap G^n)(H^s \cap S) = \{\id\}$, elementary calculations reveal that
\[
	\bigang{1,(\phi \cdot \rho)\downarrow_{H^s\cap K}^K}  = \big\langle 1 , \phi\downarrow_{H^s\cap G^n}^{G^n} \big\rangle \cdot \big\langle 1 , \rho\downarrow_{H^s\cap S}^S \big\rangle.
\]
Hence, the inner product in~\eqref{eqn:thing_that_has_to_be_nonzero} is non-zero if and only if
\begin{equation}
\big\langle 1 , \phi\downarrow_{H^s\cap G^n}^{G^n} \big\rangle \cdot \big\langle 1 , \rho\downarrow_{H^s\cap S}^S \big\rangle\ne 0.   \label{eqn:two_factors}
\end{equation}

The left hand side of~\eqref{eqn:two_factors} is non-zero if and only if both factors are non-zero. We focus on the second factor first.

As $H^s$ is conjugate to $H$, it is the stabiliser of a $\usigma$-tabloid. Hence, by Lemma~\ref{lem:properties of H}, $H^s\cap S_n$ is a Young subgroup of $S_n$ whose factors correspond to the parts of $\usigma(U)$ as $U$ ranges over the subgroups of $G$. Thus, there is an ordered set partition $Q$ of $[n]$ corresponding to the Young subgroup $H^s \cap S_n$ where every part is associated with a row of a Young diagram $\usigma(U)$ for some subgroup $U$.

Recall that $S$ is a Young subgroup with corresponding ordered set partition~$P$ where every part is associated with an element $g \in G$. It follows that $H^s \cap S$ is a Young subgroup where the corresponding ordered set partition consists precisely of the (non-empty) parts $p \cap q$ with $p \in P$ and $q \in Q$. This ordered set partition defines a colouring of $\usigma$ in the following way.

For every element $p \cap q$ with $p \in P$ and $q \in Q$, we colour the boxes corresponding to $p \cap q$ in the row associated with $q$ with the element $g$ associated with $p$. Hence, by collecting all parts associated with an element $g \in G$ into a partition $\umu(g)$, we get a partition tuple $\umu\in\Lambda_n(G)$ with $\abs{\umu(g)}=|\ulambda(g)|$ for all $g\in G$ and
\begin{equation}
H^s\cap S = \prod_{g \in G} S_{\umu(g)}.		\label{eqn:subgroup_umu}
\end{equation}
In other words, $\umu$ describes how the factors $S_g$ of $S$ decompose after intersecting $H^s$ with $S$. Each part of $\umu(g)$ corresponds to a row of $\usigma(U)$ for some $U\le G$ with exactly that many boxes coloured with $g$. Moreover, we have
\begin{align*}
\bigang{1 , \rho\downarrow_{H^s\cap S}^S}&=\bigang{1 ,\bigotimes_{g\in G}\chi^{\ulambda(g)}\downarrow_{H^s\cap S}^S}\\
&=\prod_{g\in G}\bigang{1 ,\chi^{\ulambda(g)}\downarrow^{S_{g}}_{S_{\umu(g)}}}\\
&=\prod_{g\in G}\bigang{1 \uparrow^{S_{g}}_{S_{\umu(g)}},\chi^{\ulambda(g)}}.
\end{align*}
The left hand side is non-zero if and only if every factor on the right hand side is non-zero. By Young's rule (see for example \cite[Theorem 2.11.2]{Sag2001}), this happens if and only if $\umu(g)\unlhd\ulambda(g)$ for each $g\in G$.

Now we consider the first factor on the left hand side of~\eqref{eqn:two_factors}. Since $\phi$ is an irreducible character of $G^n$, it is of degree $1$ and hence also an irreducible character of $H^s \cap G^n$. Because the trivial character $1$ is also an irreducible character of $H^s \cap G^n$, it follows that the inner product is non-zero if and only if $\phi \downarrow^{G^n}_{H^s\cap G^n} = 1$. In other words, $\phi(x) = 1$ for all $x \in H^s \cap G^n$.

Since $H^s$ is conjugate to $H$, it is the stabiliser of a $\usigma$-tabloid. Hence, $H^s\cap G^n$ is a direct product of subgroups of $G$ where $U\le G$ occurs $\abs{\usigma(U)}$ times (see Lemma~\ref{lem:properties of H}). This means that to each $i\in\{1,2,\dots,n\}$, we assign a subgroup $U_i$ of $G$ such that $H^s\cap G^n=U_1\times\cdots\times U_n$. Remember that the character~$\phi$ is of type $(g_1,\dots,g_n)$. Since $\phi$ is trivial on $H^s \cap G^n = U_1\times\cdots\times U_n$, we find that $(g_1,\dots,g_n)\in U_1^\circ\times\cdots\times U_n^\circ$. Hence, the left factor in \eqref{eqn:two_factors} is non-zero if and only if all boxes of the Young diagram $\usigma(U)$ are coloured with elements of $U^\circ$ for every subgroup $U \leq G$.

Having completed the setup, we can now prove the statement of the theorem. For the forward implication, suppose that $\bigang{\xi^\usigma, \chi^{\ulambda}} \neq 0$. Then both factors in~\eqref{eqn:two_factors} are non-zero. From~\eqref{eqn:subgroup_umu}, we obtain a colouring $\umu \in \Lambda_n(G)$ of $\usigma$. Since the first factor of~\eqref{eqn:two_factors} is non-zero, all rows of $\usigma(U)$ are coloured with elements of $U^\circ$, corresponding to step (1) in Definition~\ref{def:arrow_order}. Every part of $\umu(g)$ comes from a row with that many boxes coloured with $g$, corresponding to step~(2) in Definition~\ref{def:arrow_order}. Since the second factor of~\eqref{eqn:two_factors} is non-zero, we have $\umu(g) \unlhd \ulambda(g)$ for every $g \in G$. Thus, we obtain $\usigma \to \ulambda$.

For the reverse direction, assume that $\usigma \to \ulambda$. Then there exists a colouring $\umu \in \Lambda_n(G)$ of $\usigma$ obtained from steps (1) and (2) in Definition~\ref{def:arrow_order}. Now consider the Young subgroup
\[
	T=\prod_{g\in G}T_{g}
\]
of $S_n$ where each part $T_g$ permutes the boxes of $\usigma$ (equivalently elements $i \in [n]$) that are coloured with $g \in G$ in the colouring $\umu$. Notice that it follows from the definition of the relation $\to$ that $|\umu(g)|=|\ulambda(g)|$ for every $g \in G$.

Recall that the tuple $(g_1,\ldots,g_n)$ we picked in the beginning of the proof was arbitrary. The only requirement was that the element $g$ appears precisely $|\ulambda(g)|$ times. Hence, we can rearrange the tuple $(g_1,\ldots,g_n)$ such that $T_g$ permutes precisely the indices $i \in [n]$ with $g_i = g$. We find that all of the calculations we did for $S$ also hold for $T$. In fact, $S$ and $T$ are Young subgroups of the same shape so they are conjugate.

Remember that $H\cap S_n$ is a Young subgroup of $S_n$ whose parts correspond to the parts of $\usigma(U)$ for $U \leq G$. The group $H$ permutes the boxes in the rows of a Young diagram $\usigma(U)$ and the group $T$ permutes all boxes that are coloured with the same element $g \in G$. It follows that the group $H \cap T$ permutes all boxes that are coloured with the same element $g \in G$ in a common row of a Young diagram $\usigma(U)$. Hence, we find that $H \cap T$ is a Young subgroup of the form $\prod_{g\in G}S_{\umu(g)}$ for our colouring~$\umu$. Thus, Equation \eqref{eqn:subgroup_umu} holds for $S=T$ and $s = \id$.

Since $\umu(g)\unlhd\ulambda(g)$ for each $g\in G$, the second factor of \eqref{eqn:two_factors} is non-zero for $S=T$ and $s = \id$. Moreover, $\phi$ is trivial on $H\cap G^n$, so the first factor of \eqref{eqn:two_factors} is non-zero for $s=\id$. All in all, we obtain $\bigang{\xi^\usigma, \chi^{\ulambda}} \neq 0$.
\end{proof}

We can use Theorem \ref{thm:decomposition_perm_char} to obtain a characterisation of $\usigma$-transitive sets in $W$.
\begin{theorem}
\label{thm:characterisation_designs}
Let $\usigma\in\Sigma_n(G)$ and let $Y$ be a non-empty subset of $W$ with dual distribution $(a'_\ulambda)$. Then $Y$ is $\usigma$-transitive if and only if 
\[
a'_\ulambda=0\quad\text{for all $\ulambda\in\Lambda_n(G)$ satisfying $\usigma\to\ulambda$ and $\ulambda(0)\ne (n)$},
\]
where $0$ is the identity of $G$.
\end{theorem}

\begin{proof}
Follows immediately from Theorem \ref{thm:general_method} and Theorem \ref{thm:decomposition_perm_char}.
\end{proof}
Note that if we let $G = C_1$, then we have $W = C_1 \wr S_n = S_n$ and $\usigma$-transitivity becomes $\sigma$-transitivity in $S_n$. In this case, the set $\Lambda_n(C_1)$ is the set of partitions of $n$ and the relation $\to$ becomes the dominance order $\unlhd$. Thus, Theorem \ref{thm:characterisation_designs} specialises to Theorem 4 in \cite{MarSag2007}.

Theorem~\ref{thm:characterisation_designs} is crucial in order to derive a generalisation of the Livingstone--Wagner theorem in Section~\ref{sec:livingstone wagner} and to describe $t$-transitive sets in $C_r \wr S_n$ in Section~\ref{sec:polys_and_designs}.

\subsection{Cliques}

In this section we consider so-called \emph{cliques} in $W$ and discuss their relationship to transitive subsets.
\begin{definition}
Let $\usigma \in \Sigma_n(G)$. A non-empty subset $Y$ of $W$ is called a \emph{$\usigma$-clique} if for every pair of distinct elements $x,y \in Y$ there is no $\usigma$-tabloid that is fixed by $x^{-1}y$.
\end{definition}

We are interested in an algebraic characterisation of $\usigma$-cliques as cliques in an association scheme. This comes down to determining the conjugacy classes of the elements in the stabiliser $H_{\usigma}$ of a $\usigma$-tabloid. Note that all stabilisers of $\usigma$-tabloids are conjugate to each other, so the answer does not depend on the specific $\usigma$-tabloid chosen. Remember that $H_{\usigma}$ is an inner direct product of subgroups of the form $U \wr S_k$ for a subgroup $U \leq G$ and a positive integer $k$. The conjugacy classes of the group $U \wr S_k$ are indexed by the set $\Lambda_k(U)$. We interpret an element $\ulambda \in \Lambda_k(U)$ as an element of $\Lambda_k(G)$ by defining $\ulambda(g) = \emptyset$ for every $g \in G$ with $g \not\in U$. Moreover,  for $k < n$ we interpret an element $\ulambda \in \Lambda_k(G)$ as an element $\widetilde{\ulambda} \in \Lambda_n(G)$ by defining $\widetilde{\ulambda}(0) = \ulambda(0) \cup 1^{n-k}$ and $\widetilde{\ulambda}(g) = \ulambda(g)$ for all $g \in G \,\setminus \{0\}$.

For two elements $\ulambda \in \Lambda_m(G)$, $\umu \in \Lambda_n(G)$, we define the element $\ulambda \cup \umu \in \Lambda_{m+n}(G)$ via  $\left(\ulambda \cup \umu\right)(g) = \ulambda(g) \cup \umu(g)$ for every $g \in G$. We can now describe the conjugacy classes of the direct product of two groups of the form $U \wr S_k$.
\begin{lemma}\label{lem:conjugacy_classes_direct_product}
Let $U,V \leq G$ be subgroups of $G$ and $k,\ell \in \N$ with $k+\ell \leq n$. Consider the subgroup $H=\left( U \wr S_A \right) \times \left( V \wr S_B \right)$ of $G \wr S_n$ where $A,B \subseteq [n]$ are disjoint subsets with $|A| = k$ and $|B| = \ell$. Then the conjugacy classes of $H$ are indexed by the set
\[
	\left\lbrace \ulambda \cup \umu : \ulambda \in \Lambda_k(U),\;\umu \in \Lambda_\ell(V) \right\rbrace.
\]
\end{lemma}
\begin{proof}
Let $x=(f,\sigma) \in U \wr S_A$ and $y=(g,\pi) \in V \wr S_B$. Then their product is given by
\[
	xy = (f+\sigma g,\sigma\pi).
\]
Here, we have $f = (f_1,\ldots,f_n) \in G^n$ with $f_a \in U$ for $a \in A$ and $f_i = 0$ otherwise. Similarly, we have $g = (g_1,\ldots,g_n) \in G^n$ with $g_b \in V$ for $b \in B$ and $g_i = 0$ otherwise. Note that since $A$ and $B$ are disjoint, $\sigma$ acts trivially on $g$ and we find that
\[
	xy = (f+g,\sigma\pi).
\]
Recall that the conjugacy classes of $x$ and $y$ are given by the cycle lengths of $\sigma$ and $\pi$, respectively, together with the signs of the cycles (see Section~\ref{sec:cc_wreath_product}). Since $S_A$ and $S_B$ permute disjoint subsets of $[n]$, we find that the set of cycle lengths of $\sigma\pi$ is simply the disjoint union of the set of cycle lengths of $\sigma$ and the set of cycle lengths of $\pi$. Moreover, since $A$ and $B$ are disjoint, the cycle signs of $\sigma$ and $\pi$ do not change. Hence, if $x$ is of cycle type $\ulambda$ and $y$ is of cycle type $\umu$, then $xy$ is of cycle type $\ulambda \cup \umu$. The cycle type of $xy$ as an element of $G \wr S_n$ is then obtained by adding cycles of length 1 for all elements of $[n] \,\setminus\left(A\cup B\right)$ (if $k+\ell < n$). This corresponds exactly to how we identified elements of $\Lambda_{k+\ell}(G)$ with elements of $\Lambda_n(G)$.
\end{proof}
Denote by $C(\usigma)$ the set of all $\ulambda \in \Lambda_n(G)$ such that there exists an element in~$H_{\usigma}$ that lies in the conjugacy class $C_{\ulambda}$ of $W$. Applying Lemma~\ref{lem:conjugacy_classes_direct_product} repeatedly immediately gives the following result.
\begin{theorem}\label{thm:cc_stabiliser}
Let $\usigma \in \Lambda_n(G)$. Consider a stabiliser $H_{\usigma}$ of a $\usigma$-tabloid and write
\[
	H_{\usigma} = H_1 \times \dots \times H_\ell
\]
for some $\ell \in \N$ where $H_i = U_i \wr S_{k_i}$ for some subgroup $U_i \leq G$ and $k_i \in \N$. Then we have
\[
	C(\usigma) = \left\lbrace \bigcup_{i = 1}^\ell \ulambda^{i} : \ulambda^{i}\in \Lambda_{k_i}(U_i) \right\rbrace.
\]
\end{theorem}
The following result should be compared to Theorem~\ref{thm:characterisation_designs} and highlights that the concept of a $\usigma$-clique is dual to the concept of $\usigma$-transitivity. Recall that the conjugacy class of the identity of $W$ is indexed by the unique element $\ulambda \in \Lambda_n(G)$ with $\ulambda(0)=(1^n)$.
\begin{theorem}\label{thm:characterisation_cliques}
Let $\usigma \in \Sigma_n(G)$ and let $Y$ be a non-empty subset of $W$ with inner distribution $(a_{\ulambda})$. Then $Y$ is a $\usigma$-clique if and only if
\[
	a_\ulambda=0\quad\text{for all $\ulambda\in C(\usigma)$ with $\ulambda(0)\ne (1^n)$}.
\]
\end{theorem}
\begin{proof}
Let $x,y \in Y$ be distinct elements of $Y$. Then the conjugacy class of $x^{-1}y$ is not that of the identity of $W$. If $a_{\ulambda} = 0$ for every $\ulambda \in C(\usigma)$ with $\ulambda(0) \neq (1^n)$, then the conjugacy class of $x^{-1}y$ is not indexed by an element $\ulambda \in C(\usigma)$ because in view of \eqref{eqn:def_inner_distribution_general}, $x^{-1}y \in C_{\ulambda}$ implies $a_{\ulambda} > 0$. It follows that $x^{-1}y$ cannot be an element of $H^w = wH_{\usigma}w^{-1}$ for any $w \in W$. Since the groups $H^w$ for $w \in W$ are exactly all the stabilisers of $\usigma$-tabloids, $x^{-1}y$ does not fix a $\usigma$-tabloid and thus $Y$ is a $\usigma$-clique.

For the converse direction, note that if $x \in C_{\ulambda}$ stabilises a $\usigma$-tabloid $T$, then $wxw^{-1}$ stabilises the $\usigma$-tabloid $wT$. Hence, either all elements of a conjugacy class $C_{\ulambda}$ stabilise a $\usigma$-tabloid or none. Thus, if $Y$ is a $\usigma$-clique, then the conjugacy class of the quotient $x^{-1}y$ for distinct $x,y \in Y$ cannot be a conjugacy class of an element of $H_{\usigma}$. This shows that $a_{\ulambda} = 0$ for $\ulambda \in C(\usigma)$ with $\ulambda(0) \neq (1^n)$.
\end{proof}

We now establish a connection between $\usigma$-transitive sets and $\usigma$-cliques in $W$.
\begin{theorem}\label{thm:clique_design_bound}
Let $\usigma \in \Lambda_n(G)$, let $H=H_{\usigma}$ be the stabiliser of a $\usigma$-tabloid and let $Y$ be a non-empty subset of $W$.
\begin{enumerate}[label=\textup{(}\alph*\textup{)}]
\item{
If $Y$ is a $\usigma$-clique, then $|Y| \leq |W|/|H|$ with equality if and only if $Y$ is $\usigma$-transitive.
}
\item{
If $Y$ is $\usigma$-transitive, then $|Y| \geq |W|/|H|$ with equality if and only if $Y$ is a $\usigma$-clique.
}
\end{enumerate}
In both cases, equality implies that $Y$ is sharply $\usigma$-transitive.
\end{theorem}
\begin{proof}
Observe that for each $(x,y) \in H \times Y$ there is precisely one element $w \in W$ with $wx = y$ (namely $w = yx^{-1}$). Thus, we have
\begin{equation}
	\sum_{w \in W} \left|Y \cap wH \right| = |Y| \cdot |H|. \label{eqn:sum_clique_transitive}
\end{equation}
The quotient of any two elements in $Y \cap wH$ fixes a $\usigma$-tabloid. Hence, if $Y$ is a $\usigma$-clique, then every summand on the left hand side of \eqref{eqn:sum_clique_transitive} is at most 1. This gives $|Y| \cdot |H| \leq |W|$, proving the bound in (a). Now, since $H$ is the stabiliser of a $\usigma$-tabloid $T$, the coset $w H$ contains precisely the elements of $W$ that map $T$ to $wT$. Thus, if $Y$ is $\usigma$-transitive, then every summand on the left hand side of \eqref{eqn:sum_clique_transitive} is at least 1. This gives $|Y| \cdot |H| \geq |W|$, proving the bound in (b).

In both cases, equality occurs if and only if $|Y \cap w \widetilde{H}| = 1$ for every $w \in W$ and every stabiliser $\widetilde{H}$ of a $\usigma$-tabloid. This is equivalent to $Y$ being sharply $\usigma$-transitive.
\end{proof}
Theorem~\ref{thm:clique_design_bound} can also be obtained using the so-called clique-coclique bound (see \cite[Theorem 3.9]{Del1973}).

\section{Comparison of transitivity types}\label{sec:livingstone wagner}

We now turn to generalisations of Theorem \ref{thm:livingstone wagner} by Livingstone and Wagner. Let $\usigma,\utau\in\Sigma_n(G)$. We write $\usigma\succeq\utau$ if for every subset $Y$ of $W$ we have
\[
\text{$Y$ is $\usigma$-transitive $\;\Longrightarrow\;$ $Y$ is $\utau$-transitive}.
\]

Hence, writing
\[
M_{\usigma}=\{\ulambda\in\Lambda_n(G):\usigma\to\ulambda\},
\]
Theorem~\ref{thm:characterisation_designs} implies
\[
\usigma\succeq\utau\;\Longleftrightarrow\; M_{\usigma}\supseteq M_{\utau}.
\]

We shall give a characterisation of the relation $\succeq$ in special cases. We say that $\usigma\in\Sigma_n(G)$ is \emph{simple} if $\usigma(U)=\emptyset$ for each non-trivial proper subgroup~$U$ of $G$. If $|G| > 1$, then the \emph{shape} of a simple $\usigma\in\Sigma_n(G)$ is the pair $(\sigma,\rho)$, where $\usigma(\{0\})=\sigma$ and $\usigma(G)=\rho$. In order to fit the case $|G|=1$ into this framework, we define the shape of a partition $\sigma \in \Sigma_n(\{0\})$ to be $(\sigma,\emptyset)$. We shall identify a simple $\usigma\in\Sigma_n(G)$ with its shape. We define the \emph{size} of a double partition $(\sigma,\rho)$ to be $\abs{\sigma}+\abs{\rho}$. For double partitions $(\sigma,\rho)$ and $(\tau,\pi)$ of the same size $n$, we write $(\sigma,\rho)\succeq(\tau,\pi)$ if their corresponding symbols in $\Sigma_n(G)$ are related in this way. Note that if $G$ has prime order, then simplicity of $\usigma\in\Sigma_n(G)$ is no restriction.

We call $\usigma \in \Sigma_n(G)$ \emph{parabolic} if it is simple of shape $(\sigma,(k))$ for a non-negative integer~$k$ (the case $k=0$ should be read as $(\sigma,\emptyset)$). We write $(\sigma,k)$ for the shape $(\sigma,(k))$. The term "parabolic" stems from the fact that if $\usigma$ is parabolic in the hyperoctahedral group, the case $G = C_2$, then the corresponding stabiliser $H_{\usigma}$ is a parabolic subgroup if $W$ is viewed as a Coxeter group. It turns out that parabolic $\usigma$ are especially nice to work with and capture important transitivity types like $t$-transitivity for $\usigma = ((1^t),n-t)$ and $t$-homogeneity for $\usigma = ((t),n-t)$.

We start with some simple yet helpful properties of the relation $\succeq$.
\begin{proposition}
\label{pro:necessary_conditions}
Let $(\sigma,\rho)$ and $(\tau,\pi)$ be two double partitions of the same size satisfying $(\sigma,\rho)\succeq(\tau,\pi)$. Then we have
\begin{enumerate}[label=\textup{(}\alph*\textup{)}]
\item{
	$\rho\unlhd\pi$,
}
\item{
	$\sigma'\unrhd\tau'$,
}
\item{
	$\sigma\cup\rho\unlhd\tau\cup\pi$.
}
\end{enumerate}
\end{proposition}
\begin{proof}
To prove (a) and (b), consider $\ulambda\in\Lambda_n(G)$ given by $\ulambda(0)=\pi$ and $\ulambda(g)=\tau$ for some fixed non-zero $g\in G$ (in the case $G = \{0\}$, we have $\pi=\emptyset$ and use $\ulambda\in \Lambda_n(\{0\})$ with $\ulambda(0) = \tau$). Then we have $(\tau,\pi)\to\ulambda$, and since $(\sigma,\rho)\succeq(\tau,\pi)$, we must also have $(\sigma,\rho)\to\ulambda$. In other words, there must be a colouring of $(\sigma,\rho)$ giving rise to $\umu \in \Lambda_n(G)$ with $\umu(g) \unlhd \ulambda(g)$ for every $g \in G$.

Let $\umu\in\Lambda_n(G)$ be obtained from such a colouring of $(\sigma,\rho)$. The parts of~$\rho$ must occur as parts in $\umu(0)$. Hence, $\rho\subseteq \umu(0)\unlhd\ulambda(0) = \pi$, and so $\rho\unlhd\pi$, proving~(a).

Now consider the partition $\umu(g)$ for the fixed non-zero $g \in G$ from above. Only boxes of $\sigma$ can be coloured with $g$, so all boxes of $\umu(g)$ have to come from $\sigma$. Notice that for every $k$, the number of boxes in the first $k$ columns of $\umu(g)$ cannot be larger than the number of boxes in the first $k$ columns of $\sigma$. Hence $\sigma'\unrhd\umu(g)' \unrhd \ulambda(g)' = \tau'$ and so $\sigma'\unrhd\tau'$, proving~(b).

To prove (c), consider $\ulambda\in\Lambda_n(G)$ given by $\ulambda(0)=\tau\cup \pi$. Then we have $(\tau,\pi)\to\ulambda$ and so $(\sigma,\rho)\to\ulambda$. The only way to achieve this is to colour every box of $(\sigma,\rho)$ with the element $0$. This forces $\sigma\cup \rho\unlhd\tau\cup \pi$.
\end{proof}

\begin{example}\label{ex:not sufficient}
The following example shows that the conditions (a), (b), (c) in Proposition~\ref{pro:necessary_conditions} are in general not sufficient to guarantee $(\sigma,\rho)\succeq(\tau,\pi)$. Take
\[
	(\sigma,\rho)=((311),(22)) \; \text{ and } \; (\tau,\pi)=((22),(311)).
\]
Then, the conditions of Proposition~\ref{pro:necessary_conditions} are satisfied.

Let $\ulambda\in\Lambda_9(G)$ be given by $\ulambda(0)=(31^4)$ and $\ulambda(g)=(1^2)$ for some fixed non-zero $g\in G$. Then $(\tau,\pi)\to\ulambda$ holds, but $(\sigma,\rho)\to\ulambda$ does not hold. Hence, $(\sigma,\rho)\succeq(\tau,\pi)$ is not true.

Notice however that in the case $G=\{0\}$, $G \wr S_n$ is the symmetric group and all three conditions in Proposition~\ref{pro:necessary_conditions} reduce to $\sigma \unlhd \tau$. This is sufficient to guarantee that a $\sigma$-transitive set is also $\tau$-transitive by Theorem~\ref{thm:lambda transitivity Livingstone Wagner}.
\end{example}

We shall however prove that the conditions (a) and (c) in Proposition~\ref{pro:necessary_conditions} are sufficient to guarantee $(\sigma,\rho)\succeq(\tau,\pi)$ in the parabolic case.

Let us first investigate how the relation $\succeq$ behaves when changing the shape $(\sigma,\rho)$ slightly.
\begin{lemma}[Dominance rule]\label{lem:dominance rule}
Let $(\sigma,\rho)$ and $(\tau,\pi)$ be two double partitions of the same size satisfying $\abs{\sigma}=\abs{\tau}$, $\sigma \unlhd\tau$ and $\rho\unlhd \pi$. Then we have $(\sigma,\rho) \succeq (\tau,\pi)$.
\end{lemma}

\begin{proof}
Since $\succeq$ is a transitive relation, we can show $(\sigma,\rho) \succeq (\tau,\pi)$ in two steps via $(\sigma,\rho) \succeq (\sigma,\pi) \succeq (\tau,\pi)$.\par
First, we show that $(\sigma,\rho) \succeq (\sigma,\pi)$. Consider $\ulambda\in\Lambda_n(G)$ with $(\sigma,\pi) \to \ulambda$. We have to show $(\sigma,\rho) \to \ulambda$. According to Definition~\ref{def:arrow_order}, there is $\ubeta\in\Lambda_n(G)$ arising from a colouring $C$ of the boxes of $(\sigma,\pi)$ such that $\ubeta(g)\unlhd\ulambda(g)$ for all $g\in G$. We need to find a colouring $\widetilde{C}$ of $(\sigma,\rho)$ in step (1) of Definition~\ref{def:arrow_order} such that we obtain $(\sigma,\rho) \to \ulambda$.

Define $\widetilde{C}$ by colouring the boxes of $\sigma$ in $(\sigma,\rho)$ in the same way as in the colouring $C$ of $(\sigma,\pi)$. Colour all boxes of $\rho$ with $0$ (this is forced because of $G^\circ = \{0\}$). This assignment defines a colouring $\widetilde{C}$ of $(\sigma,\rho)$ which gives $\underline{\gamma} \in \Lambda_n(G)$ with $\underline{\gamma}(g) = \ubeta(g)$ for $g \neq 0$. Moreover, writing $\alpha$ for the partition of boxes of $\sigma$ that are coloured with $0$, we have $\underline{\gamma}(0) = \alpha \cup \rho$ and $\ubeta(0) = \alpha \cup \pi$. Since we have $\rho \unlhd \pi$, we obtain $\underline{\gamma}(0) \unlhd \ubeta(0)$. Hence, we find that $\underline{\gamma}(g) \unlhd \ulambda(g)$ for every $g \in G$, proving $(\sigma,\rho) \to \ulambda$.

Now we show $(\sigma,\pi) \succeq (\tau,\pi)$. Consider $\ulambda\in\Lambda_n(G)$ with $(\tau,\pi) \to \ulambda$. We have to show $(\sigma,\pi) \to \ulambda$. Since $\sigma \unlhd \tau$, the Young diagram of $\tau$ can be obtained from the Young diagram of $\sigma$ by moving up boxes one at a time. Since the relation $\succeq$ is transitive, it is enough to consider a single box that is moved up, say from row $j$ to row $i$. Thus, we can assume that $\tau_i = \sigma_i+1$, $\tau_j = \sigma_j-1$ and $\tau_k = \sigma_k$ for every $k \neq i,j$.

According to Definition~\ref{def:arrow_order}, there is $\ubeta\in\Lambda_n(G)$ arising from a colouring of the boxes of $(\tau,\pi)$ such that $\ubeta(g)\unlhd\ulambda(g)$ for all $g\in G$. Consider the colouring of the boxes of $\tau$ in row $i$ and row $j$. Since $\sigma_i \geq \sigma_j$, we have $\tau_i \geq \tau_j + 2$. Thus, there exists an element $g \in G$ such that in row $i$ of $\tau$ there are more boxes coloured with $g$ than in row $j$ of $\tau$. Because the actual position of the colours in a row does not matter, we can assume without loss of generality that the boxes at the end of row $i$ are coloured with $g$.

Now define a colouring of $(\sigma,\pi)$ by copying the colouring of the boxes of $\tau$. In other words, the box that was moved from row $j$ to row $i$ to obtain $\tau$ from $\sigma$ is coloured with $g$ and all other boxes of $\sigma$ are coloured in the same way that they were coloured in $\tau$. All boxes of $\pi$ have to be coloured with $0$. This colouring of $(\sigma,\pi)$ gives $\underline{\gamma} \in \Lambda_n(G)$ with $\underline{\gamma}(h) = \ubeta(h)$ for every $h \neq g$.

Now consider $\ubeta(g)$ and $\underline{\gamma}(g)$. Let $x$ denote the number of times that $g$ appears in row $j$ of $\tau$ in the colouring giving rise to $\ubeta$. By the choice of $g$, we find that $g$ appears $x+k$ times in row $i$ of $\tau$ for some $k \geq 1$. Thus, $\ubeta(g)$ contains the parts $x$ and $x+k$. By construction of the colouring of $(\sigma,\pi)$, $\underline{\gamma}(g)$ contains the parts $x+1$ and $x+k-1$. Since all other parts of $\ubeta(g)$ and $\underline{\gamma}(g)$ are the same, we get $\underline{\gamma}(g) \unlhd \ubeta(g)$. Hence, we find that $\underline{\gamma}(g) \unlhd \ulambda(g)$ for every $g \in G$, proving $(\sigma,\pi) \to \ulambda$.
\end{proof}

Proposition~\ref{pro:necessary_conditions} and Lemma~\ref{lem:dominance rule} immediately give a characterisation of the relation~$\succeq$ in the case $|\sigma| = |\tau|$.

\begin{corollary}
Let $(\sigma,\rho)$ and $(\tau,\pi)$ be two double partitions of the same size. If $|\sigma| = |\tau|$ $($thus $|\rho| = |\pi|)$, then
\[
	(\sigma,\rho) \succeq (\tau,\pi) \;\Longleftrightarrow\; \sigma \unlhd \tau \;\text{and}\; \rho \unlhd \pi.
\]
\end{corollary}

If $\sigma$ is a partition and $k$ and $r$ are positive integers, we denote with $\sigma +_r k$ the partition obtained by adding $k$ to the part $\sigma_r$ of $\sigma$ and then possibly reordering the parts to obtain a partition. In other words, we add the sequence $0^{r-1}k$ to the sequence~$\sigma$ and reorder $\sigma$ such that it becomes a partition.

The next lemma describes a way to move some boxes from the partition $\sigma$ to the partition $\rho$.

\begin{lemma}[Sum rule]\label{lem:sum rule}
Let $(\sigma,\rho)$ be a double partition satisfying $\sigma_r\leq\rho_s$ for some positive integers $r$ and $s$. Then, for each non-negative integer $k$, we have
\[
(\sigma +_r k,\rho) \succeq (\sigma,\rho +_s k).
\]
\end{lemma}
\begin{proof}
Let $\ulambda\in\Lambda_n(G)$ with $(\sigma,\rho +_s k) \to \ulambda$. We have to show that $(\sigma+_rk,\rho) \to \ulambda$. According to Definition~\ref{def:arrow_order}, there is $\ubeta\in\Lambda_n(G)$ arising from a colouring $C$ of $(\sigma,\rho +_s k)$ such that $\ubeta(g)\unlhd\ulambda(g)$ for all $g\in G$.

Define a colouring $\widetilde{C}$ of $(\sigma+_r k,\rho)$ as follows. Colour all boxes of $\sigma+_r k$ that are also boxes of $\sigma$ in the same way they are coloured in the colouring $C$ of $(\sigma,\rho +_s k)$. Colour the $k$ remaining boxes (that is the boxes that were added to the part $\sigma_r$) and all boxes of $\rho$ with the element $0$ (the colouring of the boxes of~$\rho$ is forced). This colouring $\widetilde{C}$ of $(\sigma+_rk,\rho)$ gives $\underline{\gamma} \in \Lambda_n(G)$ with $\underline{\gamma}(g)=\ubeta(g)$ for all $g \neq 0$.

Note that in the part $\sigma_r + k$, there are $x+k$ boxes coloured with $0\in G$ for some integer~$x$ satisfying $0\le x\le\sigma_r$. Hence, $\ubeta(0)$ contains the parts $x$ and $\rho_s + k$ while $\underline{\gamma}(0)$ contains the parts $x+k$ and $\rho_s$. We now distinguish the order of these parts.

If $x+k\le \rho_s$, then $(\rho_s,x+k)$ is a partition satisfying $(\rho_s,x+k)\unlhd(\rho_s+k,x)$. If $x+k>\rho_s$, then $(x+k,\rho_s)$ is a partition satisfying $(x+k,\rho_s)\unlhd (\rho_s+k,x)$ since $x\le \sigma_r\le \rho_s$. In both cases, we find that $\underline{\gamma}(0)\unlhd\ubeta(0)\unlhd\ulambda(0)$ and so $\underline{\gamma}(0)\unlhd\ulambda(0)$. Hence, we find that $\underline{\gamma}(g) \unlhd \ulambda(g)$ for every $g \in G$, proving $(\sigma +_r k,\rho) \to \ulambda$.
\end{proof}

Figure~\ref{fig:sum rule} illustrates the sum rule.

\begin{figure}[h]
\begin{center}
\begin{tikzpicture}[scale=0.6]
\draw [thick] (0,0) rectangle (1,1);
\draw [thick] (0,1) grid (2,2);
\draw [thick,fill = black!30!white] (1,2) rectangle (4,3);
\draw [thick] (0,2) grid (4,3);

\draw node at (4.5,2) {\Large ,};

\draw [thick] (5,1) grid (7,2);
\draw [thick] (5,2) grid (7,3);

\draw node at (8,2) {\Large $\succeq$};

\draw [thick] (9,0) rectangle (10,1);
\draw [thick] (9,1) grid (10,2);
\draw [thick] (9,2) grid (11,3);

\draw node at (11.5,2) {\Large ,};

\draw [thick] (12,1) grid (14,2);
\draw [thick,fill = black!30!white] (14,2) rectangle (17,3);
\draw [thick] (12,2) grid (17,3);
\end{tikzpicture}
\end{center}
\vspace*{-0.5cm}
\caption{The sum rule}\label{fig:sum rule}
\end{figure}

We emphasise two special cases of the sum rule. The first rule says that we can move whole parts from the left entry of the double partition $(\sigma,\rho)$ to the right entry.
\begin{corollary}[Union rule]\label{cor:union_rule}
Let $(\sigma,\rho)$ be a double partition. Then for every partition $\alpha$ we have
\[
(\sigma\cup\alpha,\rho) \succeq (\sigma,\rho \cup \alpha).
\]
\end{corollary}

\begin{proof}
Since the relation $\succeq$ is transitive, it is enough to consider $\alpha = (k)$ for some integer $k$. Take $r$ and $s$ in the sum rule (Lemma~\ref{lem:sum rule}) such that $\sigma_r = \rho_s=0$. Then $\sigma +_r k = \sigma \cup (k)$ and $\rho +_s k = \rho \cup (k)$.
\end{proof}

The second special case states that we can exchange a part from the left entry of the double partition $(\sigma,\rho)$ with a part from the right entry provided that the part on the left is bigger than the part on the right.

\begin{corollary}[Exchange rule]\label{cor:exchange_rule}
Let $(\sigma,\rho)$ be a double partition and let $i,j$ be integers with $\sigma_i > \rho_j$. Define the double partition $(\widetilde{\sigma},\widetilde{\rho})$ as the double partition $(\sigma,\rho)$ but with the part $\sigma_i$ of $\sigma$ and the part $\rho_j$ of $\rho$ exchanged (then possibly reordering the parts to obtain a partition). Then we have $(\sigma,\rho) \succeq (\widetilde{\sigma},\widetilde{\rho})$.
\end{corollary}

\begin{proof}
Write $\sigma_i = \rho_j + (\sigma_i - \rho_j)$ and take $k=\sigma_i - \rho_j$ in the sum rule (Lemma~\ref{lem:sum rule}).
\end{proof}

It turns out that using the dominance rule and the sum rule, we can show that the necessary conditions from Proposition~\ref{pro:necessary_conditions} are actually sufficient for parabolic types. Hence, we can completely characterise the relation $\succeq$ for parabolic types.

Before we prove this, we define the union of a partition with a sequence of the form $(0^k,\ell-k)$ for some integers $0 \leq k \leq \ell$. This will be useful to neatly write down some conjugate partitions.

\begin{definition}\label{def:skew tableau union}
Let $\lambda$ be a partition and let $k,\ell \in \N$ with $0 \leq k \leq \ell$. Let $i \in \N$ be minimal such that $\lambda_i \leq k$ (such an $i$ always exists as $\lambda$ is an infinite sequence ending in zeroes). Then we define $\lambda \cup (0^k,\ell - k)$ to be the partition given by
\[
	\lambda \cup (0^k,\ell - k) = \begin{cases}
\lambda +_i (\ell -k), & \text{if $i=1$ or $ i\geq 2$ and $\ell \leq \lambda_{i-1}$ },\\
\lambda +_{i-1} (\ell - \lambda_{i-1}) +_i (\lambda_{i-1}-k), & \text{otherwise}.
\end{cases}
\]
\end{definition}

The operation of adding a part of the form $(0^k,\ell-k)$ to a partition corresponds to adding a box to the columns $k+1,k+2,\ldots,\ell$. If there is a gap between some boxes in a row, then the boxes are moved to the left to close the gap. Figure~\ref{fig:adding a part} gives two examples using the partitions $(111)\cup(0^2,3)=(411)$ and $(4311)\cup(0^2,3)=(5421)$. 
\begin{figure}[h]
\begin{center}
\begin{tikzpicture}[scale=0.6]
\draw [thick] (0,0) rectangle (1,1);
\draw [thick,fill = black!30!white] (2,-2) rectangle (5,-1);
\draw [thick] (0,1) grid (1,2);
\draw [dashed] (0,-2) grid (2,-1);
\draw [thick] (2,-2) grid (5,-1);
\draw [thick] (0,2) grid (1,3);

\draw [->,thick] (3.5,-0.5) -- (3.5,0.5);

\draw node at (6.5,1.5) {\Large $\rightsquigarrow$};

\draw [thick] (8,0) rectangle (9,1);
\draw [thick,fill = black!30!white] (10,2) rectangle (13,3);
\draw [thick] (8,1) grid (9,2);
\draw [thick] (8,2) grid (9,3);
\draw [thick] (10,2) grid (13,3);

\draw node at (14.5,1.5) {\Large $\rightsquigarrow$};

\draw [thick] (16,0) rectangle (17,1);
\draw [thick,fill = black!30!white] (17,2) rectangle (20,3);
\draw [thick] (16,1) grid (17,2);
\draw [thick] (16,2) grid (20,3);
\end{tikzpicture}
\\
\vspace*{1cm}
\begin{tikzpicture}[scale=0.6]
\draw [thick] (0,0) rectangle (1,1);
\draw [thick,fill = black!30!white] (2,-2) rectangle (5,-1);
\draw [thick] (0,1) grid (1,2);
\draw [dashed] (0,-2) grid (2,-1);
\draw [thick] (2,-2) grid (5,-1);
\draw [thick] (0,2) grid (3,3);
\draw [thick] (0,3) grid (4,4);

\draw [->,thick] (3.5,-0.5) -- (3.5,0.5);

\draw node at (6.5,2) {\Large $\rightsquigarrow$};

\draw [thick] (8,0) rectangle (9,1);
\draw [thick,fill = black!30!white] (12,3) rectangle (13,4);
\draw [thick,fill = black!30!white] (11,2) rectangle (12,3);
\draw [thick,fill = black!30!white] (10,1) rectangle (11,2);
\draw [thick] (8,1) grid (9,2);
\draw [thick] (8,2) grid (12,3);
\draw [thick] (8,3) grid (13,4);

\draw node at (14.5,2) {\Large $\rightsquigarrow$};

\draw [thick] (16,0) rectangle (17,1);
\draw [thick,fill = black!30!white] (19,3) rectangle (21,4);
\draw [thick] (16,1) grid (17,2);
\draw [thick,fill = black!30!white] (17,1) rectangle (18,2);
\draw [thick] (16,2) grid (20,3);
\draw [thick] (16,3) grid (21,4);
\end{tikzpicture}
\end{center}
\vspace*{-0.5cm}
\caption{The partitions $(111)\cup(0^2,3)$ and $(4311)\cup(0^2,3)$}\label{fig:adding a part}
\end{figure}

The reason for introducing this operation is the following lemma.

\begin{lemma}\label{lem:skew tableau union}
Let $n,k,\ell \in \N$ with $0 \leq k \leq \ell$ and let $\lambda \vdash n$ and $\mu \vdash n-(\ell-k)$ be partitions. Then we have
\[
	\lambda \unrhd \mu + 0^k 1^{\ell -k} \;\Longrightarrow \; \lambda' \unlhd \mu' \cup (0^k,\ell -k).
\]
\end{lemma}

\begin{proof}
Assume that $\lambda \unrhd \mu + 0^k 1^{\ell -k}$ and write $\widetilde{\mu} = \mu + 0^k 1^{\ell - k}$. In general, $\widetilde{\mu}$ is a composition. Notice that adding boxes to the rows of $\mu$ is equivalent to adding boxes to the columns of $\mu'$. Hence, if $\widetilde{\mu}$ is a partition, then we can simply take duals and deduce
\[
	\lambda' \unlhd \mu' \cup (0^k,\ell -k).
\]
We show that this is still true if $\widetilde{\mu}$ is not a partition. If $\widetilde{\mu}$ is not a partition, then there exists $i \in \N$ with $\widetilde{\mu}_i < \widetilde{\mu}_{i+1}$. We distinguish two cases.

If $\lambda_i < \widetilde{\mu}_{i+1}$, then we find $\lambda_{i+1} \leq \lambda_i < \widetilde{\mu}_{i+1}$. Because of $\lambda \unrhd \widetilde{\mu}$, we have
\[
	\lambda_1 + \ldots + \lambda_i \geq \widetilde{\mu}_1 + \ldots + \widetilde{\mu}_i.
\]
If this inequality was an equality, then by adding $\lambda_{i+1}$ to both sides it would follow that
\[
	\lambda_1 + \ldots + \lambda_i + \lambda_{i+1} = \widetilde{\mu}_1 + \ldots + \widetilde{\mu}_i + \lambda_{i+1} < \widetilde{\mu}_1 + \ldots + \widetilde{\mu}_i + \widetilde{\mu}_{i+1},
\]
contradicting $\lambda \unrhd \widetilde{\mu}$. Hence, we find $\lambda_1 + \ldots + \lambda_i \geq \widetilde{\mu}_1 + \ldots + \widetilde{\mu}_i + 1$. It follows that~$\lambda$ also dominates the composition $(\widetilde{\mu}_1,\ldots,\widetilde{\mu}_i+1,\widetilde{\mu}_{i+1}-1,\ldots)$ that is obtained from~$\widetilde{\mu}$ by moving a box from row $i+1$ to row $i$.

If $\lambda_i \geq \widetilde{\mu}_{i+1}$, then $\lambda_i > \widetilde{\mu}_i$, so $\lambda_i - \widetilde{\mu}_i >0$. Because of $\lambda \unrhd \widetilde{\mu}$, we find
\[
	\lambda_1 + \ldots + \lambda_i \geq \widetilde{\mu}_1 + \ldots + \widetilde{\mu}_i.
\]
If this inequality was an equality, then it would follow that
\[
	\lambda_1 + \ldots + \lambda_{i-1} < \lambda_1 + \ldots + \lambda_{i-1} + \lambda_i - \widetilde{\mu}_i = \widetilde{\mu}_1 + \ldots + \widetilde{\mu}_i - \widetilde{\mu}_{i} = \widetilde{\mu}_1 + \ldots + \widetilde{\mu}_{i-1},
\]
contradicting $\lambda \unrhd \widetilde{\mu}$. So we again deduce that $\lambda$ also dominates the composition $(\widetilde{\mu}_1,\ldots,\widetilde{\mu}_i+1,\widetilde{\mu}_{i+1}-1,\ldots)$ that is obtained from $\widetilde{\mu}$ by moving a box from row $i+1$ to row $i$.

As long as $\widetilde{\mu}$ is not a partition, so as long as we find an index $i$ with $\widetilde{\mu}_i < \widetilde{\mu}_{i+1}$, we can repeat the argument. Applied to the composition $\widetilde{\mu} = \mu + 0^k 1^{\ell -k}$, we find that if $\widetilde{\mu}$ is not a partition, then $\widetilde{\mu}_{k+1} = \widetilde{\mu}_k+1 > \widetilde{\mu}_k$. Hence, we can move the box at the end of row $k+1$ as far up in its column as it will go. We can do the same for all boxes in the column. This creates a partition.

It follows that if $\lambda \unrhd \mu + 0^k 1^{\ell -k}$, then $\lambda$ also dominates the partition that is obtained from $\mu + 0^k 1^{\ell -k}$ by moving all boxes in a column as far up as possible. In the dual, this corresponds to moving all boxes in a row as far left as possible. This agrees exactly with Definition \ref{def:skew tableau union}, so we obtain
\[
	\lambda' \unlhd \mu' \cup (0^k,\ell -k).\qedhere
\]
\end{proof}

We can now characterise the relation $\succeq$ for parabolic types.

\begin{theorem}
\label{thm:designs_parabolic}
Let $(\sigma,k)$ and $(\tau,\ell)$ be two double partitions of the same size. Then
\[
(\sigma,k)\succeq(\tau,\ell)\;\Longleftrightarrow\;\text{$k\le \ell$ and $(\sigma\cup (k))\unlhd (\tau\cup(\ell))$}.
\]
\end{theorem}

\begin{proof}
The forward direction follows from Proposition~\ref{pro:necessary_conditions}.

For the reverse direction, assume that $k\le \ell$ and $(\sigma\cup (k))\unlhd (\tau\cup(\ell))$. Taking duals, it follows that $\sigma'+1^k\unrhd\tau'+1^\ell$ and therefore
\[ 
\sigma'\unrhd\tau'+0^k1^{\ell-k}.
\]
Hence, from Lemma~\ref{lem:skew tableau union} we find
\[
\sigma\unlhd\tau\cup (0^k,\ell-k),
\]
so we deduce that $(\sigma,k) \succeq (\tau \cup (0^k,\ell -k),k)$ using the dominance rule (Lemma~\ref{lem:dominance rule}). We now distinguish the two cases in Definition~\ref{def:skew tableau union}.

If $\tau\cup (0^k,\ell-k) = \tau +_i (\ell -k)$ for some $i$, then $\tau_i \leq k$ and we find
\[
	(\sigma,k) \succeq (\tau +_i (\ell -k),k) \succeq (\tau,k + (\ell -k)) = (\tau,\ell)
\]
using the sum rule (Lemma~\ref{lem:sum rule}).

If $\tau\cup (0^k,\ell-k) = \tau +_{i-1} (\ell - \tau_{i-1}) +_i (\tau_{i-1}-k)$ for some $i$, then again $\tau_i \leq k$ and we find
\[
	(\sigma,k) \succeq (\tau +_{i-1} (\ell - \tau_{i-1}) +_i (\tau_{i-1}-k),k) \succeq (\tau +_{i-1} (\ell - \tau_{i-1}),\tau_{i-1}) = (\tau,\ell)
\]
using the sum rule (Lemma~\ref{lem:sum rule}) twice.
\end{proof}

This theorem gives an easy criterion to compare parabolic transitivity types in the group $W$. It is a generalisation of the Livingstone--Wagner theorem (see Theorem~\ref{thm:livingstone wagner}) from $S_n$ to the group $W$. Moreover, in the case $G = C_1$, the parabolic types become the double partitions $(\sigma,\emptyset)$, so we have $k = 0$ and Theorem~\ref{thm:designs_parabolic} specialises to Theorem~5 in \cite{MarSag2007}, the generalisation of the Livingstone--Wagner theorem by Martin and Sagan for the symmetric group~$S_n$.

Let us apply Theorem~\ref{thm:designs_parabolic} to $W = C_2 \wr S_n$, the symmetry group of the $n$-dimensional cube.

\begin{example}\label{ex:lw_for_cube}
Let $W = C_2 \wr S_n$ be the hyperoctahedral group, also denoted as $B_n$. Then all types $\usigma$ are simple.

First, let $n=3$ and consider the parabolic types $\usigma = (2,1)$ and $\utau = (1,2)$. They correspond to the subgroups $H_{\usigma} = S_2 \times B_1$ and $H_{\utau} = S_1 \times B_2$. Geometrically, the double partition $(2,1)$ stands for edge-transitivity and the double partition $(1,2)$ stands for face-transitivity. Applying Theorem~\ref{thm:designs_parabolic} gives $(2,1) \succeq (1,2)$, that is every $(2,1)$-transitive set is also $(1,2)$-transitive. Geometrically, this means that every subset of $C_2 \wr S_3$ that is transitive on the edges of a $3$-dimensional cube is also transitive on the faces of a $3$-dimensional cube.

We can generalise this observation to the $n$-dimensional cube. The stabiliser of a $t$-dimensional face of the cube is $S_{n-t} \times B_t$. Thus, the double partition $(n-t,t)$ describes transitivity on the $t$-dimensional faces of the $n$-dimensional cube. Theorem~\ref{thm:designs_parabolic} gives
\[
	(n-t,t) \succeq (n-s,s) \; \Longleftrightarrow\; t \leq s \text{ and } ((n-t) \cup (t)) \unlhd ((n-s) \cup (s)).
\]
We distinguish the cases $t \geq n-t$ and $t < n-t$.

If $t \geq n-t$, then Theorem~\ref{thm:designs_parabolic} gives $(n-t,t) \succeq (n-s,s)$ for every $s \geq t$ and $(n-t,t) \not\succeq (n-s,s)$ for every $s < t$. Thus, if $n/2 \leq t \leq n-1$, then a set that acts transitively on the $t$-dimensional faces of an $n$-dimensional cube also acts transitively on the $(t+1)$-dimensional faces.

If $t < n-t$, then Theorem~\ref{thm:designs_parabolic} gives $(n-t,t) \succeq (t,n-t)$ and $(n-t,t) \not \succeq (n-s,s)$ for every $s$ with $0 \leq s < n-t$ except $s = t$. Thus, if $0 \leq t < n/2$, then transitivity on the $t$-dimensional faces of an $n$-dimensional cube implies transitivity on the $(n-t)$-dimensional faces and we are in the first case again.
\end{example}

We summarise our findings in the following theorem.

\begin{theorem}
Let $Y \subseteq C_2 \wr S_n$ be a non-empty subset.
\begin{enumerate}[label=\textup{(}\alph*\textup{)}]
\item{
If $t$ is an integer with $n/2 \leq t < n$ and $Y$ is transitive on the faces of dimension~$t$ of an $n$-dimensional cube, then $Y$ is transitive on the faces of dimension $t+1$.
}
\item{
If $t$ is an integer with $1 \leq t < n/2$ and $Y$ is transitive on the faces of dimension~$t$ of an $n$-dimensional cube, then $Y$ is transitive on the faces of dimension $n-t$.
}
\end{enumerate}
\end{theorem}

Similar results hold for complex regular polytopes, that is the cases $W = C_r \wr S_n$ with $r > 2$ in Theorem~\ref{thm:designs_parabolic}.

We strongly believe that the dominance rule and the sum rule characterise the relation $\succeq$ completely. That is, we make the following conjecture.

\begin{conjecture}\label{conj:small steps}
Let $(\sigma,\rho)$ and $(\tau,\pi)$ be two double partitions of the same size. Then the following are equivalent.
\begin{enumerate}[label=\textup{(}\alph*\textup{)}]
\item{
$(\sigma,\rho) \succeq (\tau,\pi)$
}
\item{
There exists a sequence $(\sigma_i,\rho_i)$ with 
\[
	(\sigma,\rho) = (\sigma_1,\rho_1) \succeq (\sigma_2,\rho_2) \succeq \ldots \succeq (\sigma_k,\rho_k) = (\tau,\pi)
\]
such that every step $(\sigma_i,\rho_i) \succeq (\sigma_{i+1},\rho_{i+1})$ for $i = 1,\ldots, k-1$ follows from Lemma~\ref{lem:dominance rule} (dominance rule) or Lemma~\ref{lem:sum rule} (sum rule).
}
\end{enumerate}
\end{conjecture}

We have verified this conjecture for the hyperoctahedral group $C_2 \wr S_n$ for every $n \leq 25$ for every double partition of size $n$. Note that the implication $(b)~\Longrightarrow ~(a)$ follows because the relation $\succeq$ is transitive.

We also have computational evidence that the following criterion characterises the relation $\succeq$.

\begin{conjecture}\label{conj:equivalent conditions}
Let $(\sigma,\rho)$ and $(\tau,\pi)$ be two double partitions of the same size. Then the following are equivalent.
\begin{enumerate}[label=\textup{(}\alph*\textup{)}]
\item{
$(\sigma,\rho) \succeq (\tau,\pi)$
}
\item{
There exists a partition $\widetilde{\rho}$ with $|\widetilde{\rho}| = |\rho|$ and $\widetilde{\rho} \unrhd \rho$ such that 
\[
	\widetilde{\rho} \subseteq \pi \text{ and } \sigma \cup \widetilde{\rho} \unlhd \tau \cup \pi.
\]
}
\end{enumerate}
\end{conjecture}

Notice that the condition in (b) is similar to the necessary conditions of Proposition~\ref{pro:necessary_conditions}. If $(\sigma,\rho) \succeq (\tau,\pi)$, then it follows that $\rho \unlhd \pi$, which is equivalent to the existence of a partition $\widetilde{\rho}$ with $|\widetilde{\rho}| = |\rho|$ and $\widetilde{\rho} \unrhd \rho$ such that $\widetilde{\rho}\subseteq \pi$. The conjecture claims that there is such a partition $\widetilde{\rho}$ that satisfies $\sigma \cup \widetilde{\rho} \unlhd \tau \cup \pi$ while Proposition~\ref{pro:necessary_conditions} only asserts that $\sigma \cup \rho \unlhd \tau \cup \pi$. Notice that the stronger condition rules out the counterexample in Example~\ref{ex:not sufficient} that shows that the conditions of Proposition~\ref{pro:necessary_conditions} are not sufficient. Furthermore, notice that Theorem~\ref{thm:designs_parabolic} is a special case of Conjecture~\ref{conj:equivalent conditions} because $k \leq \ell \Longleftrightarrow (k) \subseteq (\ell)$.

There is a relation between Conjecture~\ref{conj:small steps} and Conjecture~\ref{conj:equivalent conditions}. We can think of the condition in Conjecture~\ref{conj:equivalent conditions} as first using the dominance rule for the double partitions $(\sigma,\rho)$ and $(\sigma,\widetilde{\rho})$ and then using the sum rule. It is not hard to show that if $(\sigma,\widetilde{\rho}) \succeq (\tau,\pi)$ follows from a sequence of applications of the sum rule, then the conditions of Conjecture~\ref{conj:equivalent conditions} are satisfied. Furthermore, through careful case analysis, it can be shown that a chain
\[
	(\sigma,\rho) = (\sigma_1,\rho_1) \succeq (\sigma_2,\rho_2) \succeq \ldots \succeq (\sigma_k,\rho_k) = (\tau,\pi)
\]
where each step can be explained through the dominance rule and the sum rule can be transformed into a chain where the first steps follow from the dominance rule and the remaining steps follow from the sum rule. Hence, we get a condition very similar to the condition in Conjecture~\ref{conj:equivalent conditions}. We expect that Definition~\ref{def:skew tableau union} can be generalised to skew tableaux, giving a way to prove the conjecture.

\begin{remark}\label{rem:general types}
The results of this section can be generalised to non-simple types $\usigma \in \Sigma_n(G)$. The partitions in $\usigma$ are indexed by the subgroups of $G$. It is not hard to see that suitable generalisations of the dominance rule (Lemma~\ref{lem:dominance rule}) and the sum rule (Lemma~\ref{lem:sum rule}) hold in this case.

First, it is clear that the dominance rule holds for every component $\usigma(U)$ of~$\usigma$ independently as the proof of Lemma~\ref{lem:dominance rule} works for a colouring coming from an arbitrary group $G$.

For the sum rule, consider two subgroups $U,V \leq G$ such that $U \subseteq V$ (equivalently $U^\circ \supseteq V^\circ$). Then, we can look at the double partition $(\usigma(U),\usigma(V))$ of two entries of $\usigma$. For this double partition, the sum rule holds, that is, we can move some boxes from the partition $\usigma(U)$ to the partition $\usigma(V)$. The reason is that the condition $U \subseteq V$ ensures that the colouring can be copied as in the proof of Lemma~\ref{lem:sum rule}. Thus, the sum rule holds for any two entries of $\usigma$ indexed by $U$ and $V$ provided that $U \subseteq V$. It follows that the union rule and the exchange rule also hold because they are special cases of the sum rule.

Finally, the property from Proposition~\ref{pro:necessary_conditions} (c) generalises to non-simple types by taking the union of all partitions $\usigma(U)$ for $U \leq G$.
\end{remark}

\section{Designs, codes and orthogonal polynomials}
\label{sec:polys_and_designs}

Certain association schemes, namely $P$- and $Q$-polynomial schemes, are closely related to orthogonal polynomials in the sense that their character tables arise as evaluations of such polynomials (see \cite{BanIto1984} or \cite{Del1973}). The conjugacy class scheme of $C_r \wr S_n$ does not have these properties. Nevertheless, there is still a relationship to certain orthogonal polynomials, namely the Charlier polynomials. We show that there are characters associated with the Charlier polynomials which have a nice decomposition into irreducible characters. These characters and their decomposition can then be used to characterise codes and designs in the group $C_r \wr S_n$.

\subsection{The branching rule}

We first recall the well-know \emph{branching rule} that will be useful in this section. It gives rise to much of the combinatorial structure of the representation theory of the symmetric group.
\begin{definition}
Let $n \in \N$ and $\lambda \vdash n$. The set $\lambda^-$ is the set of all partitions of $n-1$ that can be obtained from $\lambda$ by decreasing a part by $1$. Similarly, for $n \in \N_0$ and $\lambda \vdash n$, the set $\lambda^+$ is the set of all partitions of $n+1$ that can be obtained from~$\lambda$ by increasing a part by $1$ or by adding a part of size $1$ to $\lambda$.
\end{definition}

The branching rule can now be stated in terms of the sets $\lambda^-$ and $\lambda^+$.

\begin{theorem}\label{thm:branching_symmetric_group}
Let $n \in \N$ and $\lambda \vdash n$. Then we have
\begin{enumerate}[label=\textup{(}\alph*\textup{)}]
\item{
	$\chi^\lambda \downarrow^{S_n}_{S_{n-1}} = \sum\limits_{\mu \in \lambda^-} \chi^\mu$\,,
}
\item{
	$\chi^\lambda \uparrow^{S_{n+1}}_{S_n} = \sum\limits_{\mu \in \lambda^+} \chi^\mu$\,.
}
\end{enumerate}
\end{theorem}
For a proof of the formula, see \cite[Theorem 2.8.3]{Sag2001} for example.

There is a variant of the branching rule for wreath products of the form $G \wr S_n$ (see \cite[Theorem 10]{Pus99}). We adapt the statement to groups of the form $C_r \wr S_n$. First, we give the appropriate analogues of the sets $\lambda^-$ and $\lambda^+$ for partition tuples.

\begin{definition}
Let $n \in \N$ and $\ulambda \in \Lambda_n(C_r)$. The set $\ulambda^-$ is the set of all partition tuples $\umu \in \Lambda_{n-1}(C_r)$ that can be obtained from $\ulambda$ by decreasing a part of a partition $\ulambda(g)$ for some $g \in C_r$ by $1$. Similarly, for $n \in \N_0$ and $\ulambda \in \Lambda_n(C_r)$, the set $\ulambda^+$ is the set of all partition tuples $\umu \in \Lambda_{n+1}(C_r)$ that can be obtained from~$\ulambda$ by increasing a part of a partition $\ulambda(g)$ by $1$ or by adding a part of size~$1$ to $\ulambda(g)$ for some $g \in C_r$.
\end{definition}

In other words, $\ulambda^-$ (resp. $\ulambda^+$) is the set of all partition tuples that can be obtained from $\ulambda$ by replacing precisely one entry $\ulambda(g)$ with an element of $\ulambda(g)^-$ (resp. $\ulambda(g)^+$).

\begin{example}
Let $r=3$. We identify the elements of $C_3$ with the set $\{0,1,2\}$ and write an element $\ulambda \in \Lambda_n(C_3)$ as a tuple $\ulambda = \big(\ulambda(0),\ulambda(1),\ulambda(2)\big)$. For the partition tuple $(2,\emptyset,1) \in \Lambda_3(C_3)$, we have that
\begin{align*}
	(2,\emptyset,1)^- &= \{(1,\emptyset,1),(2,\emptyset,\emptyset)\},\\
	(2,\emptyset,1)^+ &= \{(3,\emptyset,1),(21,\emptyset,1),(2,1,1),(2,\emptyset,2),(2,\emptyset,11)\}.
\end{align*}
\end{example}

We can now state the branching rule for the group $C_r \wr S_n$.

\begin{theorem}\label{thm:branching_generalised_symmetric_group}
Let $r,n \in \N$ and $\ulambda \in \Lambda_n(C_r)$. Then we have
\begin{enumerate}[label=\textup{(}\alph*\textup{)}]
\item{
	$\chi^\ulambda \downarrow^{C_r \wr S_n}_{C_r \wr S_{n-1}} = \sum\limits_{\umu \in \ulambda^-} \chi^{\umu}$\,,
}
\item{
	$\chi^\ulambda \uparrow^{C_r \wr S_{n+1}}_{C_r \wr S_n} = \sum\limits_{\umu \in \ulambda^+} \chi^{\umu}$\,.
}
\end{enumerate}
\end{theorem}

For a proof, see \cite[Theorem 10]{Pus99}. We remark that a proof can also be obtained from our proof of Theorem~\ref{thm:decomposition_perm_char} by going through the proof for the case $H = C_r \wr S_{n-1}$ in more detail. Further, note that Theorem~\ref{thm:branching_generalised_symmetric_group} reduces to Theorem~\ref{thm:branching_symmetric_group} in the case $r=1$.

A well-known result that will be useful together with the branching rule is the following. It connects induction and restriction of characters.

\begin{theorem}\label{thm:induction_of_product_and_restriction}
Let $G$ be a finite group and let $H$ be a subgroup of $G$. Furthermore, let $\chi$ be a character of $H$ and let $\psi$ be a character of $G$. Then we have
\[
	\left(\chi \uparrow_H^G\right) \cdot \psi = \left( \chi \cdot \psi \downarrow_H^G \right)\uparrow_H^G.
\]
\end{theorem}

For a proof, see \cite[Chapter 7.2]{Ser1977}.

\subsection{Charlier Polynomials}
The \emph{Charlier polynomial} of degree $k$ with parameter $a\in\R$, $a > 0$, is given by
\[
C_k^{(a)}(x)=\sum_{j=0}^k(-1)^{k-j}\binom{k}{j}a^{-j}(x)_j,
\]
where $(x)_j=x(x-1)\cdots(x-j+1)$ is the falling factorial. Some properties of these polynomials can be found in \cite[Chapter 1.12]{KoeSwa1998}. The first polynomials are
\begin{align*}
C_0^{(a)}(x) &= 1,\\
C_1^{(a)}(x) &= a^{-1}x-1,\\
C_2^{(a)}(x) &= a^{-2}x^2-(2a^{-1}+a^{-2})x+1.
\end{align*}

In fact, $C^{(a)}_1,C^{(a)}_2,\dots$ form a system of orthogonal polynomials with respect to the Poisson distribution with mean $a$. That is, we have
\begin{equation}\label{eqn:def_charlier}
 \sum_{i=0}^\infty \frac{a^i}{i!}\,C^{(a)}_k(i)C^{(a)}_\ell(i)=0\quad\text{for $k\ne\ell$}.
\end{equation}
Sometimes the Charlier polynomials are defined to be the polynomials $a^k C_k^{(a)}$ or $(-1)^k C_k^{(a)}$. However, all of them satisfy the orthogonality relation \eqref{eqn:def_charlier}.

It is well-known that the monic Charlier polynomials $p^{(a)}_k$ satisfy the recurrence
\[
	xp_k(x) = p_{k+1}(x) +(k+a)p_k(x) + kap_{k-1}(x)
\]
for every $k \in \N$ where $p^{(a)}_0(x) = 1$ and $p^{(a)}_1(x) = x-a$. Since the leading coefficient of $C^{(a)}_k$ is $a^{-k}$, it follows that the Charlier polynomials $C^{(a)}_k$ in the way we defined them satisfy the recurrence
\begin{equation}\label{eqn:Charlier_recurrence}
xC^{(a)}_k(x) = aC^{(a)}_{k+1}(x) +(k+a)C^{(a)}_k(x) + kC^{(a)}_{k-1}(x)
\end{equation}
for every $k \in \N$.

For $g\in C_r\wr S_n$, let 
\[
\theta(g)=\frac{1}{r}\,\left|\left\{(c,i)\in[r]\times [n] : g(c,i)=(c,i)\right\}\right|.
\]

That is, $\theta(g)$ is the number of fixed points of the natural action of $C_r \wr S_n$ on $[r]\times [n]$, scaled by $1/r$. Then $\theta$ is a class function of $C_r\wr S_n$ taking values in $\{0,1,\dots,n\}$ (note that $g(c,i) = (c,i)$ for some $c \in [r]$ and $i \in [n]$ implies $g(\widetilde{c},i) = (\widetilde{c},i)$ for every $\widetilde{c} \in [r]$). We write $w_i$ for the number of elements $g \in C_r \wr S_n$ satisfying $\theta(g) = i$. Using the inclusion-exclusion principle, one can derive the expression
\[
w_i=\frac{n!\,r^{n-i}}{i!}\sum_{j=0}^{n-i}\frac{(-1)^j}{j!\,r^j},
\]
see for example \cite[Section 7]{ChoMan2012}. We shall later see that this expression also follows from our results (see Remark~\ref{rem:Poisson_moments}).

The function $\theta$ is a class function and it defines a uniformly distributed discrete random variable on $C_r \wr S_n$. It was shown in \cite[Prop.~7.4]{ChoMan2012} that, for $k\le n$, the $k$-th moment of~$\theta$ equals the $k$-th moment of the Poisson distribution with mean $1/r$ (the condition $k\le n$ was erroneously omitted in \cite[Prop.~7.4]{ChoMan2012}). The $k$-th moment of the Poisson distribution is given by the expression
\[
	e^{-\frac{1}{r}}\sum_{i=0}^\infty \frac{\left(\frac{1}{r}\right)^{i}}{i!} \cdot i^k .
\]
Hence, we obtain
\begin{equation}\label{eqn:moments}
	\mathbb{E}[\theta^k] = \sum_{g \in G} \frac{1}{|G|}\cdot\theta(g)^k = \frac{1}{|G|} \sum_{i=0}^n w_i \cdot i^k = e^{-\frac{1}{r}} \sum_{i=0}^\infty \frac{\left(\frac{1}{r}\right)^{i}}{i!} \cdot i^k
\end{equation}
for every $k \leq n$. We mention in passing that the $k$-th moment of the Poisson distribution with mean $1/r$ is given by
\[
\sum_{i=1}^k\left(\frac{1}{r}\right)^iS(k,i),
\]
where $S(k,i)$ are the Stirling numbers of the second kind, counting the number of partitions of a set of size $k$ into $i$ non-empty subsets.

As $C_k^{(1/r)}C_l^{(1/r)}$ is a polynomial of degree $k+l$, we find from \eqref{eqn:def_charlier} and \eqref{eqn:moments} that the Charlier polynomials also satisfy the orthogonality relation
\begin{equation}
\sum_{i=0}^nw_i\,C^{(1/r)}_k(i)C^{(1/r)}_\ell(i)=0\quad\text{for $k\ne\ell$ and $k+\ell\le n$}.   \label{eqn:Charlier_orthogonal}
\end{equation}

In this section, we frequently use the well-known \emph{binomial transform}. For a sequence $(a_k)_{k \in \N_0}$, we have
\begin{equation}\label{eqn:binomial_inversion}
	s_n = \sum_{k=0}^n \binom{n}{k} a_k \; \Longleftrightarrow \; a_k = \sum_{n=0}^k (-1)^{k-n}\binom{k}{n} s_n,
\end{equation}
also called \emph{binomial inversion}, which can be obtained from the binomial theorem by straightforward calculations. Furthermore, if $t \in \N_0$ is an integer and we consider the sequence $(s_n)_{n =0,\ldots,t}$ defined by
\[
	s_n = \sum_{k=0}^t \binom{k}{n} a_k,
\]
then a variant of the binomial transform is given by
\begin{equation}\label{eqn:binomial_inversion_variant}
	s_n = \sum_{k=0}^t \binom{k}{n} a_k \; \Longleftrightarrow \; a_k = \sum_{n=0}^t (-1)^{n-k}\binom{n}{k} s_n.
\end{equation}

With every polynomial $f(x)=f_nx^n+\cdots+f_1x+f_0$ in $\R[x]$, we associate the class function $f(\theta)=f_n\theta^n+\cdots+f_1\theta+f_0 1_{C_r \wr S_n}$. This induces an algebra homomorphism from~$\R[x]$ to the set of class functions of $C_r\wr S_n$.

Let~$\xi_j$ be the permutation character on ordered $j$-tuples of pairs $(c_k,i_k) \in [r]\times [n]$ with all $i_k$ pairwise distinct. By convention, $\xi_0$ is the trivial character of $C_r\wr S_n$. Note that
\[
\xi_j = r^j \theta(\theta - 1) \ldots (\theta -j +1)=r^j\,(\theta)_j.
\]
Hence, we have
\begin{equation}
C^{(1/r)}_k(\theta)=\sum_{j=0}^k(-1)^{k-j}{\binom{k}{j}}\xi_j\quad\text{for $k=0,1,\dots,n$}, \label{eqn:C_from_xi}
\end{equation}
and from~\eqref{eqn:binomial_inversion}, we obtain
\begin{equation}
\xi_j=\sum_{k=0}^j{\binom{j}{k}}C^{(1/r)}_k(\theta)\quad\text{for $j=0,1,\dots,n$}.   \label{eqn:xi_from_C}
\end{equation}
For $0\le k\le n/2$, we now determine the irreducible constituents of $C^{(1/r)}_k(\theta)$. Recall that for two elements $\ulambda \in \Lambda_\ell(G)$ and $\umu \in \Lambda_m(G)$, their union $\ulambda \cup \umu \in \Lambda_{\ell+m}(G)$ is defined as $(\ulambda \cup \umu)(g) = \ulambda(g) \cup \umu(g)$ for every $g \in G$.
\begin{theorem}
\label{thm:decomposition_Charlier_chars}
For each $k$ satisfying $0\le k\le n/2$, we have
\[
	C^{(1/r)}_k(\theta) = \sum_{\ulambda \in \Lambda_k(C_r)} \deg(\chi^\ulambda) \, \chi^{(n-k,\emptyset,\ldots,\emptyset) \cup \ulambda}\,.
\]
In particular, $C^{(1/r)}_k(\theta)$ is a character of $C_r\wr S_n$, and for $\ulambda \in \Lambda_n(C_r)$ we have
\[
\bigang{C^{(1/r)}_k(\theta),\chi^{\ulambda}}\ne 0\;\Longleftrightarrow\; \ulambda(0)_1=n-k.
\]
\end{theorem}
\begin{proof}
From \eqref{eqn:Charlier_recurrence}, we find that for $a = 1/r$ we have the recurrence
\[
rxC^{(1/r)}_k(x) = C^{(1/r)}_{k+1}(x) +(rk+1)C^{(1/r)}_k(x) + rkC^{(1/r)}_{k-1}(x)
\]
for every $k \in \N$. The Charlier polynomials are uniquely defined by this recurrence and the first two polynomials $C^{(1/r)}_0(x) = 1$ and $C^{(1/r)}_1(x) = rx-1$. Letting $x = \theta$, we find the recurrence
\[
r\theta C^{(1/r)}_k(\theta) = C^{(1/r)}_{k+1}(\theta) +(rk+1)C^{(1/r)}_k(\theta) + rkC^{(1/r)}_{k-1}(\theta).
\]
for every $k \in \N$. This is an identity of class functions.

Let
\[
	F_k^{(r)} = \sum_{\ulambda \in \Lambda_k(C_r)} \deg(\chi^\ulambda) \,\chi^{(n-k,\emptyset,\ldots,\emptyset) \cup \ulambda}
\]
denote the proposed expression for $C^{(1/r)}_k(\theta)$. We show $C^{(1/r)}_k(\theta) = F_k^{(r)}$ for $k \leq n/2$ by showing that $C^{(1/r)}_0(\theta) = F_0^{(r)}$ and $C^{(1/r)}_1(\theta) = F_1^{(r)}$ and then showing that $F_k^{(r)}$ satisfies the recurrence
\begin{equation}\label{eqn:recurrence_for_Fk}
r\theta F_k^{(r)} = F_{k+1}^{(r)}(\theta) +(rk+1)F_k^{(r)} + rkF_{k-1}^{(r)}
\end{equation}
for every $k \in \N$ with $1 \leq k \leq (n-2)/2$.

First, notice that
\begin{align*}
	C^{(1/r)}_0(\theta) &= 1_{C_r \wr S_n} = \chi^{(n,\emptyset,\ldots,\emptyset)} = F_0^{(r)},\\
	C^{(1/r)}_1(\theta) &= r\theta - 1_{C_r \wr S_n} = \xi_1 - \chi^{(n,\emptyset,\ldots,\emptyset)}.
\end{align*}
Notice that $\xi_1$ is the fixed point character of $C_r \wr S_n$ in its natural action. The stabiliser of a point is $C_r \wr S_{n-1}$, so $\xi_1 = 1\uparrow_{C_r \wr S_{n-1}}^{C_r \wr S_n}$. The trivial character of $C_r \wr S_{n-1}$ is indexed by the partition tuple $(n-1,\emptyset,\ldots,\emptyset)$. Hence, we find from the branching rule (Theorem~\ref{thm:branching_generalised_symmetric_group}) that
\[
	\xi_1 = 1 \uparrow^{C_r \wr S_{n}}_{C_r \wr S_{n-1}} = \sum\limits_{\umu \in \ulambda^+} \chi^{\umu}
\]
where $\ulambda = (n-1,\emptyset,\ldots,\emptyset)$. The partition tuples in the set $(n-1,\emptyset,\ldots,\emptyset)^+$ are easily seen to be
\begin{align*}
	(n-1,\emptyset,\ldots,\emptyset)^+ =& \{(n,\emptyset,\ldots,\emptyset),((n-1,1),\emptyset,\ldots,\emptyset),\\
		&\phantom{\}}(n-1,1,\emptyset,\ldots,\emptyset),\ldots,(n-1,\emptyset,\ldots,\emptyset,1).
\end{align*}
These are precisely the partition tuples of the form $(n-1,\emptyset,\ldots,\emptyset) \cup \umu$ where $\umu \in \Lambda_1(C_r)$ together with the partition tuple $(n,\emptyset,\ldots,\emptyset)$ which indexes the trivial character. As $C_r \wr S_1 \cong C_r$, we find that all characters $\chi^\umu$ with $\umu \in \Lambda_1(C_r)$ are $1$-dimensional. Hence, we find that
\[
	C^{(1/r)}_1(\theta) = \xi_1 - \chi^{(n,\emptyset,\ldots,\emptyset)} = \sum_{\umu \in \Lambda_1(C_r)} \deg(\chi^\umu) \,\chi^{(n-1,\emptyset,\ldots,\emptyset) \cup \umu} = F_1^{(r)}.
\]
Now let $k \in \N$ with $1 \leq k \leq (n-2)/2$. Since $r\theta = \xi_1 = 1\uparrow_{C_r \wr S_{n-1}}^{C_r \wr S_n}$, we first investigate the product $\xi_1 \chi^{(n-k,\emptyset,\ldots,\emptyset) \cup \ulambda}$ for some $\ulambda \in \Lambda_k(C_r)$. From Theorem~\ref{thm:induction_of_product_and_restriction}, we find
\begin{align*}
	\xi_1 \chi^{(n-k,\emptyset,\ldots,\emptyset) \cup \ulambda} =& \left(1\uparrow_{C_r \wr S_{n-1}}^{C_r \wr S_n}\right) \chi^{(n-k,\emptyset,\ldots,\emptyset) \cup \ulambda} \\
	=& \left(1 \cdot \chi^{(n-k,\emptyset,\ldots,\emptyset) \cup \ulambda}\downarrow_{C_r \wr S_{n-1}}^{C_r \wr S_n}\right)\uparrow_{C_r \wr S_{n-1}}^{C_r \wr S_n} \\
	=& \left(\chi^{(n-k,\emptyset,\ldots,\emptyset) \cup \ulambda}\downarrow_{C_r \wr S_{n-1}}^{C_r \wr S_n}\right)\uparrow_{C_r \wr S_{n-1}}^{C_r \wr S_n}.
\end{align*}
The restriction of the character $\chi^{(n-k,\emptyset,\ldots,\emptyset) \cup \ulambda}$ to $C_r \wr S_{n-1}$ can be computed using the branching rule (Theorem~\ref{thm:branching_generalised_symmetric_group}). Notice that because we assume $k \leq (n-2)/2$, it follows that $k < n-k$. Hence, the partition tuple $\ulambda$ cannot contain a part of size $n-k$. It follows that when using the branching rule for the partition tuple $(n-k,\emptyset,\ldots,\emptyset) \cup \ulambda$, we may decrease the part of size $n-k$ and the parts of $\ulambda$ independently. Thus, we obtain
\[
	\chi^{(n-k,\emptyset,\ldots,\emptyset) \cup \ulambda}\downarrow_{C_r \wr S_{n-1}}^{C_r \wr S_n} = \chi^{(n-k-1,\emptyset,\ldots,\emptyset) \cup \ulambda} + \sum_{\umu \in \ulambda^-}\chi^{(n-k,\emptyset,\ldots,\emptyset) \cup \umu}.
\]
Similarly, we can compute the induction of the characters using the branching rule. Notice that we have $k < n-k-1$. It follows that when using the branching rule for the partition tuple $(n-k-1,\emptyset,\ldots,\emptyset) \cup \ulambda$, we may increase the part of size $n-k-1$ and the parts of $\ulambda$ independently. Thus, we obtain
\[
	\chi^{(n-k-1,\emptyset,\ldots,\emptyset) \cup \ulambda}\uparrow_{C_r \wr S_{n-1}}^{C_r \wr S_n} = \chi^{(n-k,\emptyset,\ldots,\emptyset) \cup \ulambda} + \sum_{\umu \in \ulambda^+}\chi^{(n-k-1,\emptyset,\ldots,\emptyset) \cup \umu}
\]
and
\begin{align*}
	&\left(\sum_{\umu \in \ulambda^-}\chi^{(n-k,\emptyset,\ldots,\emptyset) \cup \umu}\right)\uparrow_{C_r \wr S_{n-1}}^{C_r \wr S_n}\\
	=&\sum_{\umu \in \ulambda^-}\left(\chi^{(n-k,\emptyset,\ldots,\emptyset) \cup \umu}\right)\uparrow_{C_r \wr S_{n-1}}^{C_r \wr S_n}\\
	=&\sum_{\umu \in \ulambda^-}\chi^{(n-k+1,\emptyset,\ldots,\emptyset) \cup \umu} + \sum_{\umu \in \ulambda^-}\sum_{\underline{\tau} \in \umu^+}\chi^{(n-k,\emptyset,\ldots,\emptyset) \cup \underline{\tau}}.
\end{align*}
Putting it all together, we find
\begin{align*}
	&\phantom{+}\;\xi_1 \chi^{(n-k,\emptyset,\ldots,\emptyset) \cup \ulambda} = \;\left(\chi^{(n-k,\emptyset,\ldots,\emptyset) \cup \ulambda}\downarrow_{C_r \wr S_{n-1}}^{C_r \wr S_n}\right)\uparrow_{C_r \wr S_{n-1}}^{C_r \wr S_n}\\
	=&\phantom{+}\;\chi^{(n-k,\emptyset,\ldots,\emptyset) \cup \ulambda}\\
	&+ \sum_{\umu \in \ulambda^+}\chi^{(n-k-1,\emptyset,\ldots,\emptyset) \cup \umu}\\
	&+ \sum_{\umu \in \ulambda^-}\chi^{(n-k+1,\emptyset,\ldots,\emptyset) \cup \umu}\\
	&+ \sum_{\umu \in \ulambda^-}\sum_{\underline{\tau} \in \umu^+}\chi^{(n-k,\emptyset,\ldots,\emptyset) \cup \underline{\tau}}.
\end{align*}
Now we can compute the product $r\theta F_k^{(r)}$ and obtain
\begin{align*}
	&\phantom{+}\; r\theta F_k^{(r)} = \xi_1 F_k^{(r)} = \sum_{\ulambda \in \Lambda_k(C_r)} \deg(\chi^\ulambda) \,\xi_1\chi^{(n-k,\emptyset,\ldots,\emptyset) \cup \ulambda}\\
	=&\phantom{+}\;\sum_{\ulambda \in \Lambda_k(C_r)} \deg(\chi^\ulambda) \,\chi^{(n-k,\emptyset,\ldots,\emptyset) \cup \ulambda}\\
	&+ \sum_{\ulambda \in \Lambda_k(C_r)} \deg(\chi^\ulambda) \,\sum_{\umu \in \ulambda^+}\chi^{(n-k-1,\emptyset,\ldots,\emptyset) \cup \umu}\\
	&+ \sum_{\ulambda \in \Lambda_k(C_r)} \deg(\chi^\ulambda) \,\sum_{\umu \in \ulambda^-}\chi^{(n-k+1,\emptyset,\ldots,\emptyset) \cup \umu}\\
	&+ \sum_{\ulambda \in \Lambda_k(C_r)} \deg(\chi^\ulambda) \,\sum_{\umu \in \ulambda^-}\sum_{\underline{\tau} \in \umu^+}\chi^{(n-k,\emptyset,\ldots,\emptyset) \cup \underline{\tau}}\\
	=&\phantom{+}\; F_k^{(r)} +A + B +C
\end{align*}
where 
\begin{align*}
	A &= \sum_{\ulambda \in \Lambda_k(C_r)} \deg(\chi^\ulambda) \,\sum_{\umu \in \ulambda^+}\chi^{(n-k-1,\emptyset,\ldots,\emptyset) \cup \umu},\\
	B &= \sum_{\ulambda \in \Lambda_k(C_r)} \deg(\chi^\ulambda) \,\sum_{\umu \in \ulambda^-}\chi^{(n-k+1,\emptyset,\ldots,\emptyset) \cup \umu},\\
	C &= \sum_{\ulambda \in \Lambda_k(C_r)} \deg(\chi^\ulambda) \,\sum_{\umu \in \ulambda^-}\sum_{\underline{\tau} \in \umu^+}\chi^{(n-k,\emptyset,\ldots,\emptyset) \cup \underline{\tau}}.
\end{align*}
In view of Equation~\eqref{eqn:recurrence_for_Fk}, it remains to show that $A = F_{k+1}^{(r)}$, $B = rkF_{k-1}^{(r)}$ and $C = rkF_k^{(r)}$.

First, notice that summing over all pairs $(\ulambda,\umu)$ where $\ulambda \in \Lambda_k(C_r)$ and $\umu \in \ulambda^+$ is equivalent to summing over all pairs $(\ulambda,\umu)$ where $\umu \in \Lambda_{k+1}(C_r)$ and $\ulambda \in \umu^-$. This is easily seen by double counting the set of all pairs $(\ulambda,\umu)$ such that $\ulambda \in \Lambda_{k}(C_r)$, $\umu \in \Lambda_{k+1}(C_r)$, $\ulambda(g) \subseteq \umu(g)$ for some $g \in C_r$ and $\ulambda(h) = \umu(h)$ for all $h \neq g$. It follows that
\begin{align*}
	A &= \sum_{\ulambda \in \Lambda_k(C_r)} \deg(\chi^\ulambda) \,\sum_{\umu \in \ulambda^+}\chi^{(n-k-1,\emptyset,\ldots,\emptyset) \cup \umu}\\
	&= \sum_{\ulambda \in \Lambda_k(C_r)} \sum_{\umu \in \ulambda^+}\,\deg(\chi^\ulambda) \,\chi^{(n-k-1,\emptyset,\ldots,\emptyset) \cup \umu}\\
	&= \sum_{\umu \in \Lambda_{k+1}(C_r)} \sum_{\ulambda \in \umu^-}\,\deg(\chi^\ulambda) \,\chi^{(n-k-1,\emptyset,\ldots,\emptyset) \cup \umu}\\
	&= \sum_{\umu \in \Lambda_{k+1}(C_r)} \left(\sum_{\ulambda \in \umu^-}\,\deg(\chi^\ulambda) \right) \chi^{(n-k-1,\emptyset,\ldots,\emptyset) \cup \umu}.
\end{align*}
Now let $\umu \in \Lambda_{k+1}(C_r)$ and recall that $\deg(\chi^\umu) = \chi^\umu (e)$. Since restricting a character does not change its dimension, we find from the branching rule that
\[
	\deg(\chi^\umu) = \deg\left(\chi^\umu \downarrow_{C_r \wr S_{k}}^{C_r \wr S_{k+1}}\right) = \left(\chi^\umu \downarrow_{C_r \wr S_{k}}^{C_r \wr S_{k+1}}\right)(e) =  \sum_{\ulambda \in \umu^-} \chi^\ulambda \,(e) = \sum_{\ulambda \in \umu^-}\,\deg(\chi^\ulambda).
\]
Thus, we deduce that $A = F_{k+1}^{(r)}$.

Similarly, notice that summing over all pairs $(\ulambda,\umu)$ where $\ulambda \in \Lambda_k(C_r)$ and $\umu \in \ulambda^-$ is equivalent to summing over all pairs $(\ulambda,\umu)$ where $\umu \in \Lambda_{k-1}(C_r)$ and $\ulambda \in \umu^+$. It follows that
\begin{align*}
	B &= \sum_{\ulambda \in \Lambda_k(C_r)} \deg(\chi^\ulambda) \,\sum_{\umu \in \ulambda^-}\chi^{(n-k+1,\emptyset,\ldots,\emptyset) \cup \umu}\\
	&= \sum_{\ulambda \in \Lambda_k(C_r)} \sum_{\umu \in \ulambda^-}\,\deg(\chi^\ulambda) \,\chi^{(n-k+1,\emptyset,\ldots,\emptyset) \cup \umu}\\
	&= \sum_{\umu \in \Lambda_{k-1}(C_r)} \sum_{\ulambda \in \umu^+}\,\deg(\chi^\ulambda) \,\chi^{(n-k+1,\emptyset,\ldots,\emptyset) \cup \umu}\\
	&= \sum_{\umu \in \Lambda_{k-1}(C_r)} \left(\sum_{\ulambda \in \umu^+}\,\deg(\chi^\ulambda) \right)\chi^{(n-k+1,\emptyset,\ldots,\emptyset) \cup \umu}.
\end{align*}
Now let $\umu \in \Lambda_{k-1}(C_r)$. Since inducing a character from $C_r \wr S_{k-1}$ to $C_r \wr S_k$ changes its degree by $\frac{|C_r \wr S_k|}{|C_r \wr S_{k-1}|} =rk$, we find that
\[
	rk\deg(\chi^\umu) = \deg\left(\chi^\umu \uparrow_{C_r \wr S_{k-1}}^{C_r \wr S_{k}}\right) = \left(\chi^\umu \uparrow_{C_r \wr S_{k-1}}^{C_r \wr S_{k}}\right)(e) =  \sum_{\ulambda \in \umu^+} \chi^\ulambda \,(e) = \sum_{\ulambda \in \umu^+}\,\deg(\chi^\ulambda).
\]
Thus, we deduce that $B = rkF_{k-1}^{(r)}$.

Finally, we turn to $C$. Notice that summing over all triples $(\ulambda,\umu,\underline{\tau})$ with $\ulambda \in \Lambda_k(C_r)$ such that $\umu \in \ulambda^-$ and $\underline{\tau} \in \umu^+$ is equivalent to summing over all triples $(\ulambda,\umu,\underline{\tau})$ with $\underline{\tau} \in \Lambda_k(C_r)$ such that $\umu \in \underline{\tau}^-$ and $\ulambda \in \umu^+$. It follows that
\begin{align*}
	C &= \sum_{\ulambda \in \Lambda_k(C_r)} \deg(\chi^\ulambda) \,\sum_{\umu \in \ulambda^-}\sum_{\underline{\tau} \in \umu^+}\chi^{(n-k,\emptyset,\ldots,\emptyset) \cup \underline{\tau}}\\
	&= \sum_{\ulambda \in \Lambda_k(C_r)} \sum_{\umu \in \ulambda^-}\sum_{\underline{\tau} \in \umu^+}\,\deg(\chi^\ulambda) \,\chi^{(n-k,\emptyset,\ldots,\emptyset) \cup \underline{\tau}}\\
	&= \sum_{\underline{\tau} \in \Lambda_k(C_r)}  \sum_{\umu \in \underline{\tau}^-}\sum_{\ulambda \in \umu^+}\,\deg(\chi^\ulambda)\,\chi^{(n-k,\emptyset,\ldots,\emptyset) \cup \underline{\tau}}\\
	&= \sum_{\underline{\tau} \in \Lambda_k(C_r)}  \left(\sum_{\umu \in \underline{\tau}^-}\sum_{\ulambda \in \umu^+}\,\deg(\chi^\ulambda)\right)\chi^{(n-k,\emptyset,\ldots,\emptyset) \cup \underline{\tau}}.
\end{align*}
Now let $\underline{\tau} \in \Lambda_{k}(C_r)$. As before, we find
\begin{align*}
	rk \deg(\chi^{\underline{\tau}}) &= \deg\left(\chi^{\underline{\tau}} \downarrow_{C_r \wr S_{k-1}}^{C_r \wr S_{k}}\uparrow_{C_r \wr S_{k-1}}^{C_r \wr S_{k}}\right) = \left(\chi^{\underline{\tau}} \downarrow_{C_r \wr S_{k-1}}^{C_r \wr S_{k}}\uparrow_{C_r \wr S_{k-1}}^{C_r \wr S_{k}}\right)(e)\\
	 &=  \sum_{\umu \in \underline{\tau}^-}\sum_{\ulambda \in \umu^+} \chi^\ulambda \,(e) = \sum_{\umu \in \underline{\tau}^-}\sum_{\ulambda \in \umu^+}\,\deg(\chi^\ulambda).
\end{align*}
Thus, we deduce that $C = rkF_{k}^{(r)}$ which completes the proof.
\end{proof}
In other words, Theorem~\ref{thm:decomposition_Charlier_chars} implies that for $0 \leq k \leq n/2$, the character~$C^{(1/r)}_k(\theta)$ decomposes into those characters $\chi^{\ulambda}$ such that the biggest part of~$\ulambda(0)$ is $n-k$. Furthermore, notice that the case $r=1$ in Theorem~\ref{thm:decomposition_Charlier_chars} gives exactly the statement of Theorem~7 in \cite{Tar99}.

We get the following corollary from Theorem~\ref{thm:decomposition_Charlier_chars}.
\begin{corollary}
For $0 \leq k , \ell \leq n/2$, we have
\[
	\langle C_k^{(1/r)}(\theta),C_\ell^{(1/r)}(\theta) \rangle = \delta_{k \ell} r^k k!.
\]
\end{corollary}

\begin{proof}
We find from Theorem~\ref{thm:decomposition_Charlier_chars} that
\begin{align*}
	&\langle C_k^{(1/r)}(\theta),C_\ell^{(1/r)}(\theta) \rangle \\
	=& \sum_{\ulambda \in \Lambda_k(C_r)} \sum_{\umu \in \Lambda_\ell(C_r)} \deg(\chi^\ulambda)\deg(\chi^\umu)\,\langle \chi^{(n-k,\emptyset,\ldots,\emptyset) \cup \ulambda},\chi^{(n-\ell,\emptyset,\ldots,\emptyset) \cup \umu} \rangle.
\end{align*}
Since $k,\ell \leq n/2$, we have that $n-k,n-\ell \geq n/2$ and thus $n-k$ and $n - \ell$ are the biggest parts of the partitions in the entry indexed by $0$. Thus, the inner products can only be non-zero if $n-k = n-\ell$, equivalently $k = \ell$. In this case, we find that
\begin{align*}
	&\sum_{\ulambda \in \Lambda_k(C_r)} \sum_{\umu \in \Lambda_k(C_r)} \deg(\chi^\ulambda)\deg(\chi^\umu)\,\langle \chi^{(n-k,\emptyset,\ldots,\emptyset) \cup \ulambda},\chi^{(n-k,\emptyset,\ldots,\emptyset) \cup \umu} \rangle\\
	=&\sum_{\ulambda \in \Lambda_k(C_r)} \sum_{\umu \in \Lambda_k(C_r)} \deg(\chi^\ulambda)\deg(\chi^\umu)\,\delta_{\ulambda\,\umu}\\
	=& \sum_{\ulambda \in \Lambda_k(C_r)}\deg(\chi^\ulambda)^2.
\end{align*}
Now the well-known formula
\[
	\sum_{\chi \in \text{Irr}(G)} \deg(\chi)^2 = |G|
\]
finishes the proof.
\end{proof}

\begin{remark}
Observe that $\xi_k$ is the induction of the trivial character of the subgroup $\left(S_1\right)^k\times (C_r\wr S_{n-k})$ to $C_r\wr S_n$. An application of Theorem~\ref{thm:decomposition_perm_char} with $\usigma\in\Sigma_n(C_r)$ given by $\usigma(\{0\})=(1^k)$ and $\usigma(C_r\wr S_n)=(n-k)$ therefore gives\[
\ang{\xi_k,\chi^\ulambda}\ne 0\;\Longleftrightarrow\; ((1^k),n-k) \to \ulambda
\]
for each $\ulambda\in\Lambda_n(C_r)$. From the definition of the relation $\to$, it is easy to see that the elements $\ulambda$ satisfying $((1^k),n-k) \to \ulambda$ are precisely those satisfying $\ulambda(0)_1 \geq n-k$.

It follows from Theorem~\ref{thm:decomposition_Charlier_chars} that for $0 \leq k \leq n/2$, inputting $\theta$ into the $k$-th Charlier polynomial $C_k^{(1/r)}$ gives the irreducible characters $\chi^\ulambda$ that satisfy $\ulambda(0)_1 = n-k$. These are precisely the irreducible constituents $\chi^{\ulambda}$ appearing in the decomposition of $\xi_k$ but not in the decomposition of $\xi_{k-1}$.
\end{remark}

\subsection{Designs and Codes}

Henceforth, we call a $((1^t),n-t)$-transitive set $Y$ of $C_r \wr S_n$ a \emph{$t$-design}. In the symmetric group, that is the case $r=1$, a $t$-design is simply a $t$-transitive subset. The \emph{index} of $Y$ is the number of elements in $Y$ that stabilise an arbitrary $((1^t),n-t)$-tabloid (so the index is the constant $c$ in the definition of a transitive set). Thus, a $t$-design is transitive on the $t$-tuples $((c_1,i_1),\ldots,(c_t,i_t))$ where the elements $c_k \in [r]$ are arbitrary and the elements $i_k \in [n]$ are pairwise distinct. We call a $t$-design of index $1$ a \emph{sharply $t$-transitive} set.

We also call a $((1^{n-d+1}),d-1)$-clique $Y$ of $C_r \wr S_n$ a \emph{$d$-code}. Hence, for any distinct $x,y \in Y$ we have that $x^{-1}y$ does not stabilise a $((1^{n-d+1}),d-1)$-tabloid. This implies that for any distinct $x,y \in Y$ there are at least $d$ elements $i \in [n]$ such that $x$ and $y$ map $(0,i)$ to different elements.

We can also give a geometric interpretation of $t$-designs.
\begin{example}\label{ex:geometric t designs}
Let $G = C_2$ and $n=3$. Then $W=C_2\wr S_3$ is the symmetry group of a $3$-dimensional cube. For a $1$-design, we need to study the subgroup $H = C_2 \wr S_2$. Geometrically, $H$ is the stabiliser of a face of the cube. Thus, $Y \subseteq C_2 \wr S_3$ is a $1$-design if and only if $Y$ is transitive on the six faces of a $3$-dimensional cube.

We can extend this example to $t$-designs and arbitrary $n \in \N$. Then we consider the subgroup $H = (S_1)^t \times C_2 \wr S_{n-t} = C_2 \wr S_{n-t}$. Here, $H$ is the stabiliser of a chain of faces $F_0 \subseteq F_1 \subseteq \ldots \subseteq F_{t}$ of the cube where $F_i$ has dimension $n-t+i$. In other words, $H$ stabilises a $((1^t),n-t)$-flag (see Chapter~\ref{sec:intro}). Thus, $Y \subseteq C_2 \wr S_n$ is a $t$-design if and only if $Y$ is transitive on the set of $((1^t),n-t)$-flags of the $n$-dimensional cube. Similar interpretations can be given for other regular polytopes.
\end{example}

Using the results from Section~\ref{sec:notions of transitivity}, we can characterise codes and designs in terms of the inner and dual distribution, respectively.
\begin{corollary}\label{cor:designs_codes_zero_distributions}
Let $Y$ be a subset of $C_r\wr S_n$ with inner distribution $(a_\umu)$ and dual distribution $(a'_{\ulambda})$. Then $Y$ is a $t$-design if and only if
\[
a'_{\ulambda}=0\quad\text{for each $\ulambda\in\Lambda_n(C_r)$ satisfying $n-t\le \ulambda(0)_1<n$},
\]
and a $d$-code if and only if
\[
a_{\umu}=0\quad\text{for each $\umu\in\Lambda_n(C_r)$ satisfying $n-d+1\le\umu(0)'_1 - \umu(0)'_2<n$}.
\]
\end{corollary}
\begin{proof}
The first part follows from Theorem~\ref{thm:characterisation_designs}. Note that for $\usigma \in \Sigma_n(G)$ given by $\usigma(\{0\}) = (1^t)$ and $\usigma(G) = (n-t)$, we have $\usigma \to \ulambda$ if and only if $\ulambda(0)_1 \geq n-t$, as was already observed before.

The second part follows from Theorem~\ref{thm:characterisation_cliques}. We have to compute the set $C(\usigma)$ for $\usigma = ((1^{n-d+1}),d-1)$. First, note that the conjugacy classes of $C_r \wr S_n$ appearing in the subgroup $C_r \wr S_{d-1}$ are exactly the classes indexed by the set $\Lambda_{d-1}(C_r)$. Moreover, the trivial subgroup $S_1$ only has the trivial conjugacy class $\umu \in \Lambda_1(C_r)$ with $\umu (0)=1$ and $\umu(g) = \emptyset$ otherwise. Hence, Theorem~\ref{thm:cc_stabiliser} gives that the elements of $C(\usigma)$ are precisely the elements $\umu \in \Lambda_n(C_r)$ satisfying
\[
	\umu = ((1^{n-d+1}),\emptyset,\ldots,\emptyset) \cup \ulambda
\]
for some element $\ulambda \in \Lambda_{d-1}(C_r)$. Taking duals, we see that these are precisely the elements $\umu\in\Lambda_n(G)$ with $\umu(0)'_1 - \umu(0)'_2 \geq n-d+1$.
\end{proof}
Note that $d(x,y)=n-\theta(x^{-1}y)$ defines a metric on $C_r\wr S_n$. We can therefore define codes and a distance distribution in the usual way. For a non-empty subset $Y$ of $C_r \wr S_n$, we define the \emph{distance distribution} to be the tuple $(A_i)_{0\leq i \leq n}$, where
\[
	A_i = \frac{1}{|Y|}\left| \{(x,y) \in Y \times Y : d(x,y) = i\} \right|,
\]
and the \emph{dual distance distribution} of $Y$ to be the tuple $(A'_k)_{0 \leq k \leq n}$, where
\[
	A'_k=\sum_{i=0}^n C^{(1/r)}_k(n-i)A_i,
\]
or equivalently
\begin{equation}\label{eqn:def_dual_distance}
A'_k=\frac{1}{\abs{Y}}\sum_{x,y\in Y}C^{(1/r)}_k(\theta(x^{-1}y))=\frac{1}{\abs{Y}}\sum_{x,y\in Y}C^{(1/r)}_k(n-d(x,y)).
\end{equation}
We can characterise $t$-designs in terms of zeroes in the dual distance distribution.
\begin{proposition}
\label{pro:t-designs_dual_dist}
Let $Y$ be a subset of $C_r\wr S_n$ with dual distance distribution $(A'_k)$ and let $t$ be an integer satisfying $1\le t\le n$. If $Y$ is a $t$-design, then $A'_k=0$ for all $k$ satisfying $1\le k\le t$. Moreover, the converse also holds if $t\le n/2$. That is, if $t\le n/2$ and $A'_k=0$ for all $k$ satisfying $1\le k\le t$, then $Y$ is a $t$-design.
\end{proposition}
\begin{proof}
First, assume that $Y \subseteq C_r \wr S_n$ is a $t$-design. From \eqref{eqn:def_dual_distance} and \eqref{eqn:C_from_xi}, we find that
\begin{equation}\label{eqn:dual_dist_permchar}
	A'_k = \frac{1}{|Y|}\sum_{j=0}^k (-1)^{k-j} \binom{k}{j}\sum_{x,y \in Y} \xi_j(x^{-1}y).
\end{equation}
By Theorem~\ref{thm:decomposition_perm_char}, the permutation character $\xi_j$ decomposes into those irreducible characters $\chi^{\ulambda}$ for which $\ulambda(0)_1 \geq n-j$. Moreover, since $\xi_j$ is a permutation character, it contains the trivial character $1_{C_r \wr S_n}$ as an irreducible constituent with multiplicity~$1$. From \eqref{eqn:dual_distribution_characters_general} and Corollary~\ref{cor:designs_codes_zero_distributions}, it then follows that for all $j$ with $0 \leq j \leq t$, the inner sum in \eqref{eqn:dual_dist_permchar} is
\[
	\sum_{x,y \in Y} \xi_j(x^{-1}y) = \sum_{x,y \in Y} 1_{C_r \wr S_n}(x^{-1}y)=|Y|^2.
\]
Hence, for all $k$ with $0 \leq k \leq t$, we have
\[
	A'_k = |Y| \sum_{j=0}^k (-1)^{k-j} \binom{k}{j} = |Y| (1+ (-1))^k = |Y| \delta_{k,0}.
\]
So we find that $A'_k = 0$ for all $k$ satisfying $1 \leq k \leq t$.

Now, for each $k$ with $0 \leq k \leq n/2$, we find from Theorem~\ref{thm:decomposition_Charlier_chars} that
\[
	C_k^{(1/r)}(\theta) = \sum_{\substack{\ulambda \in \Lambda_n(C_r)\;:\\\ulambda(0)_1 = n-k}} \deg(\chi^{\ulambda})\,\chi^{\ulambda}.
\]
Together with \eqref{eqn:def_dual_distance} and \eqref{eqn:dual_distribution_characters_general}, we have that
\[
	A'_k = \frac{1}{|Y|} \sum_{\substack{\ulambda \in \Lambda_n(C_r)\;:\\\ulambda(0)_1 = n-k}} \deg(\chi^{\ulambda})\sum_{x,y \in Y}\chi^{\ulambda}(x^{-1}y)  = \sum_{\substack{\ulambda \in \Lambda_n(C_r)\;:\\\ulambda(0)_1 = n-k}} a'_{\ulambda}.
\]
Suppose that $t \leq n/2$ and $A'_k = 0$ for all $k$ satisfying $1 \leq k \leq t$. Since the dual distribution $a'$ is non-negative, it follows that $a'_{\ulambda} = 0$ for all $\ulambda \in \Lambda_n(C_r)$ with $n-t \leq \ulambda(0)_1 < n$. Corollary~\ref{cor:designs_codes_zero_distributions} then implies that $Y$ is a $t$-design.
\end{proof}
We also have the following bounds.
\begin{corollary} \label{cor:bound_codes}
Let $Y$ be a subset of $C_r\wr S_n$ and let $d$ and $t$ be the largest integers such that $Y$ is a $d$-code and a $t$-design. Then
\[
r^t\,(n)_t\le\abs{Y}\le r^{n-d+1}\,(n)_{n-d+1}.
\]
Moreover, if equality holds in one of the bounds, then equality also holds in the other and this case happens if and only if $d=n-t+1$.
\end{corollary}
\begin{proof}
Use Theorem~\ref{thm:clique_design_bound} for $\usigma = ((1^t),n-t)$ and $\usigma = ((1^{n-d+1}),d-1)$.
\end{proof}
If we let $r=1$, then the upper bound in Corollary~\ref{cor:bound_codes} becomes a well-known bound for permutation codes (see \cite{BlaCohDez1979}). Moreover, the bounds in Corollary~\ref{cor:bound_codes} can be achieved by sharply $t$-transitive sets. A construction for sharply $2$-transitive sets is given in Section~\ref{sec:constructions}.

The last result is about distance distributions. It turns out that the distance distribution of a subset $Y$ of $C_r \wr S_n$ is uniquely determined, provided that $Y$ is a $t$-design and a $d$-code where $d \geq n-t$.
\begin{theorem}
\label{thm:inner_dist}
Suppose that $Y$ is a $t$-design and an $(n-t)$-code in $C_r\wr S_n$. Then the distance distribution $(A_i)$ of $Y$ satisfies
\[
A_{n-i}=\sum_{j=i}^t(-1)^{j-i}{\binom{j}{i}}{\binom{n}{j}}\bigg(\frac{\abs{Y}}{r^j(n)_j}-1\bigg)
\]
for each $i\in\{0,1,\dots,n-1\}$.
\end{theorem}
\begin{proof}
We have
\begin{equation}\label{eqn:dual_dist_in_proof}
A'_k=\sum_{i=0}^nC^{(1/r)}_k(i)A_{n-i}.
\end{equation}
Using binomial inversion \eqref{eqn:binomial_inversion}, we find from the definition of the Charlier polynomials that
\begin{equation}
r^j(x)_j=\sum_{k=0}^j{\binom{j}{k}}C^{(1/r)}_k(x).   \label{eqn:Charlier_inverse}
\end{equation}
Multiply both sides of \eqref{eqn:dual_dist_in_proof} by ${\binom{j}{k}}$, sum over $k$, and use~\eqref{eqn:Charlier_inverse} to find that
\[
\sum_{k=0}^j{\binom{j}{k}} A'_k=\sum_{i=0}^nA_{n-i}\,r^j\,(i)_j.
\]
Since $Y$ is a $t$-design, we find from Proposition~\ref{pro:t-designs_dual_dist} that $A'_1=\cdots=A'_t=0$. So for every $j \in \{1,\ldots,t\}$, we have
\[
	A'_0 = \sum_{i=0}^nA_{n-i}\,r^j\,(i)_j.
\]
Similarly, since $Y$ is an $(n-t)$-code, we have $A_1=\ldots=A_{n-t-1}=0$, and we find
\[
	A'_0 = A_0 r^j (n)_j+ \sum_{i=0}^t A_{n-i}\,r^j\,(i)_j.
\]
for every $j \in \{1,\ldots,t\}$. Moreover, we have $A_0=1$ and $A'_0=\abs{Y}$ and therefore
\[
\abs{Y}-r^j\,(n)_j=\sum_{i=0}^tA_{n-i}\,r^j\,(i)_j
\]
for each $j\in\{1,2,\dots,t\}$. Divide by $j!\,r^j$ to obtain
\[
\sum_{i=0}^t{\binom{i}{j}} A_{n-i}=\frac{\abs{Y}-r^j\,(n)_j}{r^j\,j!} = \frac{(n)_j}{j!}\left(\frac{|Y| - r^j(n)_j}{r^j (n)_j}\right) ={\binom{n}{j}}\bigg(\frac{\abs{Y}}{r^j\,(n)_j}-1\bigg)
\]
for each $j\in\{1,2,\dots,t\}$. Now \eqref{eqn:binomial_inversion_variant} gives the desired result.
\end{proof}
\begin{remark}\label{rem:Poisson_moments}
Note that for $Y=C_r\wr S_n$, we have that $Y$ is an $n$-design and a $1$-code. Moreover, $A_{n-i}$ equals the number of elements in $C_r\wr S_n$ with precisely $i$ fixed points, hence $A_{n-i}=w_i$. In this case, Theorem~\ref{thm:inner_dist} gives the well-known expression
\[
w_i=\frac{n!\,r^{n-i}}{i!}\sum_{j=0}^{n-i}\frac{(-1)^j}{j!\,r^j}
\]
for $i=0,1,\ldots,n-1$ after some elementary calculations. Additionally, we clearly have the value $w_n = 1$.
\end{remark}


\section{Existence results}\label{sec:constructions}

In this section, we give a simple construction of $t$-designs in $C_r\wr S_n$ using $t$-designs in~$S_n$ and orthogonal arrays. This construction is reminiscent of constructions of designs in permutation groups using Gelfand pairs, see~\cite{BanNakOkuZha2022} and~\cite{Ito2004} for general groups,~\cite{MarSag2007} for~$S_n$, and~\cite{ErnSch2022} for~$\GL(n,q)$. For a Gelfand pair $(G,H)$, such constructions use designs in $H$ and $G/H$ to construct designs in~$G$. We also present an application of our results to the well-known prime power conjecture for finite projective planes.

It is well-known that $(C_r\wr S_n,S_n)$ is a Gelfand pair (since $C_r$ is abelian) that has applications to parking functions (see \cite{AkeCan12} for details). It also gives rise to an association scheme. The point set of this scheme is $(C_r\wr S_n)/S_n$, which we identify with $C_r^n$. We do not need to analyse this association scheme further but we remark that orthogonal arrays with $n$ columns and $r$ symbols can be interpreted as subsets of this scheme.

In what follows, we identify the elements of $C_r\wr S_n$ with pairs $(y,z)$ where $y\in C_r^n$ and $z\in S_n$. The following result is immediate.
\begin{proposition}
\label{pro:iterative_construction}
Let $Y$ be a $t$-design in $S_n$ and let~$D$ be an orthogonal array of strength $t$ with $n$ columns and $r$ symbols. Then
\[
\{(g,y) : g\in D,y\in Y\}
\]
is a $t$-design in $C_r\wr S_n$.
\end{proposition}
\begin{proof}
Note that there are constantly many elements $(g,y)$ which map a $t$-tuple $((c_1,i_1),\ldots,(c_t,i_t))$ with all $i_k$ pairwise distinct to another $t$-tuple of this type.
\end{proof}
The special case that $Y=S_n$ in Proposition~\ref{pro:iterative_construction} was also stated by Ito~\cite[Section~8]{Ito2004} in a more general context.

\begin{example}
If $r$ is a prime power and $r\ge n+1$, then it is well-known that there are orthogonal arrays of strength $t$ in $H(n,r)$ of minimum size $r^t$. These objects are equivalent to MDS codes in $H(n,r)$ of minimum distance $n-t+1$. Hence, if $r\ge n+1$ and there exists a $t$-design in $S_n$ of index $1$ (namely a sharply $t$-transitive set in $S_n$), then by Proposition~\ref{pro:iterative_construction} there exists also a $t$-design of index $1$ in $C_r\wr S_n$. For example, if $n$ is a prime power, then $\AGL(1,n)$ inside $S_n$ is a $2$-design of index $1$ and hence there exist $2$-designs of index 1 in $C_r\wr S_n$ whenever $r$ and $n$ are prime powers with $r\ge n+1$.
\end{example}

Not every $t$-design in $C_r \wr S_n$ comes from the construction of Proposition~\ref{pro:iterative_construction}. For example, there is a $2$-design of index $1$ in $C_2 \wr S_5$ consisting of $80$ elements (see appendix). If this example came from Proposition~\ref{pro:iterative_construction}, then it would have to come from a sharply $2$-transitive set in $S_5$ and an orthogonal array~$A$ of strength $2$ and index $1$. However, it is easy to see that such an orthogonal array does not exist.

Now note that by Corollary~\ref{cor:bound_codes}, each $t$-design $Y$ in $C_r\wr S_n$ must satisfy $\abs{Y}\ge r^t(n)_t$. Using the results of \cite{KupLovPel2017}, we now show that there are always $t$-designs in $C_r\wr S_n$ almost as small as this lower bound.

\begin{corollary}
\label{cor:existence_designs}
Let $r,n,t$ be integers satisfying $r\ge 2$, $n\ge 1$, and $1\le t\le n$. Then there exists a $t$-design $Y$ in $C_r\wr S_n$ satisfying $\abs{Y}\le (crn)^{ct}$ for some universal constant $c>0$.
\end{corollary}
\begin{proof}
From~\cite[Theorems 1.1 and 1.6]{KupLovPel2017}, we find that there exists an orthogonal array~$D$ of strength~$t$ in $H(n,r)$ satisfying $\abs{D}\le \left(\frac{c_1rn}{t}\right)^{c_1t}$ and a $t$-design~$Y$ in $S_n$ satisfying $\abs{Y}\le (c_2n)^{c_2t}$ for some universal constants $c_1,c_2>0$. The statement of the corollary then follows from Proposition~\ref{pro:iterative_construction}.
\end{proof}

Note that Corollary~\ref{cor:existence_designs} also gives the existence of more general designs in $C_r\wr S_n$, since in view of Theorem~\ref{thm:designs_parabolic}, every $t$-design is also a $(\sigma,\rho)$-design, where $(\sigma,\rho)$ is a double partition of $n$ such that $\rho_1\ge n-t$.

\section{Application to projective planes}\label{sec:further research}

In this section, we explore a connection to projective planes. For an introduction to projective planes, we refer the reader to \cite{HugPip1973}. We make use of two well-known results on the existence of projective planes.

\begin{theorem}[Bose \cite{Bos1938}]\label{thm:nonexistence bose}
There exists a projective plane of order $n$ if and only if there exists a complete set of $n-1$ mutually orthogonal Latin squares of order $n$.
\end{theorem}

It is well-known that there exists a complete set of mutually orthogonal Latin squares of order $n$ if and only if there exists an orthogonal array of strength~$2$ with $n+1$ columns, $n$ symbols and index $1$ (see for example~\cite[Theorem~6.38]{Sti04}). Thus, the existence of a projective plane of order $n$ is equivalent to the existence of a $2$-design in the Hamming scheme $H(n+1,n)$ of index $1$.

It is also well-known that the existence of a finite projective plane of order $n$ is connected to the existence of a sharply $2$-transitive set in $S_n$, so a $2$-design in $S_n$ of index $1$.

\begin{theorem}[Hall \cite{Hal1943}]\label{thm:nonexistence transitive}
There exists a projective plane of order $n$ if and only if there exists a sharply $2$-transitive subset of $S_n$. 
\end{theorem}

Using Proposition~\ref{pro:iterative_construction}, we can construct a $2$-design in $C_r \wr S_n$ from a $2$-design in $S_n$ and an orthogonal array of strength $2$. This gives the following theorem.

\begin{theorem}\label{thm:nonexistence projective planes}
Let $n$ be a positive integer. If there is no $2$-design of index $1$ in $C_{n-1} \wr S_{n}$, then there is no projective plane of order $n-1$ or there is no projective plane of order $n$.
\end{theorem}

\begin{proof}
We show that the existence of a projective plane of order $n-1$ and order $n$ implies the existence of a $2$-design of index $1$ in $C_{n-1} \wr S_n$.

A projective plane of order $n$ gives rise to a $2$-design in $S_n$ of index $1$. A projective plane of order $n-1$ gives rise to an orthogonal array of strength $2$ with $n$ columns and $n-1$ symbols of index $1$. Now Proposition~\ref{pro:iterative_construction} gives a $2$-design in $C_{n-1} \wr S_n$ of index $1$.
\end{proof}

If $n$ is a prime power, then there exists a projective plane of order or $n$. In this case, the non-existence of a $2$-design in $C_{n-1} \wr S_n$ implies the non-existence of a projective plane of order $n-1$. Similarly, if $n-1$ is a prime power, then the non-existence of a $2$-design in $C_{n-1} \wr S_n$ implies the non-existence of a projective plane of order $n$. We summarise our findings in the following corollary.

\begin{corollary}\label{cor:fpp}
Let $q$ be a prime power. If there is no $2$-design of index~$1$ in $C_{q-1} \wr S_q$, then there is no finite projective plane of order~$q-1$. If there is no $2$-design of index~$1$ in $C_{q} \wr S_{q+1}$, then there is no finite projective plane of order~$q+1$. 
\end{corollary}

Note that it is possible that there exist $2$-designs of index $1$ in $C_{n-1} \wr S_n$ that do not come from the construction of Proposition~\ref{pro:iterative_construction}. In this case, Theorem~\ref{thm:nonexistence projective planes} cannot rule out the existence of a finite projective plane. More detailed analysis of the structure of transitive sets in $C_{n-1} \wr S_n$ is needed to explore this connection further.


\section*{Open Problems}

We list some open problems for further research that the reader might be interested to work on.
\begin{enumerate}[label=\textup{(}\arabic*\textup{)}]
    \item{
        Prove Conjecture~\ref{conj:small steps} or Conjecture~\ref{conj:equivalent conditions} characterising the relation $\succeq$.
    }
    \item{
        Find a geometric interpretation for $\usigma$-transitivity in the group $G \wr S_n$ where $G$ is not cyclic.
    }
    \item{
        Can Corollary~\ref{cor:fpp} be used to show the non-existence of a finite projective plane? In order to rule out the case $n=12$, one would have to work in the group $C_{11} \wr S_{12}$ or $C_{12} \wr S_{13}$.
    }
\end{enumerate}


\section*{Acknowledgements}

The first author was partially funded by the Deutsche Forschungsgemeinschaft (DFG, German Research Foundation) -- Project-ID 491392403 -- TRR 358. He would like to thank Barbara Baumeister, Alena Ernst, Karen Meagher and Charlene Weiß for helpful discussions and proofreading of this work. Sadly, the second author passed away during the writing process. The first author is grateful to him for the supervision and the fruitful collaboration on this article.

\bibliographystyle{amsplain}
\bibliography{references}

\section*{Appendix}

Here, we give an explicit example of a $2$-design $Y$ of index $1$ (a sharply $2$-transitive set) in the group $C_2 \wr S_5$. It is neither a subgroup nor a coset of a subgroup. In fact, the group generated by $Y$ is the full group $C_2 \wr S_5$. We write the elements that the group acts on as $1,\,\overline{1},\,2,\,\overline{2},\,3,\,\overline{3},\,4,\,\overline{4},\,5,\,\overline{5}$.

We can write our set $Y$ as $Y=RH$ where $R$ is a set of $20$ elements and $H$ is a set of $4$ elements. For better readability, we give the generalised permutations in one line notation. That is, we write down the sequence $\big(\sigma(1),\sigma(2),\sigma(3),\sigma(4),\sigma(5)\big)$ for the permutation $\sigma$. From the definition of the hyperoctahedral group, it follows that if $\sigma(x) = y$, then $\sigma(\overline{x}) = \overline{y}$, so the one line notation contains all information about the generalised permutation.
\begin{align*}
H=&\{1\,2\,3\,4\,5,\;1\,\overline{2}\,\overline{3}\,\overline{4}\,\overline{5},\;1\,3\,2\,5\,4,\;1\,\overline{3}\,\overline{2}\,\overline{5}\,\overline{4}\,\}\\
R=&\{1\,2\,3\,4\,5,\;1\,4\,5\,2\,3,\;\overline{1}\,2\,\overline{3}\,5\,\overline{4},\;\overline{1}\,4\,\overline{5}\,3\,\overline{2},\\
	&\phantom{\lbrace}2\,1\,4\,\overline{3}\,5,\;2\,3\,5\,1\,\overline{4},\;\overline{2}\,1\,\overline{4}\,5\,3,\;\overline{2}\,3\,\overline{5}\,\overline{4}\,\overline{1},\\
	&\phantom{\lbrace}3\,1\,5\,\overline{2}\,4,\;3\,2\,4\,1\,\overline{5},\;\overline{3}\,1\,\overline{5}\,4\,2,\;\overline{3}\,2\,\overline{4}\,\overline{5}\,\overline{1},\\
	&\phantom{\lbrace}4\,1\,3\,2\,\overline{5},\;4\,2\,5\,\overline{1}\,\overline{3},\;\overline{4}\,1\,\overline{3}\,\overline{5}\,\overline{2},\;\overline{4}\,2\,\overline{5}\,\overline{3}\,1,\\
	&\phantom{\lbrace}5\,1\,2\,3\,\overline{4},\;5\,3\,4\,\overline{1}\,\overline{2},\;\overline{5}\,1\,\overline{2}\,\overline{4}\,\overline{3},\;\overline{5}\,3\,\overline{4}\,\overline{2}\,1\,\}
\end{align*}

This transitive set cannot come from the construction of Proposition~\ref{pro:iterative_construction} because there is no orthogonal array of strength $2$ over the alphabet $\{0,1\}$ with $5$ columns and index $1$.

\end{document}